\documentclass[preprint]{imsart}

\usepackage{amsmath,amsthm,amssymb,amsfonts,latexsym,hyperref,dsfont}
\usepackage{stmaryrd}
\usepackage{graphicx}
\usepackage{natbib}
\bibpunct{(}{)}{,}{a}{}{;}
\usepackage[usenames]{color}
\usepackage{geometry}

\usepackage{mathrsfs}

\newcommand\eps{\varepsilon}
\renewcommand\phi{\varphi}

\newcommand\esp[1]{\mathbb{E}\left[#1\right]}

\newcommand\espcc[2]{\mathbb{E}_{#1}\left[#2\right]}

\newcommand\espc[2]{\mathbb{E}\left[\left.#1\right|#2\right]}

\newtheorem{thm}{Theorem}[section]
\newtheorem{prop}[thm]{Proposition}
\newtheorem{cor}[thm]{Corollary}
\newtheorem{assu}{Assumption}
\newtheorem{lem}[thm]{Lemma}
\newtheorem{rem}[thm]{Remark}
\newtheorem{defi}[thm]{Definition}
\newtheorem{ex}[thm]{Example}

\newcommand\n{\mathbb{N}}
\renewcommand\r{\mathbb{R}}

\newcommand\E{\mathbb{E}}

\renewcommand\ll{\left}
\newcommand\rr{\right}

\renewcommand\P{\mathcal{P}}

\newcommand\A{\mathcal{A}}

\newcommand\D{\mathcal{D}}
\newcommand\Sh{\mathcal{S}}
\newcommand\F{\mathcal{F}}
\newcommand\p{\mathbb{P}}

\newcommand\un{\mathds{1}}
\newcommand\uno[1]{\un_{\left\{#1\right\}}}
\newcommand\poly{\mathscr{P}}

\newcommand\K{\mathcal{K}}

\renewcommand\A{\mathcal{A}}
\newcommand\G{\mathcal{G}}
\renewcommand\S{\mathcal{S}}
\newcommand\W{\mathcal{W}}
\newcommand\B{\mathcal{B}}
\newcommand\T{\mathcal{T}}

\renewcommand\k{{\bf k}}
\newcommand\X{{\bf X}}

\allowdisplaybreaks

\begin{document}

\begin{frontmatter}

\title{Generators of measure-valued jump diffusions and convergence rate of diffusive mean-field models.}

\runtitle{Measure-valued jump diffusions and diffusive mean-field limits}

\begin{aug}

\author{\fnms{Xavier} \snm{Erny}\ead[label=e4]{xavier.erny@telecom-sudparis.eu}}

\address{SAMOVAR, T\'el\'ecom SudParis, Institut Polytechnique de Paris, 91120 Palaiseau, France
}

\runauthor{X. Erny}
\end{aug}

\begin{abstract}
The paper has two objectives: proving that the rate of convergence in distribution for mean-field models with interaction strength of order~$N^{-1/2}$, and obtaining explicit expressions for the infinitesimal generators of two types of measure-valued Markov processes (conditional law of McKean-Vlasov processes, and empirical measures of McKean-Vlasov systems). The proof of the convergence of mean-field system requires the second result about the generators, and both results need to study a notion of differentiability of measure-variable functions known as linear differentiability.
Due to the particular framework that is studied, many technical difficulties arise compared to the existing literature. Two of the main problems are the following ones: the scaling~$N^{-1/2}$ implies that the limit measure-valued processes are not deterministic, and the empirical measure processes related to McKean-Vlasov equations with jumps are necessarily discontinuous. Both properties make the expressions of the generators more complicated than what is usually considered.
\end{abstract}



\end{frontmatter}

\tableofcontents

\section*{Introduction}

The notion of particle systems is widely studied in many frameworks due to the diversity of their applications: biology (e.g. neural networks \cite{fournier_toy_2016}, genomics \cite{reynaud-bouret_adaptive_2010}, ecological networks \cite{billiard_continuous_2022}), finance (e.g. order books \cite{lu_high-dimensional_2018}, high frequency data \cite{bauwens_modelling_2009}), sociology (e.g. social networks \cite{mitchell_hawkes_2009}), physics (e.g. kinetic theory \cite{vedenyapin_boltzmann_2011}). A particular case of interest concerns mean-field particle systems, meaning systems in which the particles interact with all the others in a similar way. The reason behind this interest is that, the interactions and the dynamics of such systems can be somehow characterized by their empirical measures. The advantage of studying the empirical measure rather than the particle system itself is that, if the dynamics of the particles are characterized by $\r$-valued processes, then the values taken by the empirical measure belong to the set of probability measures on~$\r$, which does not depend on the number of particles of the system. In this paper, we study large scale limit behavior, as the number of particles goes to infinity, so it is preferable to work on spaces that are independent of the size of the systems. In particular, the study of mean-field particle systems is closely related to the notion of measure-valued processes.

The study of measure-valued Markov processes goes back at least to \cite{fleming_measure-valued_1979} to model population genetics. In the seminal course \cite{dawson_measure-valued_1993}, Donald Dawson has introduced a notion of differentiation for measure-variable functions, in order to characterize infinitesimal generators and martingale problems in the context of measure-valued Markov processes. A variant of the notion of the derivative introduced by Donald Dawson, has been defined by Ren\'e Carmona and Fran\c cois Delarue (see e.g. section~5.4.1 of \cite{carmona_probabilistic_2018}) and is referred to as the {\it linear derivative}, and another rather different notion (at first sight) by Pierre-Louis Lions in a course given at Coll\`ege de France (see, for example, the lecture notes of Pierre Cardaliaguet \cite{cardaliaguet_notes_2013}). Since then, this topic still draws interest: for instance, recently, \cite{guo_itos_2023} and \cite{cox_controlled_2021} have established Ito's formulas for measure-valued semimartingales in different frameworks, and the authors of \cite{crisan_smoothing_2018} have proved some regularity properties for stochastic flows of McKean-Vlasov processes. For the more specific question of convergence in distribution of emprical measures of particle systems, we can cite Theorem~2.9 of \cite{chassagneux_weak_2022} and Theorem~6.1 of \cite{mischler_new_2015} that both obtain a convergence speed of $1/N$ in a model with interaction of order~$1/N$, and Theorem~4.13 of \cite{jourdain_central_2021} that proves the convergence of the fluctuation in a similar model. In this paper, we mostly use a notion of {\it linear derivative} and state precisely the definition at Section~\ref{def:diffmes}.

There are two main contributions compared to the existing literature, and one minor. They are all closely related.

The first main contribution consists in proving the convergence speed for large-scale limits of mean-field McKean-Vlasov particle systems with jumps in CLT regime. The previously cited papers (i.e. \cite{chassagneux_weak_2022}, \cite{mischler_new_2015} and \cite{jourdain_central_2021}) only consider these systems in LLN regime, i.e. when the interaction strength of the $N$-particle systems is $N^{-1}$, and without jumps. These two particularities entail different difficulties:
\begin{itemize}
\item In the regime of this paper (i.e. when the interaction strength between particles is of order~$N^{-1/2}$), the limit of the empirical measures is, in general, a random distribution (more precisely, it is the conditional distribution of a limit particle given the common noise of the system). In particular, the limit measure-valued processes are not deterministic, and the related generator are more complex. This problem also exists in Theorem~4.13 of \cite{jourdain_central_2021} where fluctuation of the $1/N$-regime are proved to converge. However, Theorem~4.13 of \cite{jourdain_central_2021} does not provide a rate of convergence for the fluctuation, and it only deals with a McKean-Vlasov equation whose only dependency w.r.t. the number~$N$ of particles is through the empirical measure of the system. The technics developped here may be used to obtain an explicit rate of convergence of fluctuation for more general models, dealing with SDEs with interaction terms like in this paper.
\item Secondly, the presence of the jump term in our models make our empirical measure processes not continuous. So our proofs do not only consider stochastic calculus for measure-valued processes, but also for non-continuous processes. Once again, this property makes the expression of the infinitesimal generators more complex.
\end{itemize}

In addition, few results exist for the convergence in Central Limit Theorem regime (e.g. \cite{erny_conditional_2021}, and, in the more general framework of martingale measures, \cite{erny_white-noise_2022}), and, as far as we know, the only rate of convergence for this kind of model has been obtained in \cite{erny_strong_2023} for a $L^1$-convergence using coupling techniques based on the approximation theorem of \cite{komlos_approximation_1976}. This convergence speed is very slow (i.e. $(\ln N)^{1/5}/N^{1/10}$) due to the use of an Euler scheme whose step size depends on~$N$ combined with the KMT approximation theorem applied at each step. In this paper, we prove that the convergence speed in distribution is the expected one, in~$N^{-1/2}$  (see Theorem~\ref{mainresultdiffu}) which is not a strictly stronger result than the one of \cite{erny_strong_2023}, since, here, the convergence speed is faster but the convergence sense is weaker. The result of this paper is even new in simple frameworks, like the one of Corollary~\ref{cordiffu}. Let us also mention that our technics can be used to prove that the rate of convergence for the $1/N$-regime with common noise is still~$N^{-1}$ (i.e. the same as the one obtained in the aforementioned papers without common noise). 

The second main contribution is the study of some Markovian properties of measure-valued processes, and in particular, we obtain the explicit expressions of generators for two kinds of processes: the conditional laws of solutions of (conditional) McKean-Vlasov equations, and the empirical measures of McKean-Vlasov particle systems. These generators have already been used more or less explicitly in the aforementioned papers: \cite{chassagneux_weak_2022}, \cite{mischler_new_2015} and \cite{jourdain_central_2021}. More precisely, the generators (in the martingale problem sense) of the limit systems appear explicitly, and the generators of the empirical measures is somehow characterized in the proofs of these papers, but their expressions are not explicitly given. In this paper, we give (explicit) expressions which are more general than in these papers for the reasons given before: our limit equations being conditional McKean-Vlasov equations (due to the CLT regime), the limit measure-valued processes are not deterministic (whereas the limit processes in \cite{chassagneux_weak_2022} and \cite{mischler_new_2015} are deterministic), and because we consider McKean-Vlasov equations with jumps, our empirical measure processes are not continuous. Let us also mention that, since \cite{guo_itos_2023} and \cite{cox_controlled_2021} have established Ito's formula for measure-valued processes, it is possible to use their result to obtain an expression for the corresponding infinitesimal generator with some differences compared to our paper:
\begin{itemize}
\item both papers do not consider empirical measure processes,
\item the notion of second order derivative used in \cite{cox_controlled_2021} is not equivalent as the one used in this paper and is somehow ad hoc w.r.t. the notion of Taylor expansion (here, it is defined as the iteration of the first order derivative, and Taylor's formulas are developped from this definition)
\item in \cite{guo_itos_2023}, the authors only consider deterministic measure-valued processes, but these processes are defined as the (non-conditional) law processes of general semimartingales that are not assumed to be solutions of SDEs.
\end{itemize}

Besides, in most of these papers, the measure-valued processes are studied as semimartingales, but not as Markov processes. Consequently, we need to prove some additional technical results in our paper.

So, to the best of our knowledge, obtaining Ito's formula or generator for a measure-valued processes with such general jumps is new. The expressions of these generators are given at Theorems~\ref{barXmarkov} for the conditional laws of McKean-Vlasov processes, and~\ref{XNmarkov} for the empirical measures of McKean-Vlasov particle systems.

Let us remark that, in general, infinitesimal generators are differential operators acting on functions whose domain is the space of values of the Markov processes. In particular the study of measure-valued Markov processes is related to the analysis of measure-variable functions. Consequently, this paper also deals with this last topic, which is the minor contribution of this paper. The properties that we prove and use are not completely new, but due to some particularities of our framework, we need to adapt the proofs and to state some new technical results. For example, the notion of mixed derivatives and Taylor's formula for measure-variable functions have been established by \cite{chassagneux_weak_2022}, the notion of measure-variable polynomials is studied by \cite{cuchiero_probability_2019}, and the regularity of the flows of the solutions of McKean-Vlasov equations has been proved in \cite{chassagneux_probabilistic_2022} (however since we work on conditional McKean-Vlasov equations it is not clear that the proof of this last result can be used directly, so we state and prove the result in Lemma~\ref{reguflot}).

{\bf Organization of the paper.} The next section introduces the notation that are used throughout the paper. Section~\ref{mainresult} is dedicated to the statement of the main results: Section~\ref{resultconvergence} for the result about the convergence speed in distribution for diffusive mean-field particle systems, and Section~\ref{resultmarkov} for the expressions of the generators of the measure-valued Markov processes. The goal of Section~\ref{analysismesvar} is to study some analytical properties of measure-variable functions. In particular, Taylor-Lagrange's inequality, a notion of measure-variable polynomials, and also establishing some technical lemmas. Section~\ref{markov} aims to prove the results of Section~\ref{resultmarkov} and to study some Markovian properties of measure-valued processes: establishing Trotter-Kato's formula and studying the regularity of semigroups (which is required to use Trotter-Kato's formula jointly with Taylor-Lagrange's inequality in the proof of the results of Section~\ref{resultconvergence}). The proofs of the results stated at Section~\ref{resultconvergence} are given at Section~\ref{proofmain}. Finally, the Appendix gathers some technical results and technical proofs in order to ease the reading of the paper.

\section*{Notation and convention}

In the notation below, $(E,d)$ is always a Polish space (or just a set if the notation does not need a particular structure on~$E$).
\begin{itemize}
\item For $x=(x_1,...,x_n)\in\r^n$, we denote $||x||_1$ the classical $L^1$-norm:
$$||x||_1 = \sum_{k=1}^n |x_k|.$$
\item $\P(E)$ is the set of probability measures on~$E$, endowed with the topology of Prohorov metric (i.e. the topology of the weak convergence).
\item If $X$ is an $E$-valued random variable, $\mathcal{L}(X)\in\P(E)$ denote its law.
\item For $p\geq 1$, $\P_p(E)$ is the set of probability measures on~$E$ with a finite $p$-th order moment, and $\P_\infty(E)$ is the intersection of every $\P_p(E)$ over all~$p\geq 1$. Each $\P_p$ is endowed with the $p$-th order Wasserstein metric.
\item For~$p\geq 1$ and~$m,\mu\in\P_p(E)$, the $p$-th order Wasserstein metric between~$m$ and~$\mu$ is defined as the best $L^p$-coupling between~$m$ and~$\mu$:
$$W_p(m,\mu)=\underset{\mathcal{L}(X)=m,\mathcal{L}(Y)=\mu}{\inf} \esp{d(X,Y)^p}^{1/p}.$$
\item For~$m,\mu\in\P_1(E)$, $D_{KR}(m,\mu)$ denotes the Kantorovich-Rubinstein metric between~$m$ and~$\mu$:
$$D_{KR}(m,\mu)=\underset{f\in \textrm{Lip}_1}{\sup} \int_E f(x)\,d\ll(m-\mu\rr)(x),$$
where $\textrm{Lip}_1$ denotes the set of Lipschitz continuous functions~$f:E\rightarrow\r$ such that, for all~$x,y\in E$, $|f(x) - f(y)|\leq d(x,y)$.
By the Kantorovich-Rubinstein duality (e.g. Remark~6.5 of \cite{villani_optimal_2009}), $D_{KR}=W_1$. In the paper, we mostly use the characterization of~$D_{KR}$ and its notation rather than~$W_1$.
\item Note that each $\P_p(\r)$ is Polish by Theorem~6.18 of \cite{villani_optimal_2009}. For $M_1,M_2\in\P_1(\P_1(\r))$, we denote specifically by $\D_{KR}(M_1,M_2)$ the Kantorovich-Rubinstein metric between~$M_1$ and~$M_2$.
\item A function~$f:\P_1(\r)\times\r^n\rightarrow\r$ is said to be Lipschitz continuous if it is Lipschitz continuous w.r.t. both variables together: there exists~$C>0$ such that, for all~$m_1,m_2\in\P_1(\r)$ and~$x_1,x_2\in\r^n$,
$$\ll|f(m_1,x_1) - f(m_2,x_2)\rr|\leq C\ll(||x_1-x_2||_1 + D_{KR}(m_1,m_2)\rr).$$
\item A function~$f:\P_1(\r)\times\r^n\rightarrow\r$ is said to be sublinear if: there exists~$C>0$ such that, for all~$m\in\P_1(\r)$ and~$x\in\r^n$,
$$\ll|f(m,x)\rr|\leq C\ll(1+||x||_1 + \int_\r |y|dm(y)\rr).$$
\item $C(E)$ is the set of continuous functions~$f:E\rightarrow\r$.
\item For~$k\in\n\cup\{\infty\}$ and~$n\in\n^*$, $C^k(\r^n)$ denotes the space of functions~$f:\r^n\rightarrow\r$ of class~$C^k$.
\item For~$f\in C^k(\r^n)$ and $\alpha$ a multi-index belonging to $\n^n$, the size of $\alpha$ is
$$|\alpha| = \sum_{i=1}^n \alpha_i,$$
and, if $|\alpha|\leq k$, $\partial_\alpha f$ denotes the derivative of~$f$ w.r.t.~$\alpha$.
\item When no doubt is possible (and to make some expressions more readable), we use usual notation for partial derivatives, for example, if $g\in C^3(\r^2)$, for all~$x,y\in\r$,
$$\partial^3_{xxy}g(y,x) = \partial_{(1,2)}g(y,x)\textrm{ and }\ll(\partial^2_{y_1y_2} g(y_1,y_2)\rr)_{\ll|y_1=y_2=x\rr.} = \partial_{(1,1)}g(x,x),$$
and the same notation replacing the derived variables by their indices:
$$\partial^3_{xxy} g(y,x) = \partial^3_{2,2,1}g(y,x) = \partial_{(1,2)}g(y,x).$$
\item For a function~$f:\P_1(\r)\times\r\rightarrow\r$ and~$m\in\P_1(\r),x\in\r$, we use the notation
$$\partial_{(0,1)}f(m,x) = \partial_{x} f(m,x)\textrm{ and }\partial_{(0,2)}f(m,x) = \partial^2_{xx} f(m,x),$$
but, in this context where $f$ depends on variables of different nature (i.e. belonging to $\P_1(\r)$ and $\r$) the index of the measure-variable is always zero.
\item $C^k_b(\r^n)$ is the subspace of $C^k(\r^n)$ containing the functions~$f$ such that: for all multi-index~$\alpha$ of size non-greater than~$k$,
$$\underset{x\in\r^n}{\sup}\ll|\partial_\alpha f(x)\rr|<\infty.$$
\item $C^k_c(\r^n)$ is the subspace of $C^k_b(\r^n)$ containing the compactly supported functions.
\item For~$f\in C^k_b(\r^n)$, we denote
$$||f||_{k,\infty} = \sum_{|\alpha|\leq k}\underset{x\in\r^n}{\sup}~|\partial_\alpha f (x)|.$$
\item A tuple~$x\in E^n$ is always indexed from~1 to~$n$: $x=(x_1,...,x_n)$,
\item For~$x\in E^n$ and~$1\leq k\leq n$, $x\backslash_k$ denotes the tuple $x$ without the index~$k$:
$$x\backslash_k = (x_1,...,x_{k-1},x_{k+1},...,x_n)\in E^{n-1},$$
and, for~$y\in E$, $(x\backslash_k y)$ denotes the tuple $x$ where $x_k$ is replaced by~$y$:
$$x\backslash_k y = (x_1,...,x_{k-1},y,x_{k+1},...,x_n)\in E^{n}.$$
In addition, for~$k\neq l$, and $y_1,y_2\in E$, we use the notation $x\backslash_{(k,l)}$ for $x$ without the indices~$k$ and~$l$, and $(x\backslash_{(k,l)}(y_1,y_2))$ for the tuple $x$ where $x_k$ is replaced by~$y_1$ and~$x_l$ by~$y_2$.

For example, this notation is used in the following context: if $f\in C^0_b(\r^3)$, $m\in\P(\r)$ and~$y\in\r,$
\begin{multline*}
\sum_{k=1}^3 \int_{\r^2}f(x\backslash_k y)dm^{\otimes 2}(x\backslash_k) = \int_{\r^2} f(y,x_2,x_3)dm^{\otimes 2}(x_2,x_3)\\
+\int_{\r^2} f(x_1,y,x_3)dm^{\otimes 2}(x_1,x_3)+\int_{\r^2} f(x_1,x_2,y)dm^{\otimes 2}(x_1,x_2).
\end{multline*}
\item For $n\in\n^*$ and $F:\P_1(\r)\times\r^n\rightarrow \r$ and $x\in\r^n$, we denote by $F_x$ the following function
$$F_x : m\in\P_1(\r)\longmapsto F_x(m) = F(m,x).$$
\item For $m\in\P(\r)$ and~$\lambda\in\r$, we denote by $\Sh(m,\lambda)$ the probability measure~$m$ shifted by~$\lambda$:
$$\Sh(m,\lambda) : A\in\B(\r)\longmapsto m\ll(\ll\{x-\lambda : x\in A\rr\}\rr).$$
\item We use $C$ to denote any arbitrary positive constant. The value of~$C$ can change from line to line inside an equation. In addition, if $C$ depends on some variables~$\theta$, we denote it by $C_\theta$.
\end{itemize}

\section{Main results}\label{mainresult}

The next two subsections state our main results: an explicit rate of convergence for diffusive mean-field models, and explicit expressions for the generators of two classes of measure-valued Markov processes.

\subsection{Convergence speed of diffusive mean-field model}\label{resultconvergence}

The results of this section concern diffusive mean-field particle systems. Meaning, particle systems where the interacting strength is of order~$N^{-1/2}$ with $N$ the number of particles. For this kind of model to converge as the number of particles tends to infinity, the interactions between particles need to be centered. So, the dynamics depends on a probability measure called~$\nu$, that satisfies
$$\int_\r u\,d\nu(u)= 0,$$
and for computation reason, our proofs require the following condition
$$\int_{\r} |u|^3d\nu(u)<\infty.$$

Before stating the main result under a general form (cf Theorem~\ref{mainresultdiffu}), let us begin by introducing a particular case (cf Corollary~\ref{cordiffu}) that does not require the technical assumptions about the regularity of measure-variable functions. Let us notice that, even the result of Corollary~\ref{cordiffu} is new.

Let the particle system~$(Y^{N,k})$ be defined as: for~$N\in\n^*$ and~$1\leq k\leq N$,
\begin{equation}\label{diffuYN}
dY^{N,k}_t = \tilde b\ll(Y^{N,k}_t\rr)dt + \tilde\sigma\ll(Y^{N,k}_t\rr)dB^k_t + \frac1{\sqrt{N}}\sum_{l=1}^N \int_{\r_+\times\r} u\cdot \uno{z\leq \tilde f\ll(Y^{N,l}_{t-}\rr)}d\pi^l(t,z,u),
\end{equation}
where $B^k$ ($k\geq 1$) are standard Brownian motions, $\pi^l$ ($l\geq 1$) are Poisson measures with intensity~$dt\cdot dz\cdot d\nu(u)$ and, the variables $Y^{N,k}_0$ ($1\leq k\leq N$) are i.i.d. with finite $p$-th order moment for all~$p\in\n^*$.  The Brownian motions, Poisson measures and initial conditions are assumed to be independent.

And let $\bar Y$ be solution of
\begin{equation}\label{diffubarY}
d\bar Y_t = \tilde b\ll(\bar Y_t\rr)dt + \tilde\sigma\ll(\bar Y_t\rr)dB_t  + \ll(\int_\r u^2 d\nu(u)\rr)^{1/2}\sqrt{\esp{\ll.\tilde f(\bar Y_t)\rr|\W_t}}dW_t,
\end{equation}
with $B,W$ independent standard Brownian motions, $\W$ the filtration of~$W$, and $\bar Y_0$ having the same law as the variables $Y^{N,k}_0$. The Brownian motions~$B,W$ are assumed to be independent of~$\bar Y_0$.

\begin{cor}\label{cordiffu}
Assume that the functions~$\tilde b,\tilde\sigma$ and~$\tilde f$ are $C^5$ and that their derivatives of orders from one to five are bounded (the functions $\tilde b$ and $\tilde\sigma$ need not be bounded). In addition, we assume that $\tilde f$ is bounded and lower-bounded by some positive constants. We also assume that $\nu$ is centered with a finite third order moment. Then, for any~$T>0$, there exists~$C_{T}>0$ such that, for all~$N\in\n^*,$ and $\phi\in C^4_b(\r),$
$$\underset{t\leq T}{\sup}~\ll|\esp{\phi\ll(Y^{N,1}_t\rr)} - \esp{\phi\ll(\bar Y_t\rr)}\rr| \leq C_{T} \ll|\ll|\phi\rr|\rr|_{3,\infty} N^{-1/2}.$$
\end{cor}

Now, let us state our general diffusive model. For any~$N\in\n^*$, we introduce $X^{N,k}$ ($1\leq k\leq N$) as solutions to
\begin{align}
dX^{N,k}_t =& b\ll(\mu^N_t,X^{N,k}_t\rr)dt + \sigma\ll(\mu^N_t,X^{N,k}_t\rr)dB^k_t  \label{diffuXN}\\
&+ \frac{1}{\sqrt{N}}\sum_{l=1}^N \int_{\r_+\times\r} u\cdot h\ll(\mu^N_{t-},X^{N,l}_{t-}\rr) \uno{z\leq f\ll(\mu^N_{t-},X^{N,l}_{t-}\rr)}d\pi^l(t,z,u),\nonumber
\end{align}
where $B^k$ ($k\geq 1$) are standard Brownian motions of dimension one, $\pi^l$ ($l\geq 1$) are Poisson measures on~$\r_+^2\times\r$ with intensity~$dt\cdot dz\cdot d\nu(u)$, and $X^{N,k}_0$ ($1\leq k\leq N$) are i.i.d. such that $B^k$ ($k\geq 1$), $\pi^l$ ($l\geq 1$) and $X^{N,k}_0$ ($1\leq k\leq N$) are mutually independent, and
$$\mu^N_t = \frac1N\sum_{k=1}^N \delta_{X^{N,k}_t}.$$

The main result of this section states the rate of convergence in distribution of the $N$-particle system~\eqref{diffuXN} to the limit process~$\bar X$ solution to the following SDE:
\begin{align}
d\bar X_t=& b\ll(\bar\mu_t,\bar X_t\rr)dt + \sigma\ll(\bar\mu_t,\bar X_t\rr)dB_t \label{diffubarX}\\
&+ \ll(\int_\r u^2d\nu(u)\rr)^{1/2}\ll(\int_\r h(\bar\mu_t,x)^2f(\bar\mu_t,x)d\bar\mu_t(x)\rr)^{1/2}dW_t,\nonumber
\end{align}
where $B,W$ are standard Brownian motions of dimension one,  $\bar X_0$ is real-valued random variable distributed as~$\bar X^{N,k}_0$ such that $B,W,\bar X_0$ are independent, and
$$\bar\mu_t = \mathcal{L}\ll(\bar X_t| \W_t\rr),$$
with $\W$ the filtration of~$W$.

\begin{rem}
The convergence in distribution of~$\mu^N$ to~$\bar\mu$ has already been proved in \cite{erny_white-noise_2022} in a more general framework (cf Theorem~2.6 for the general case, and Example~1 for a model closer to the one of this section). But the technics used in the proofs did not allow to obtain an explicit convergence speed.
\end{rem}

For the previous equations to be well-posed, we assume the following.
\begin{assu}\label{assulinearlip}
For any~$m_1,m_2\in\P_1(\r)$, and~$x_1,x_2\in\r,$
\begin{multline*}
\ll|b(m_1,x_1) - b(m_2,x_2)\rr| + \ll|\sigma(m_1,x_1) - \sigma(m_2,x_2)\rr| + \ll|\varsigma(m_1,x_1) - \varsigma(m_2,x_2)\rr|\\
\leq C\ll(|x_1-x_2| + D_{KR}(m_1,m_2)\rr),
\end{multline*}
and
$$\int_0^{+\infty}\ll| h(m_1,x_1)\uno{z\leq f(m_1,x_1)} - h(m_2,x_2)\uno{z\leq f(m_2,x_2)}\rr|dz \leq C\ll(|x_1-x_2| + D_{KR}(m_1,m_2)\rr).$$

In addition, we also suppose that, for any~$p\in\n^*$, there exists some~$C_p>0$ such that, for all~$m\in \P_1(\r),x\in\r,$
$$\ll|\ll(x + h(m,x)\rr)^p - x^p\rr|f(m,x)\leq C_p\ll(1+|x|^p + \int_\r |y|^pdm(y)\rr).$$
\end{assu}

\begin{rem}
The first condition of Assumption~\ref{assulinearlip} is standard: $b,\sigma,\varsigma$ are Lispschitz continuous. The second condition (i.e. the kind of Lipschitz condition on jointly~$h$ and~$f$) is satisfied, for example, when both functions~$f$ and~$h$ are bounded and Lipschitz continuous. The last condition is guaranteed if, for example, one of the two following properties hold:
\begin{itemize}
\item the function~$h$ is sublinear and the function~$f$ is bounded,
\item or  the function~$h$ is bounded and the function~$f$ is sublinear.
\end{itemize}
\end{rem}

We also make the following hypothesis concerning the initial conditions.
\begin{assu}\label{assulinearmoment}
For every~$N\in\n^*$, the variables~$X^{N,k}_0$ ($1\leq k\leq N$) are i.i.d.

All the initial conditions $\bar X_0$ and $X^{N,k}_0$ ($N\in\n^*$ and $1\leq k\leq N$) have a finite fourth order moment, uniformly bounded w.r.t.~$N$:
$$\esp{\ll|\bar X_0\rr|^4}<\infty~~\textrm{ and }~~\underset{N\in\n^*}{\sup}~\esp{\ll|X^{N,1}_0\rr|^4}<\infty.$$
\end{assu}


Under Assumptions~\ref{assulinearlip} and~\ref{assulinearmoment}, it is known that the SDEs~\eqref{diffuXN} and~\eqref{diffubarX} are well-posed (see e.g. Theorem~2.1 of \cite{graham_mckean-vlasov_1992}) in the following sense: there exist unique strong solutions~$(X^{N,k})_k$ and~$\bar X$ of the respective SDEs that satisfy, for all~$T>0$, $1\leq k\leq N,$
\begin{equation}\label{controlsde}
\esp{\underset{t\leq T}{\sup}~\ll|\bar X_t\rr|^4}<\infty~~\textrm{ and }~~\underset{N\in\n^*}{\sup}\esp{\underset{t\leq T}{\sup}~\ll|X^{N,k}_t\rr|^4}<\infty.
\end{equation}

Note that, a priori, uniqueness only holds for processes that satisfy the above controls, which are not a priori estimates. The controls~\eqref{controlsde} can be guaranteed on the solutions that are built via Banach-Picard schemes for example (see the {\it Step~1} of the proof of Lemma~\ref{reguflot} at Appendix~\ref{proofreguflot} for the SDE~\eqref{diffubarX}). This technical point is due to the problem of dealing with stopping times in McKean-Vlasov framework: roughly speaking, if $\mu_t := \mathcal{L}(X_t)$ (for $t\in\r_+$) and $\tau$ is some stopping time, then $\mu_\tau \neq \mathcal{L}(X_\tau)$.

Our first main result requires an additional technical assumption. This assumption relies on the notion of differentiability of measure-variable functions as it is defined in Section~\ref{def:diffmes}.
\begin{assu}\label{assuderivativediffu}
 the functions $b,\sigma$ and~$m\mapsto (\int_\r h(m,x)^2f(m,x)dm(x))^{1/2}$ admit fifth order mixed derivatives (in the sense of Definition~\ref{mixderiv}) and all their mixed derivatives of orders from one to five are bounded. In addition, if $g:\P_1(\r)\times\r\rightarrow\r$ is any of the previous functions, then, there exists~$C>0$ such that for all~$x\in\r,$ the function~$g_x$ belongs to $C^2_b(\P_1(\r))$ (the set is introduced in Definition~\ref{CnkP}), and
$$\underset{\substack{m\in\P_1(\r)\\y_1,y_2\in\r}}{\sup}\ll|\partial^2_{y_1y_2}\delta^2 (g_x)(m,y_1,y_2)\rr|\leq C(1+|x|).$$
\end{assu}

\begin{thm}\label{mainresultdiffu}
Grant Assumptions~\ref{assulinearlip},~\ref{assulinearmoment} and~\ref{assuderivativediffu}.  Assume that the probability measure~$\nu$ is centered with a finite third moment. Then,  for all~$T>0$ and~$n\in\n^*$, there exists~$C_{T,n}>0$ such that, for all~$0<t_1<...<t_n\leq T$, for any measure-variable polynomials~$G_1,...,G_n$ of order four (in the sense of Definition~\ref{def:poly}), and for each~$N\in\n^*,$
$$\ll|\esp{G_1\ll(\mu^N_{t_1}\rr)...G_n\ll(\mu^N_{t_n}\rr)} - \esp{G_1\ll(\bar\mu_{t_1}\rr)...G_n\ll(\bar\mu_{t_n}\rr)}\rr|\leq C_{T,n}||G_1||_3...||G_n||_3~\frac{1}{\sqrt{N}},$$
where $||G||_3$ is a bound of the mixed derivatives of~$G$ up to order three (defined at Definition~\ref{bornederivn}).
\end{thm}

The formal results of this section treat the case of the diffusive mean-field limits (i.e. when the interaction strength is of order~$N^{-1/2}$) to obtain rates of convergence in distribution of order~$N^{-1/2}$. Let us notice that the proofs can also be applied to show that the rates of convergence in distribution of linear mean-field limits (i.e. when the interaction strength between the particles is of order~$N^{-1}$) are of order~$N^{-1}$. As mentioned in the introduction, this result has already been proved in Theorem~2.9 of \cite{chassagneux_weak_2022} and Theorem~6.1 of \cite{mischler_new_2015} in the context of the true propagation of chaos, where the empirical measures of the particle systems converge to a deterministic distribution. The technics used in our proofs allow to consider mean-field particle systems with common noise (similarly as the ones considered in \cite{carmona_mean_2016} for example) and to prove the convergence of the empirical measures in distribution at speed~$N^{-1}$ to a conditional distribution. In addition, our results permit to deal with jump interaction terms in the particle systems, what makes the empirical measure processes not continuous.

\subsection{Infinitesimal generators for measure-valued Markov processes}\label{resultmarkov}

The aim of this section is to introduce the expression of the infinitesimal generators of two kinds of measure-valued Markov processes: the conditional laws of diffusions, as $(\bar\mu_t)_{t\geq 0}$ in~\eqref{diffubarX} (see Theorem~\ref{barXmarkov}), and the empirical measures of particle systems, as $(\mu^N_t)_{t\geq 0}$ in~\eqref{diffuXN} (see Theorem~\ref{XNmarkov}). The notion of generator that we use is defined in a martingale problem sense (see Definition~\ref{defi:markov} for precise definition). We prove that the domains of the infinitesimal generators contain the set of measure-variable polynomials of order two (we refer to Section~\ref{sec:poly} for the definition and the study of these functions), which is denoted~$\poly^2$.


In this section, we work in a bit more general framework, where we consider $\bar X$ defined as the solution of
\begin{align}
d\bar X_t =&  b\ll(\bar\mu_t,\bar X_t\rr)dt  + \sigma\ll(\bar\mu_t,\bar X_t\rr)dB_t+ \varsigma\ll(\bar\mu_t,\bar X_t\rr)dW_t\label{defbarX}\\
& +h\ll(\bar\mu_{t-},\bar X_{t-}\rr)\int_{\r_+}\uno{z\leq f\ll(\bar \mu_{t-},\bar X_{t-}\rr)}d\pi(t,z),\nonumber
\end{align}
where $B,W$ are standard Brownian motions of dimension one, $\pi$ a Poisson measure on~$\r_+^2$ with Lebesgue intensity such that $B,W$ and~$\pi$ are mutually independent, and
$$\bar\mu_t = \mathcal{L}\ll(\bar X_t~|~\W_t\rr),$$
with $\W_t=\sigma({W}_s~:~s\leq t)$ the filtration of the Brownian motion~$W$ at time~$t$.

\begin{thm}\label{barXmarkov}
Grant Assumption~\ref{assulinearlip}. The process $(\bar \mu_t)_t$ is a $\P_2(\r)$-valued Markov process. The domain~$\mathcal{D}_\G(\bar \mu)$ contains the set~$\poly^2$, and the generator~$\bar A$ of~$\bar \mu$ satisfies: for any~$G\in\poly^2$ and~$m\in\P_2(\r)$,
\begin{align}
\bar A G(m) =&  \int_\r \bar L(\delta G(m,\bullet))(m,x)dm(x) \label{limitgenerator}\\
&+ \frac12\int_{\r^2} \ll(\partial^2_{xy}\delta^2G(m,x,y)\rr)\varsigma(m,x)\varsigma(m,y)dm^{\otimes 2}(x,y),\nonumber
\end{align}
where $\bar L$ is the operator defined as: for $g\in C^2(\r),x\in\r,m\in\P_1(\r),$
\begin{align*}
\bar L g(m,x)=& b(m,x)g'(x) + \frac12\ll(\sigma(m,x)^2 + \varsigma(m,x)^2\rr)g''(x)\\
&+ f(m,x)\ll[g\ll(x + h(m,x)\rr) - g(x)\rr].
\end{align*}
\end{thm}

The proof of Theorem~\ref{barXmarkov} is given at Section~\ref{conditionallaw}.

It can be noticed that the operator~$\bar L$ is the generator of the process~$\bar X$ in the ``martingale problem'' sense. The process~$\bar X$ itself being a non-Markov semimartingale we prefer not to call~$\bar L$ a generator to avoid confusion.

\begin{rem}\label{rembarXgen}
The expression of the generator~$\bar A$ is a reminiscent of classical expressions for generators of $\r$-valued Markov processes. Indeed, the drift terms of the generator corresponds to the first order derivative of the test-function~$G$ (where it can be noted that, the Brownian and jump terms of the $\r$-valued process~$\bar X$ creates some drift terms for the measure-valued process~$\bar\mu$ since they are ``averaged'' by $\bar\mu$). Besides, the Brownian motion~$W$ being still a noise source for~$\bar\mu$, it creates a Brownian term even in the dynamics of~$\bar\mu$ which corresponds to the second order derivative of~$G$. The expression at~\eqref{limitgenerator} is compact because the first term is written as an operator acting on the first order derivative of the test-function. This term would be harder to write with the notion of Lions-derivative instead of the linear derivative that we use. However, one can note that the second term would be a bit more compact with Lions-derivative instead of the linear derivative (by Corollary~\ref{corechange}).
\end{rem}

As for the previous result, we consider a framework that is a bit more general compared to the previous section. For~$N\in\n^*$, let $X^{N,k}$ ($1\leq k\leq N$) be solutions to the following SDEs:
\begin{align}
dX^{N,k}_t =& b\ll(\mu^N_t,X^{N,k}_t\rr)dt + \sigma\ll(\mu^N_t,X^{N,k}_t\rr)dB^k_t + \varsigma\ll(\mu^N_t,X^{N,k}_t\rr)dW_t\label{XNk}\\
&+ \sum_{l=1}^N \int_{\r_+\times\r} h\ll(\mu^N_{t-},X^{N,l}_{t-},u\rr)\cdot \uno{z\leq f\ll(\mu^N_{t-},X^{N,l}_{t-}\rr)}d\pi^l(t,z,u),\nonumber
\end{align}
where $B^1,...,B^N,W$ are standard Brownian motions, $\pi^1,...,\pi^N$ are Poisson measures on $\r_+\times\r_+\times\r$ with intensity $dt\cdot dz\cdot d\nu(u)$ (with $\nu$ a probability measure on~$\r$), such that they are all mutually independent, and
$$\mu^N_t = \frac1N\sum_{l=1}^N \delta_{X^{N,l}_t}.$$

\begin{thm}\label{XNmarkov}
Assume that the functions~$b,\sigma,\varsigma$ are Lipschitz continuous, and that there exists~$C>0$ such that for all~$x_1,x_2\in\r,$ and~$m_1,m_2\in\P_1(\r)$,
\begin{multline*}
\int_\r\int_\r \ll|h(m_1,x_1,u)\uno{z\leq f(m_1,x_1)} - h(m_2,x_2,u)\uno{z\leq f(m_2,x_2)}\rr|dz\,d\nu(u)\\
\leq C\ll(|x_1-x_2| + D_{KR}(m_1,m_2)\rr).
\end{multline*}

Then, the process~$(\mu^N_t)_t$ is a $\P_2(\r)$-valued Markov Process. The domain~$\mathcal{D}_\G(\mu^N)$ contains the set~$\poly^2$, and the generator~$A^N$ of~$\mu^N$ satisfies: for any~$G\in\poly^2$, and~$m\in\P_2(\r)$,
\begin{align*}
A^NG(m)=& \int_\r \ll(b(m,x)\partial_x\delta G(m,x) + \frac12\ll[\sigma(m,x)^2 + \varsigma(m,x)^2\rr]\partial^2_{xx} \delta G(m,x)\rr)dm(x) \\
&+N\int_\r\int_\r f(m,x)\ll[G(\Sh(m,h(m,x,u))) - G(m)\rr]d\nu(u)dm(x)\\
&+\frac12\int_{\r^2}\partial^2_{xy}\delta^2 G(m,x,y)\varsigma(m,x)\varsigma(m,y)dm^{\otimes 2}(x,y)\\
&+\frac{1}{2N} \int_\r \ll(\partial^2_{y_1y_2}\delta^2 G(m,y_1,y_2)\rr)_{|y_1=y_2=x}\sigma(m,x)^2dm(x),
\end{align*}
where we recall that $\Sh$ is a shift operator defined as: for $m\in\P(\r),\lambda\in\r$,
$$\Sh(m,\lambda): A\in\mathcal{B}(\r)\longmapsto m\ll(\{x-\lambda~:~x\in A\}\rr).$$
\end{thm}

The proof of Theorem~\ref{XNmarkov} is given at Section~\ref{empiricalmeasure}.

\begin{rem}\label{remXNgen}
The same phenomenon as the one described at Remark~\ref{rembarXgen} can be observed for~$\mu^N$: the expression for~$A^N$ can be interpreted in a similar way as the expression of $\r$-valued processes. The terms of the generator that depend on the first order derivatives corresponds to drift term in the dynamics of~$\mu^N$, the terms that depend on the second order derivatives to Brownian terms (here all Brownian motions $B^k,W$ give Brownian dynamics for $\mu^N$, since the law $\mu^N$, being an empirical measure instead of a (conditional) law, does not make disappear the Brownian dynamics of~$B^k$, on the contrary of the situation with~$\bar\mu$ for the Brownian motion~$B$). Compared to the generator of~$\bar\mu$, there is an additional term corresponding to the jump term of the measure-valued process~$\mu^N$ (whereas the process~$\bar\mu$ is continuous in time) that depends on the shift operator~$\Sh$ and on the increasing of the test-function~$G$.
\end{rem}

\section{Analysis of functions defined on spaces of probability measures}\label{analysismesvar}

As explained in the introduction, the notion of differentiability for measure-variable functions is not new, and some results of this section exist in different forms in the literature. However, the precise definitions and hypotheses are different from one reference to another. And, in this paper, we need technical properties on these functions from different areas (e.g. Taylor's formulas, regularity of stochastic flow of measures), hence it is necessary to write the precise definitions and results that we use to prove the theorems of the previous section.

\subsection{Definitions and some elementary properties}\label{def:diffmes}

The following notion of differentiability is a generalization of the one introduced in~\cite{dawson_measure-valued_1993}, and often referred to as the {\it linear derivative}.

\begin{defi}\label{def:deriv}
For~$m_0\in\P_1(\r)$, a function~$F:\P_1(\r)\longrightarrow\r$ is said to be differentiable at~$m_0$ if there exists a measurable sublinear function~$H_{m_0}:\r\rightarrow\r$ such that, for all~$m\in\P_1(\r)$,
\begin{equation}\label{derivH}
F(m) - F(m_0) = \int_\r H_{m_0}(x)\,d(m-m_0)(x) + \eps_{m_0}(m),
\end{equation}
where $\eps_{m_0}(m)/D_{KR}(m,m_0)$ vanishes as~$m$ converges to~$m_0$ for the metric~$D_{KR}$. The function~$H_{m_0}$ is called a version of the derivative of~$F$ at~$m_0$. We define the {\it canonical derivative} of~$F$ at~$m_0$ as the only version of the derivative that satisfies
$$\int_\r H_{m_0}(x)\,dm_0(x)=0,$$
and denote it by, for all~$x\in\r$,
$$\delta F(m_0,x) = H_{m_0}(x).$$
\end{defi}

\begin{rem}
It is easy to see that the versions of the derivative of~$F$ are never unique when they exist: if some function~$H_{m_0}:\r\rightarrow\r$ satisfies~\eqref{derivH}, then any function of the form~$x\mapsto H_{m_0}(x) + G(m_0)$ also satisfies~\eqref{derivH}. It can be noted (cf Lemma~\ref{lem:unique} below) that this notion of derivative is still unique for the following equivalence relation~: two functions~$H_1,H_2:\P_1(\r)\times\r\rightarrow\r$ are equivalent if for all~$m,x$, $H_1(m,x) - H_2(m,x)$ is independent of~$x$. Hence the canonical derivative of a differentiable function always exists and is unique.
\end{rem}

In the rest of the paper, we may omit "canonical" and ``linear'' and just mention "the derivative". For most of the properties that are studied in this paper, the canonical version of the derivative is not particularly needed, but it allows the definitions to be unquestionably well-posed, particularly for the notion of many times differentiable. Note that, the set $C^n_b(\P_1(\r))$ (see Definition~\ref{CnkP}) depends on the version of the derivative, and so Assumption~\ref{assuderivativediffu} also depends on our choice of canonical derivative. However, for example, the expressions of the infinitesimal generators given at Theorems~\ref{barXmarkov} and~\ref{XNmarkov} do not depend on this choice.

This particular choice of canonical derivative is motivated by the fact that it allows to recover ``natural properties'' (see Lemma~\ref{lemelem}), and it is close to the notion of derivative that is used in \cite{dawson_measure-valued_1993} (see Remark~\ref{defdawson}).


The following result allows to guarantee the uniqueness of the canonical derivative of differentiable functions.

\begin{lem}\label{lem:unique}
Let $m_0\in\P_1(\r)$ and $h_1,h_2:\r\rightarrow\r$ be measurable and sublinear. Assume that, for all~$m\in\P_1(\r)$,
$$\int_\r h_1(x)\,d(m-m_0)(x)=\int_\r h_2(x)\,d(m-m_0)(x) + \eps_{m_0}(m),$$
where $\eps_{m_0}(m)/D_{KR}(m,m_0)$ vanishes as~$m$ converges to~$m_0$. Then, for all~$x,y\in\r$,
$$h_1(x) - h_2(x) = h_1(y) - h_2(y).$$

In particular, if a function~$F:\P_1(\r)\rightarrow\r$ is differentiable at some~$m\in\P_1(\r)$, then its canonical derivative at~$m$ exists and is unique.
\end{lem}

Lemma~\ref{lem:unique} being not particularly useful for the understanding of the paper, we postpone its proof to Appendix~\ref{appendtechnic}.

Let us also state an elementary lemma guaranteeing expected properties for a differential operator. This lemma is implicitly used in the computation throughout the paper. For this lemma, the choice of the canonical derivative is important.
\begin{lem}\label{lemelem}$ $
\begin{itemize}
\item[$(a)$] Let $F:\P_1(\r)\rightarrow\r$ be a constant function. Then $F$ is differentiable at any~$m\in\P_1(\r)$, and its canonical derivative is the zero function.
\item[$(b)$] Let $F,G:\P_1(\r)\rightarrow\r$ be differentiable at~$m_0\in\P_1(\r),$ and~$\alpha\in\r$. Then, the function~$\alpha F+G$ is differentiable at~$m_0$, and, for all~$x\in\r$,
$$\delta (\alpha F + G) (m_0,x) = \alpha\delta F(m_0,x) + \delta G(m_0,x).$$
\item[$(c)$] Let $F,G:\P_1(\r)\rightarrow\r$ be differentiable at~$m_0\in\P_1(\r)$ and Lipschitz continuous. Then, the product function $FG$ is differentiable at~$m_0$, with, for all~$x\in\r$,
$$\delta (FG)(m_0,x) = F(m_0)\delta G(m_0,x) + G(m_0)\delta F(m_0,x).$$
\end{itemize}
\end{lem}

\begin{proof}
To prove Item~$(a)$, it is sufficient to notice that~\eqref{derivH} is trivially satisfied if~$H_{m_0}$ is the zero function, by considering $\eps_{m_0}$ to be the zero function. Item~$(b)$ is a straightforward consequence of the linearity of the integral operator. To prove Item~$(c)$, let us write, for all~$m\in\P_1(\r),$
\begin{align}
F(m)G(m) - F(m_0)G(m_0)=& F(m)G(m) - F(m_0)G(m)\nonumber\\
& + F(m_0)G(m) - F(m_0) G(m_0)\nonumber\\
=& F(m_0) \ll(G(m) - G(m_0)\rr) + G(m_0)\ll(F(m) - F(m_0)\rr) \label{produiteq}\\
&+ \ll(F(m) - F(m_0)\rr)\ll(G(m) - G(m_0)\rr).\nonumber
\end{align}

Besides, since~$F$ and~$G$ are assumed to be Lipschitz continuous, we have
$$\ll| \ll(F(m) - F(m_0)\rr)\ll(G(m) - G(m_0)\rr)\rr| \leq C D_{KR}(m,m_0)^2.$$

So, developing~\eqref{produiteq} using the differentiability of~$F$ and~$G$ proves the result.
\end{proof}

Let us notice that the derivability in the sense of Definition~\ref{def:deriv} is close to the one used in \cite{cox_controlled_2021} (as derivative), and in \cite{guo_itos_2023} and \cite{carmona_probabilistic_2018} (as linear derivative). It is however different than the derivative introduced by Pierre-Louis Lions, even if it is closely related (see Propositions~5.44,~5.48 and~5.51 of \cite{carmona_probabilistic_2018}).

Depending on the properties that are studied, each notion of differentiability can be more or less appropriate to state and prove the results. For the sake of clarity, in all the paper, only one notion of differentiability is used explicitly. We choose to use Definition~\ref{def:deriv} instead of Lions' derivative because it appears to be more appropriate to express the infinitesimal generators of measure-valued Markov processes, to retrieve properties of real-valued Markov processes (see Remarks~\ref{rembarXgen} and~\ref{remXNgen}). However, to define the notion of mixed derivatives (that is used in Assumption~\ref{assuderivativediffu}, and in the proof of Lemma~\ref{reguflot}), the Lions-derivative is particularly useful as a notation. Instead of stating the original and rather technical definition, we just give the expression of the Lions-derivative using the (linear) derivative. As in \cite{guo_itos_2023}, we use the notation~$\delta$ for the (linear) derivative, and $\partial$ for the Lions-derivative.
\begin{defi}[Lions-derivative]\label{def:lions}
Let $F : \P_1(\r)\rightarrow\r$ and~$m\in\P_1(\r)$. The function~$F$ is said to be $L$-differentiable at~$m$ if:
\begin{itemize}
\item $F$ is differentiable at~$m$ (in the sense of Definition~\ref{def:deriv}),
\item and $x\in\r\mapsto \delta F(m,x)$ is differentiable on~$\r$.
\end{itemize}
In this case, the Lions-derivative of~$F$ at~$m$ is the function~$\partial F(m,\bullet)$ defined as
$$\partial F (m,\bullet):x\in\r\longmapsto \partial F(m,x) = \partial_x \delta F(m,x).$$
\end{defi}

As mentioned earlier, the notion of derivative we use (cf Definition~\ref{def:deriv}) is related to the definition given at \cite[p. 19]{dawson_measure-valued_1993} in the following sense (which is a straightforward consequence of Lemma~\ref{cheminsegment} with $m_0=m$ and $m_1 = \delta_x$).
\begin{rem}\label{defdawson}
Let $F:\P_1(\r)\rightarrow\r$ be differentiable at~$m\in\P_1(\r)$, then, for all~$x\in\r$,
$$\frac{F((1-\eta)m + \eta\delta_x) - F(m)}{\eta}$$
converges, as~$\eta>0$ goes to zero, to the canonical derivative~$\delta F(m,x)$ of~$F$. In other words
$$\delta F(m,x) = \ll(\partial_\eta F((1-\eta)m + \eta\delta_x)\rr)_{\ll|\eta = 0\rr.}.$$
\end{rem}

One can note that the previous remark gives a useful way of characterizing the expression of the derivative of a given function, once it is known it is differentiable.

In the following, it is required to differentiate many times some functions. To define this notion, we use the following notation: if $G:\P_1(\r)\times\r^k\rightarrow \r$ (with $k\in\n^*$), then, for all~$x_1,...,x_k\in\r$
$$G_{(x_1,...,x_k)} : m \in\P_1(\r)\longmapsto G(m,x_1,...,x_k).$$

\begin{defi}\label{derivn}
A function~$F:\P_1(\r)\rightarrow\r$ is said to be $n$-times differentiable if:
\begin{itemize}
\item $F$ is $(n-1)$-times differentiable,
\item for all~$x\in\r^{n-1}$, the function~$(\delta^{n-1} F)_{x}$ is differentiable.
\end{itemize}

In that case we define inductively the canonical $n$-th derivative of~$F$ at~$m$ by: for all~$x=(x_1,...,x_n)\in\r^n$,
$$\delta^n F(m,x_1,...,x_n) = \delta \ll((\delta^{n-1} F)_{x\backslash_n}\rr)(m,x_n).$$
\end{defi}

The differentiability w.r.t. the measure-variable is not the only regularity that will be needed to prove our results. We also need regularity w.r.t. the real variables appearing when a function defined on~$\P_1(\r)$ is differentiated.

\begin{defi}\label{CnkP}
We define $C^{n,k}(\P_1(\r))$ (with $n\in\n^*$ and~$k\in\n\cup\{\infty\}$) the set of functions~$F:\P_1(\r)\rightarrow\r$ such that $F$ is $n$-times differentiable on~$\P_1(\r)$ and,  for every~$1\leq j\leq n$, the function
$$(m,x)\in\P_1(\r)\times\r^j \mapsto \delta^j F(m,x)$$
is $C^k$ w.r.t.~$x$ for fixed~$m$, and is continuous w.r.t.~$m$ for fixed~$x$.

In addition, we denote $C^{n}_b(\P_1(\r))$ the subspace of~$C^{n,n}(\P_1(\r))$ composed of the (non-necessarily bounded) functions~$F$ such that, for all~$1\leq k\leq n$ and~$1\leq j\leq k$,
$$\underset{m\in\P_1(\r),x\in\r^k}{\sup}\ll|\frac{\partial^k}{\partial x_1...\partial x_k}\delta^k F (m,x)\rr|<\infty~\textrm{ and }\underset{m\in\P_1(\r),x\in\r^k}{\sup}\ll|\partial_{x_j} \delta^k F(m,x)\rr|<\infty.$$
%
\end{defi}

\begin{rem}
The set $C^n_b(\P_1(\r))$ depends on the choice of the version of the derivatives because of the second boundedness conditions (i.e. boundedness of~$\partial_{x_j}\delta F^k(m,x)$). These conditions imply that, for any~$\mu_1,...,\mu_k\in\P_1(\r)$, the family of functions $(x\mapsto \delta F^k(m,x))_{m\in\P_1(\r)}$ is uniformly integrable w.r.t.~$\bigotimes_{i=1}^k \mu_i$.
\end{rem}


One can easily note that~$C^{1}_b(\P_1(\r))$ is included in the set of functions~$F:\P_1(\r)\rightarrow\r$ that are Lipschitz continuous functions w.r.t.~$D_{KR}$ (cf Proposition~\ref{mvthm} below).

In the following, the notion of many times $L$-differentiability is also used. The definition is essentially the same as Definition~\ref{derivn} for the Lions-derivative instead of the (linear) derivative.
\begin{defi}\label{derivnlions}
A function~$F:\P_1(\r)\rightarrow\r$ is said to be $n$-times $L$-differentiable if:
\begin{itemize}
\item $F$ is $(n-1)$-times $L$-differentiable,
\item for all~$x\in\r^{n-1}$, the function~$(\partial^{n-1} F)_{x}$ is $L$-differentiable.
\end{itemize}

In that case we define inductively the $n$-th Lions-derivative of~$F$ at~$m$ by: for all~$x=(x_1,...,x_n)\in\r^n$,
$$\partial^n F(m,x_1,...,x_n) = \partial \ll((\partial^{n-1} F)_{x\backslash_n}\rr)(m,x_n).$$
\end{defi}

Let us now introduce the notion of {\it mixed derivatives} for functions defined on~$\P_1(\r)\times\r^d$ for any~$d\in\n$. The {\it mixed derivatives} are defined as sets of functions.
\begin{defi}[Mixed derivatives]\label{mixderiv}
Let $d\in\n$ and $F : \P_1(\r)\times\r^d\rightarrow\r$. We say that $F$ admits mixed derivatives if, for all~$x=(x_1,...,x_d)\in\r^d$:
\begin{itemize}
\item for any~$1\leq k\leq d$ and~$m\in\P_1(\r)$, $y\in\r\mapsto F(m,x\backslash_k y)$ is differentiable on~$\r$,
\item and the function~$F_x$ is $L-$differentiable on~$\P_1(\r)$.
\end{itemize}

In that case, the mixed derivatives of~$F$ at~$m$ form the following set of~$d+1$ functions:
\begin{align*}
&\ll\{(m,x)\in\P_1(\r)\times \r^d \longmapsto \partial_{x_k} F(m,x)~:~1\leq k\leq d\rr\}\\
&\cup\ll\{(m,x,y)\in\P_1(\r)\times\r^d\times\r\longmapsto \partial (F_x)(m,y)\rr\},
\end{align*}
where $\partial (F_x)$ above denotes the Lions-derivative of~$F_x$ (cf Definition~\ref{def:lions}).

In addition, we define inductively the $n$-th order mixed derivatives of~$F$ as the union of the mixed derivatives of all the functions belonging to the set corresponding to the $(n-1)$-th order mixed derivatives of~$F$ (assuming all these mixed derivatives exist).
%
\end{defi}

\begin{ex}
Let $b:\P_1(\r)\times\r\rightarrow\r$. Then, the (first order) mixed derivatives of~$b$ are the two following functions:
\begin{align*}
(m,x)\in\P_1(\r)\times\r&\longmapsto \partial_x b(m,x),\\
(m,x,y)\in\P_1(\r)\times\r^2&\longmapsto \partial (b_x)(m,y). 
\end{align*}

The second order mixed derivatives of~$b$ are the five following functions
\begin{align*}
(m,x)\in\P_1(\r)\times\r&\longmapsto \partial^2_{xx} b(m,x),\\
(m,x,y)\in\P_1(\r)\times\r^2&\longmapsto \partial_x \partial (b_x)(m,y),\\
(m,x,y)\in\P_1(\r)\times\r^2&\longmapsto \partial_y \partial (b_x)(m,y),\\
(m,x,y)\in\P_1(\r)\times\r^2&\longmapsto \partial \ll(\partial_x b_x\rr)(m,y),\\
(m,x,y,z)\in\P_1(\r)\times\r^3&\longmapsto \partial^2 (b_x)(m,y,z). 
\end{align*}
\end{ex}

The only way the notion of {\it mixed derivatives} is used is to write Assumption~\ref{assuderivativediffu} in a compact way, stating that all the mixed derivatives up to order five of some functions are bounded. This kind of assumption is required to prove the regularity of the semigroups of the measure-valued Markov processes that we study in this paper (i.e. Proposition~\ref{regusemigroup}, and more precisely, the proof of Lemma~\ref{reguflot}).

Let us write now a remark that is a direct consequence of Proposition~\ref{mvthm}, explaining the reason behind the boundedness hypothesis of the mixed order derivatives.

\begin{rem}\label{remmix}
Let $F$ be a function admitting $n$-th order mixed derivatives (with $n\geq 2$) such that all its $n$-th order mixed derivatives are bounded. Then all the $(n-1)$-th order mixed derivatives of~$F$ are Lipschitz continuous: for any $G:\P_1(\r)\times\r^d$ belonging to the set of the $(n-1)$-th order mixed derivatives, there exists~$L>0$ such that, for all~$m,\mu\in\P_1(\r)$ and~$x,y\in\r^d$,
$$\ll|G(m,x) - G(\mu,y)\rr|\leq C\ll(D_{KR}(m,\mu) + ||x-y||_1\rr).$$
\end{rem}

Let us end this section with a notation.
\begin{defi}\label{bornederivn}
Let $n\in\n^*$ and $F:\P_1(\r)\rightarrow\r$ bounded and belonging to~$C^{n}_b(\P_1(\r))$. Let us assume that all the mixed derivatives of~$F$ of order up to~$n$ exist and are bounded. Then we denote $||F||_n$ a bound of all the $k$-th order mixed derivatives, for all~$0\leq k\leq n$.
\end{defi}

\subsection{Chain rules}

In this section, we state and prove some elementary results about "chain rules" for the differentiation of measure-variable functions. We want to prove the differentiability of functions of the following form:
\begin{equation*}
G\circ h: \r\overset{h}{\longrightarrow} \P_1(\r) \overset{G}{\longrightarrow} \r.
\end{equation*}

The general form that one could naturally expect for the derivative of this kind of function is the following: for~$t\in\r$,
\begin{equation}\label{Gohderiv}
(G\circ h)'(t) = \int_\r \delta G (h(t),x)\, d(h'(t))(x),
\end{equation}
where the sense of the derivative of~$h:\r\rightarrow\P_1(\r)$ has to be defined. In the first particular case, we consider, for two measures~$m_0,m_1\in\P_1(\r)$ the following function
$$h: t\in[0,1]\longmapsto (1-t)m_0 + tm_1,$$
where the derivative of~$h$ has to be understood as the quantity $(h(t)-h(t_0))/(t-t_0)$ which is constant and equal to~$m_1-m_0$.
\begin{lem}\label{cheminsegment}
Let $G:\P_1(\r)\rightarrow\r$ be differentiable, $m_0,m_1\in\P_1(\r)$, and
$$f : t\in[0,1]\longmapsto G((1-t)m_0 + tm_1).$$

Then $f$ is differentiable, and, for all~$t\in[0,1]$,
$$f'(t) = \int_\r \delta G((1-t)m_0 +tm_1, x)\, d(m_1-m_0)(x).$$
\end{lem}

\begin{proof}
We have, for any~$t,t_0\in[0,1],$
\begin{multline}\label{fGh}
f(t) - f(t_0)\\
=(t-t_0)\int_\r \delta G((1-t_0) m_0 + t_0 m_1,x) d(m_1-m_0)(x) + \eps((1-t)m_0 +tm_1),
\end{multline}
with $\eps(\mu)/D_{KR}(\mu,(1-t_0)m_0+t_0m_1)$ vanishing as $\mu$ converges to $(1-t_0)m_0+t_0m_1$. Noticing that
$$D_{KR}((1-t)m_0+tm_1,(1-t_0)m_0+t_0m_1)\leq 2|t-t_0|,$$
the proof of the lemma follows from~\eqref{fGh} dividing it by~$(t-t_0)$ and then letting~$t$ goes to~$t_0$.
\end{proof}

\begin{rem}\label{leibnizmesure}
We have the following consequence of Lemma~\ref{cheminsegment} and of the fundamental theorem of calculus: for any function~$G$ belonging to~$C^1_b(\P_1(\r))$, for all~$m_0,m_1\in\P_1(\r)$,
$$G(m_1) - G(m_0) = \int_0^1 \int_\r \delta G((1-t)m_0 + tm_1,x)d(m_1-m_0)(x).$$

The above property is used as the definition of the linear derivative by \cite{carmona_probabilistic_2018} (cf Section~4.1), \cite{guo_itos_2023} and \cite{cox_controlled_2021}. In addition, by Proposition~5.44 of \cite{carmona_probabilistic_2018}, under some generic assumptions, this property can characterize the linear derivative as it is defined at Definition~\ref{def:deriv} when the analysis of measure-variable functions is studied on~$\P_2(\r)$. 
\end{rem}

Now we prove a similar result for functions that are many times differentiable with bounded derivatives. In particular, the following result is not strictly stronger than the previous one.
\begin{lem}\label{cheminsegmentn}
Let~$n\in\n^*$, $G\in C^{n}_b(\P_1(\r))$, $m_0,m_1\in\P_1(\r)$, and
$$f : t\in[0,1]\longmapsto G((1-t)m_0 + tm_1).$$

Then $f$ is $C^n$, and, for all~$t\in[0,1]$,
$$f^{(n)}(t) = \int_{\r^n} \delta^n G((1-t)m_0 +tm_1, x)\, d(m_1-m_0)^{\otimes n}(x).$$
\end{lem}

\begin{proof}
We prove the statement by induction on~$n$. For~$n=1$, Lemma~\ref{cheminsegment} implies that $f$ is differentiable. The fact that $f$ is $C^1$ follows by the dominated convergence theorem since $G\in C^1_b(\P_1(\r))$.

For the general case (i.e. $n\geq 2$), let us assume that~$f$ is $C^{n-1}$ with, for all~$t\in[0,1]$,
$$f^{(n-1)}(t) = \int_{\r^{n-1}} \delta^{n-1} G((1-t)m_0 +tm_1, x)d(m_1-m_0)^{\otimes n-1}(x).$$
 
Note that, since $G\in C^n_b(\P_1(\r))$, the function
$$(t,x)\in [0,1]\times\r^{n-1}\longmapsto  \delta^{n-1} G((1-t)m_0 +tm_1, x)$$
is sublinear w.r.t.~$||x||_1$ (hence integrable w.r.t.~$m_1$ and~$m_0$) and differentiable w.r.t.~$t$ (thanks to Lemma~\ref{cheminsegment}) with derivative
$$(t,x)\longmapsto \int_\r \delta\ll(\ll( \delta^{n-1}G\rr)_x\rr)((1-t)m_0 +tm_1,y)d(m_1-m_0)(y),$$
which is also sublinear w.r.t.~$||x||_1$ (with a constant which is uniform w.r.t.~$t\in[0,1]$ since $G\in C^{n}_b(\P_1(\r)$). Then, by differentiation under the integral sign, for all~$t\in[0,1]$,
\begin{equation*}
f^{(n)}(t) = \int_{\r^{n-1}}\int_\r \delta\ll(\ll( \delta^{n-1}G\rr)_{x}\rr)((1-t)m_0 +tm_1,y)d(m_1-m_0)(y)d(m_1-m_0)^{\otimes n-1}(x).
\end{equation*}
Once again, the continuity of~$f^{(n)}$ is a consequence of the dominated convergence theorem and of the assumption that $G$ belongs to~$C^n_b(\P_1(\r))$.
\end{proof}

%
%

Another particular case of interest for the chain rule~\eqref{Gohderiv} is the one of the function
$$h : \lambda\in\r\longmapsto \Sh(m,\lambda),$$
for some fixed~$m\in\P_1(\r),$ with $\Sh$ a shift operator defined as:
$$\Sh(m,\lambda): A\in\mathcal{B}(\r)\longmapsto m(\{x-\lambda:x\in A\}).$$

In that case, the derivative of~$h$ could be seen as a reminiscent of the derivative in the sense of Schwartz' distribution, since it requires to differentiate the integrand instead of the measure in~\eqref{Gohderiv} (it is actually the opposite of one would expect from the derivative in the distribution sense).

\begin{lem}\label{lemshift}
Let $m\in\P_1(\r)$, $G:\P_1(\r)\rightarrow\r$ be differentiable such that, for any~$\mu\in\P_1(\r)$, $x\mapsto \delta G(\mu,x)$ is differentiable. Then the function
$$f : \lambda\in\r \longmapsto G(\Sh(m,\lambda))$$
is differentiable and, for all~$\lambda\in\r,$
$$f'(\lambda) = \int_\r \ll(\partial_y \delta G(\Sh(m,\lambda),x+y)\rr)_{|y=\lambda}dm(x).$$
\end{lem}

\begin{proof}
We have for any~$t\in\r$,
$$f(\lambda) - f(\lambda_0) = \int_\r \delta G(\Sh(m,\lambda_0),x)d(\Sh(m,\lambda) - \Sh(m,\lambda_0))(x) + \eps(\Sh(m,\lambda)),$$
with $\eps(\mu)/D_{KR}(\mu,\Sh(m,\lambda_0))$ vanishing as~$\mu$ converges to~$\Sh(m,\lambda_0)$ for~$D_{KR}$. Since, writing $m_0 := \Sh(m,\lambda_0),$
$$D_{KR}(\Sh(m,\lambda),\Sh(m,\lambda_0)) = D_{KR}(\Sh(m_0,\lambda-\lambda_0),m_0)= |\lambda-\lambda_0|,$$
we know that $\eps(\Sh(m,\lambda))/|\lambda-\lambda_0|$ vanishes as~$\lambda$ converges to~$\lambda_0$. And, by the Dominated Convergence Theorem,
\begin{align*}
&\frac1{\lambda-\lambda_0}\int_\r\delta G(\Sh(m,\lambda_0),x)d(\Sh(m,\lambda) - \Sh(m,\lambda_0))(x) \\
&\hspace*{4cm}= \int_\r \frac{\delta G(\Sh(m,\lambda_0),x+\lambda) - \delta G(\Sh(m,\lambda_0),x+\lambda_0)}{\lambda-\lambda_0}dm(x)\\
 &\hspace*{4cm}\underset{\lambda\rightarrow \lambda_0}{\longrightarrow}\int_\r \ll(\partial_y \delta G(\Sh(m,\lambda_0),x+y)\rr)_{|y=\lambda_0}dm(x),
\end{align*}
which ends the proof.
\end{proof}

The last case for which we want~\eqref{Gohderiv} is when the function~$h$ is of the form
$$h : x\in\r \longmapsto \mathcal{L}\ll(Y(x)|\T\rr).$$
with $(Y(x))_x$ a family of random variables and~$\T$ some sigma-field.
\begin{lem}\label{chainruleloi}
Let $G\in C^{1}_b(\P_1(\r))$, $\T$ some sigma-field and $Y(x)$ ($x\in\r$) be real-valued random variables such that, almost surely,
$$x\longmapsto Y(x)$$
belongs to~$C^1(\r)$, and assume that for all compact set~$K\subset\r$,
\begin{equation}\label{moment1Y}
\esp{\underset{x\in K}{\sup}|\partial_xY(x)|}<\infty.
\end{equation}
Then, almost surely, the function
$$f:x\in\r\longmapsto G\ll(\mathcal{L}[Y(x)|\T]\rr)$$
belongs to~$C^1(\r)$ and, for all~$x\in\r,$
$$f'(x) = \espc{\ll(\partial_x Y(x)\rr)\cdot \ll(\partial_y \delta G\ll(\mathcal{L}[Y(x)|\T],y\rr)\rr)_{\ll|y = Y(x)\rr.}}{\T}.$$
\end{lem}

\begin{proof}
For the sake of notation, let us denote, for all~$x\in\r$,
$$m_x = \mathcal{L}(Y(x)|\T).$$

For the rest of the proof, let us fix some~$x_0\in\r$ and some compact set~$K$ containing an open neighborhood of~$x_0$. Then, for~$x\in K$,
\begin{align*}
f(x) - f(x_0)=&G(m_x) - G(m_{x_0}) = \int_\r \delta G(m_{x_0},y)d\ll(m_x - m_{x_0}\rr)(y) + \eps_{x_0}(m_x)\\
=& \espc{\delta G(m_{x_0},Y(x))-\delta G(m_{x_0},Y(x_0))}{\T}+ \eps_{x_0}(m_x),
\end{align*}
with $\eps_{x_0}(\mu)/D_{KR}(\mu,m_{x_0})$ vanishing as $\mu$ converges to~$m_{x_0}$. Consequently, for all~$x\in K\backslash\{x_0\}$,
$$\frac{f(x) - f(x_0)}{x-x_0} = \espc{\frac{\delta G(m_{x_0},Y(x))-\delta G(m_{x_0},Y(x_0))}{x-x_0}}{\T}+ \frac{\eps_{x_0}(m_x)}{x-x_0}.$$

Since, for all~$x\in K$,
$$D_{KR}(m_x,m_{x_0})\leq \espc{|Y(x) - Y(x_0)|}{\T}\leq |x-x_0|\cdot \espc{\underset{\tilde x\in K}{\sup}\ll|\partial_{\tilde x} Y(\tilde x)\rr|}{\T},$$
we know, by~\eqref{moment1Y}, that $\eps_{x_0}(m_x)/|x-x_0|$ vanishes almost surely as~$x$ goes to~$x_0$.

Then, using~\eqref{moment1Y} and the hypothesis that $G$ belongs to~$C^{1}_b(\P_1(\r))$, the dominated convergence theorem concludes the proof.
\end{proof}

\begin{rem}\label{remchainruleloi}
With the notation of Lemma~\ref{chainruleloi}, the expression of $f'$ can be substantially simplified using the Lions-derivative instead of the (linear) derivative: for all~$x\in\r,$
$$f'(x) = \espc{\ll(\partial_x Y(x)\rr)\cdot \partial G\ll(\mathcal{L}[Y(x)|\T],Y(x)\rr)}{\T}.$$
\end{rem}

\subsection{Taylor's formulas and some corollaries}

The results of this section are mainly consequences of Taylor's theorems applied to the function~$f$ of Lemma~\ref{cheminsegment} from the previous section. Let us give the statement for the Mean Value Theorem for measure-variable functions. The proof consisting in applying directly the Mean Value Theorem for the real-variable function~$f$ of Lemma~\ref{cheminsegment}, we omit it.
\begin{prop}[Mean Value Theorem]\label{mvthm}
Let $G: \P_1(\r)\rightarrow\r$ be differentiable. Then, for any~$m_0,m_1\in\P_1(\r)$, there exists some~$t\in]0,1[$ such that
$$G(m_1) - G(m_0) = \int_\r \delta G((1-t)m_0 +tm_1, x)d(m_1-m_0)(x).$$

In particular, every function~$G$ belonging to~$C^{1}_b(\P_1(\r))$ is Lipschitz continuous: for all~$m_0,m_1\in\P_1(\r)$,
$$\ll|G(m_1) - G(m_0)\rr|\leq D_{KR}(m_0,m_1)\cdot \underset{m\in\P_1(\r),x\in\r}{\sup}\ll|\partial_x \delta G(m,x)\rr|.$$
\end{prop}

Let us now generalize the previous result by stating a version of Taylor's formula with integral remainder for measure-variable functions (which is already stated and proved in Lemma~2.2 of \cite{chassagneux_weak_2022}).
\begin{thm}[Taylor's formula with integral remainder]\label{taylorintegral}
Let~$n\in\n^*$ and $G\in C^{n+1}_b(\P_1(\r))$. Then, for any~$m_0,m_1\in\P_1(\r)$,
\begin{multline*}
G(m_1) - \sum_{k=0}^n \frac1{k!}\int_{\r^k}\delta^k G(m_0,x)\,d(m_1-m_0)^{\otimes k}(x)\\
= \frac1{n!}\int_0^1 (1-t)^n\int_{\r^{n+1}}\delta^{n+1} G((1-t)m_0 + tm_1,x)\,d(m_1-m_0)^{\otimes n+1}(x)\,dt.
\end{multline*}
\end{thm}

\begin{proof}
By Lemma~\ref{cheminsegmentn}, the function
$$f : t\in[0,1]\longmapsto G((1-t)m_0 + tm_1)$$
belongs to~$C^{n+1}([0,1])$, and for all~$0\leq k\leq n+1$ and~$t\in[0,1]$,
$$f^{(k)}(t) = \int_{\r^k} \delta^k G((1-t)m_0 + tm_1,x)d(m_1-m_0)^{\otimes k}(x).$$

Then, by the (standard) Taylor's formula with remainder term,
$$f(1) - \sum_{k=0}^n \frac{f^{(k)}(0)}{k!} = \frac{1}{n!}\int_0^1 f^{(n+1)}(t) (1-t)^n dt.$$

This last equality is exactly the statement of the theorem.
\end{proof}

In order to prove a version of the Taylor-Lagrange's inequality for measure-variable functions, the following lemma about Kantorovich-Rubinstein metric is required.
\begin{lem}\label{DKRn}
Let~$n\in\n^*$ and $g\in C_b^n(\r^n)$. Then, for any~$m_0,m_1\in\P_1(\r)$,
\begin{equation}\label{inductionDKRn}
\ll|\int_{\r^n} g(x)d(m_1-m_0)^{\otimes n}(x)\rr| \leq D_{KR}(m_0,m_1)^n \cdot\underset{x_1,...,x_n\in\r}{\sup}\ll|\frac{\partial^n}{\partial x_1...\partial x_n} g(x_1,...,x_n)\rr|.
\end{equation}
\end{lem}

\begin{proof}
We prove the result by induction on~$n$. For~$n=1$, it is a mere consequence of the definition of~$D_{KR}$, and of the fact that, any $g\in C_b^1(\r)$ is Lipschitz continuous with Lipschitz constant non-greater than~$||g'||_\infty$.

Then, let~$n\in\n^*$ and assume that~\eqref{inductionDKRn} holds true for any~$g\in C_b^n(\r^n)$. Let~$h\in C_b^{n+1}(\r^{n+1})$, and
$$g:x\in\r^n\longmapsto \int_\r h(x,y)d(m_1-m_0)(y),$$
which belongs to~$C^{n+1}_b(\r^n).$ So
\begin{align*}
\ll|\int_{\r^{n+1}} h(x)d(m_1-m_0)^{\otimes n+1}(x)\rr|=& \ll|\int_{\r^n} g(x) d(m_1-m_0)^{\otimes n}(x)\rr|\\
\leq& D_{KR}(m_0,m_1)^n\cdot \underset{x_1,...,x_n\in\r}{\sup}\ll|\frac{\partial^n}{\partial x_1...\partial x_n} g(x_1,...,x_n)\rr|.
\end{align*}

And, thanks to the case~$n=1$, we know that
\begin{multline*}
\ll|\frac{\partial^n}{\partial x_1...\partial x_n} g(x_1,...,x_n)\rr| = \ll|\int_{\r} \frac{\partial^n}{\partial x_1...\partial x_n} h(x,y)d(m_1-m_0(y))\rr|\\
\leq D_{KR}(m_0,m_1)\cdot \underset{x_1,...,x_n,y\in\r}{\sup}\ll|\frac{\partial^{n+1}}{\partial x_1...\partial x_n\partial y} h(x_1,...,x_n,y)\rr|.
\end{multline*}

Combining the two previous inequalities proves the statement of the lemma.
\end{proof}

As a straightforward consequence of Theorem~\ref{taylorintegral} and Lemma~\ref{DKRn} we obtain the following result.
\begin{thm}[Taylor-Lagrange's inequality]\label{taylorlagrange}
Let~$n\in\n^*$ and $G\in C^{n+1}_b(\P_1(\r))$. Then, for any~$m_0,m_1\in\P_1(\r)$,
\begin{multline*}
\ll|G(m_1) - \sum_{k=0}^n \frac1{k!}\int_{\r^k}\delta^k G(m_0,x)d(m_1-m_0)^{\otimes k}(x)\rr|\\
\leq \frac1{(n+1)!}D_{KR}(m_0,m_1)^{n+1}\cdot \underset{x\in\r^{n+1},m\in\P_1(\r)}{\sup}\ll|\frac{\partial^{n+1}}{\partial x_1...\partial x_{n+1}}\delta^{n+1}G(m,x)\rr|.
\end{multline*}
\end{thm}

Another consequence of Theorem~\ref{taylorintegral} is that, in some sense, it is possible to interchange the derivative w.r.t. the measure-variable and the partial derivative w.r.t. the real-variable, when it makes sense. For the sake of readability, the result is only stated for second order derivatives, but can be generalized inductively.
\begin{cor}\label{corechange}
Let $G\in C^3_b(\P_1(\r))$. Then, for all~$m\in\P_1(\r)$, and~$x,y\in\r$,
$$\partial_x\delta \ll((\delta G)_x\rr)(m,y) = \delta\ll(\partial_x (\delta G)_x\rr)(m,y),$$
or, with the Lions-derivative notation,
$$\partial^2_{xy} \delta^2 G(m,x,y)=\partial^2 G(m,x,y).$$
\end{cor}

\begin{proof}
The proof consists in showing that the function
$$(m,x,y)\in\P_1(\r)\times\r\times\r \longmapsto \partial_x\delta \ll((\delta G)_x\rr)(m,y)$$
is the derivative of the function
$$(m,x)\in\P_1(\r)\times\r \longmapsto \partial_x \delta G(m,x)$$
in the sense of Definition~\ref{def:deriv}. Let us fix in all the proof, some $m,m_0\in\P_1(\r).$

Firstly, by differentiation under the integral sign (which is permitted since $G\in C^3_b(\P_1(\r))$), we have, for any~$x\in\r$,
$$\partial_x \int_\r \delta^2 G(m,x,y)\,d(m-m_0)(y) = \int_\r \partial_x \delta^2 G(m,x,y)\,d(m-m_0)(y),$$
such that,
\begin{align*}
&\partial_x \delta G(m,x) - \partial_x \delta G(m_0,x) -  \int_\r \partial_x \delta^2 G(m,x,y)\,d(m-m_0)(y)\\
&~~~~= \partial_x\ll(\delta G(m,x) - \delta G(m_0,x) -  \int_\r \delta^2 G(m,x,y)\,d(m-m_0)(y)\rr)\\
&~~~~= \partial_x~\int_0^1 (1-t)\int_{\r^2} \delta^3 G((1-t)m_0 + tm,x,y_1,y_2)\,d(m-m_0)^{\otimes 2}(y_1,y_2)dt\\
&~~~~= \int_0^1 (1-t)\int_{\r^2} \partial_x\delta^3 G((1-t)m_0 + tm,x,y_1,y_2)\,d(m-m_0)^{\otimes 2}(y_1,y_2)dt,
\end{align*}
where we have used the Taylor's formula with integral remainder (cf Theorem~\ref{taylorintegral}) to obtain the before last equality, and we have once again differentiated under the integral sign to obtain the last one. Then, by Lemma~\ref{DKRn} and recalling that $G\in C^3_b(\P_1(\r))$, we obtain
\begin{multline*}
\ll|\partial_x \delta G(m,x) - \partial_x \delta G(m_0,x) -  \int_\r \partial_x \delta^2 G(m,x,y)\,d(m-m_0)(y)\rr|\\
\leq D_{KR}(m,m_0)^2\cdot \underset{\mu\in\P_1(\r),y\in\r^3}{\sup}\ll|\frac{\partial^3}{\partial y_1 \partial y_2 \partial y_3}\delta^3 G(\mu,y)\rr|,
\end{multline*}
which proves the result.
\end{proof}

Let us end this section with two applications of Taylor-Lagrange's inequality (i.e. Theorem~\ref{taylorlagrange}) that permits to differentiate some particular functions. Their proofs relying only on technical analysis arguments about integral convergences, they are postponed to Appendix~\ref{appendtechnic}.
\begin{cor}\label{corderivint1}
Let $H:\P_1(\r)\times\r\rightarrow\r$ and
$$F : m\in\P_1(\r)\longmapsto \int_\r H(m,x)dm(x).$$
Assume that:
\begin{itemize}
\item[$(i)$] for any~$x\in\r$, the function~$H_x$ belongs to~$C^{2}_b(\P_1(\r))$, and there exists~$C>0$ such that, for all~$x\in\r$,
$$\underset{\substack{m\in\P_1(\r)\\y_1,y_2\in\r}}{\sup}\ll|\partial^2_{y_1y_2} \delta^2(H_x)(m,y_1,y_2)\rr|\leq C(1+|x|),$$
\item[$(ii)$] for any~$m\in\P_1(\r)$ and~$x\in\r$, the function $y\mapsto H(m,y)$ is $C^1$ on~$\r$, and $\partial_x H_x$ belongs to~$C^{1,1}(\P_1(\r))$, and
$$\underset{\substack{m\in\P_1(\r)\\x,y\in\r}}{\sup}\ll|\partial_y \delta \ll(\partial_x H_x\rr)(m,y)\rr|<\infty.$$
\end{itemize}

Then, $F$ is differentiable on~$\P_1(\r)$, with: for all~$m\in\P_1(\r)$ and~$x\in\r$,
$$\delta F(m,x) = H(m,x) + \int_\r \delta H_y(m,x)dm(y) - F(m).$$
\end{cor}

The next result is similar as the previous one for a larger class of functions.

\begin{cor}\label{corderivint2}
Let $H:\P_1(\r)\times\r^2\rightarrow\r$ and
$$F : m\in\P_1(\r)\longmapsto \int_{\r^2} H(m,x)dm^{\otimes 2}(x).$$
We recall that, for $x\in\r^2$,
$$H_x : m\in\P_1(\r) \longmapsto H_x(m) := H(m,x).$$
Assume that:
\begin{itemize}
\item[$(i)$] for any~$x\in\r^2$, the function~$H_x$ belongs to~$C^{2}_b(\P_1(\r))$, and there exists~$C>0$ such that, for all~$x_1,x_2\in\r$,
$$\underset{\substack{m\in\P_1(\r)\\y_1,y_2\in\r}}{\sup}\ll|\partial^2_{y_1y_2} \delta^2(H_x)(m,y_1,y_2)\rr|\leq C(1+|x_1| + |x_2| + |x_1x_2|),$$
\item[$(ii)$] for any~$m\in\P_1(\r)$, the functions $x\in\r^2\mapsto H(m,x)$ is $C^1$, and for any~$\in\r^2,$ both $\partial_{x_1} H_x$ and~$\partial_{x_2} H_x$ belong to~$C^{1,1}(\P_1(\r))$, with
$$\underset{\substack{m\in\P_1(\r)\\x_1,x_2,y\in\r}}{\sup}\ll|\partial_y \delta \ll(\partial_{x_1} H_x\rr)(m,y)\rr|+\ll|\partial_y \delta \ll(\partial_{x_2} H_x\rr)(m,y)\rr|<\infty,$$
\item[$(iii)$] for all~$m\in\P_1(\r)$, the function $x\mapsto H(m,x)$ belongs to~$C^2_b(\r^2)$.
\end{itemize}

Then, $F$ is differentiable on~$\P_1(\r)$, with: for all~$m\in\P_1(\r)$ and~$y\in\r$,
\begin{multline*}
\delta F(m,y) = \int_{\r^2}\delta (H_x)(m,y)dm^{\otimes 2}(x) + \int_\r H(m,x_1,y)dm(x_1) \\+ \int_\r H(m,y,x_2)dm(x_2) - 2F(m).
\end{multline*}
\end{cor}

\begin{rem}\label{remcorderiv}
Note that, an idea to prove Corollary~\ref{corderivint2} could have been to write
$$F(m) = \int_{\r^2} H(m,x)dm^{\otimes 2}(x) = \int_\r G(m,x_2)dm(x_2),$$
with
$$G(m,x_2) = \int_\r H(m,x_1,x_2)dm(x_1),$$
and to apply twice Corollary~\ref{corderivint1}. The problem is that it would require more technical and stronger assumptions to be true than the assumptions of Corollary~\ref{corderivint2}.
\end{rem}

\subsection{Measure-variable polynomials}\label{sec:poly}

A set of functions of interest is the following one
$$\ll\{F : m\in\P_1(\r)\longmapsto \int_\r g(x)dm(x)~:~g\textrm{ smooth enough}\rr\}.$$
Indeed the computation (and the stochastic calculus properties) on this kind of functions defined on~$\P_1(\r)$ rely directly on the computation for functions defined on~$\r$. Since this set of functions is not a separating class for~$\P_1(\r)$, we need to introduce the richer class of functions composed of the {\it polynomials} (similar functions are also used in \cite{dawson_measure-valued_1993}, \cite{cuchiero_probability_2019} and \cite{guo_itos_2023}).



%
%

\begin{defi}[Polynomials]\label{def:poly}
For $p\in \n\cup\{\infty\}$, let $\poly^p$ be the set of polynomials of order~$p$, i.e. the functions~$F$ of the form
$$F : m\in\P_1(\r)\longmapsto \int_{\r^n} h(x)dm^{\otimes n}(x),$$
with $n\in\n$ and $h\in C^p_b(\r^n).$
\end{defi}

There exists another useful class of measure-variable functions called the {\it cylinder functions} (see e.g. \cite{cox_controlled_2021} and \cite{guo_itos_2023}) that can be defined as the set
$$\ll\{F : m\in\P_1(\r)\longmapsto g\ll(\int_\r h_1(x)dm(x),...,\int_\r h_n(x)dm(x)\rr):g,h_1,...,h_n\textrm{ smooth}\rr\}.$$

Note that the terminology used by \cite{guo_itos_2023} slightly differs from the one of this paper (in \cite{guo_itos_2023}, the functions $g,h_1,...,h_n,h$ above are required to be polynomials). Consequently, in this paper the cylinder functions and the polynomials form two different sets of functions. In this paper, the class of polynomials satisfies an important property: if $(P_t)_t$ is the semigroup of a particular measure-valued Markov process, then, for any~$G\in\poly^4$ and~$t\geq 0$, $P_tG$ belongs to~$\poly^4$ (see Proposition~\ref{regusemigroup} for a formal statement). The previous property does not seem to be true for other classes of measure-variable functions like the cylinder functions.

%

As mentioned above, the following proposition states that the set of the polynomials is a separating class for~$\P_1(\r)$. 
\begin{prop}\label{polysepare}
Let~$p\geq 1$ and $\mu_1,\mu_2$ be $\P_p(\r)$-valued random variables. If for all~$F\in\poly^\infty$,
$$\esp{F(\mu_1)} = \esp{F(\mu_2)},$$
then~$\mu_1$ and~$\mu_2$ have the same law.
\end{prop}

The proof of this statement being quite both classical and technical, it is postponed to the end of Appendix~\ref{append:separe}.


The next result concerns the regularity of the polynomials.

\begin{prop}\label{regupoly}
For all~$p\geq 2,$ the set~$\poly^p$ is included in both~$C^{\infty,p}(\P_1(\r))$ and~$C^{p}_b(\P_1(\r))$. The mixed derivatives of any~$G\in\poly^p$ up to order~$p$ exist and are bounded. In addition, for any~$n\in\n^*$, for all~$\phi\in C^p_b(\r^n)$, defining $G\in\poly^p$ as
$$G : m\in\P_1(\r)\longmapsto \int_{\r^n} \phi(x)dm^{\otimes n}(x),$$
we have, for all~$m\in\P_1(\r)$ and~$y\in\r$,
$$\delta G(m,y)= \sum_{k=1}^n \int_{\r^{n-1}} \phi(x\backslash_k y)dm^{\otimes n-1}(x\backslash_k) - n\cdot G(m),$$
and, for all~$m\in\P_1(\r)$ and~$y_1,y_2\in\r$,
$$\delta^2 G(m,y_1,y_2) = \sum_{k=1}^{n} \sum_{\substack{l=1\\l\neq k}}^n \int_{\r^{n-2}}\phi(x\backslash_{(k,l)}(y_1,y_2))dm^{\otimes n-2}(x\backslash_{(k,l)}) + C_1(m,y_1) + C_2(m,y_2),$$
with $C_1(m,y_1)$ (resp. $C_2(m,y_2)$) independent of~$y_2$ (resp. $y_1$).
\end{prop}

Before proving Proposition~\ref{regupoly}, let us remark that it is possible to have explicit and smooth expressions for the quantities denoted by~$C_1(m,y_1)$ and~$C_2(m,y_2)$: with the notation of Proposition~\ref{regupoly},
\begin{align*}
C_1(m,y_1)=& -(n-1)\int_{\r^{n-1}} \phi(x\backslash_k y_1)dm^{\otimes n-1}(x\backslash_k),\\
C_2(m,y_2)=& - n\cdot \delta G(m,y_2).
\end{align*}

We have omitted these expressions in the statement of the proposition since we do not use them in the proofs of the paper. It can still be of interest to note that the functions~$(m,y)\mapsto C_1(m,y)$ and~$(m,y)\mapsto C_2(m,y)$ are polynomials of order~$p$ w.r.t.~$m$, and $C^p_b(\r)$ w.r.t.~$y$.

\begin{rem}\label{polycontinu}
As a consequence of Proposition~\ref{regupoly}, it can be noticed that all the polynomials are continuous on the space~$\P_1(\r)$, and so, also on each space~$\P_p(\r)$, since $p$-th order Wasserstein metric is finer than~$D_{KR}$ (for any~$p\geq 1$).
\end{rem}

\begin{proof}[Proof of Proposition~\ref{regupoly}]
Let us fix $p\geq 2$ in all the proof. We prove, by induction on~$n\in\n^*$ that, for any~$\phi\in C^p_b(\r^n)$, and $G$ defined as
\begin{equation}\label{Gdefinduction}
G : m\in\P_1(\r)\longmapsto \int_{\r^n} \phi(x)dm^{\otimes n}(x),
\end{equation}

the function~$G$ is differentiable on~$\P_1(\r)$ such that: for all~$m\in\P_1(\r)$ and~$y\in\r$,
\begin{equation}\label{induction}
\delta G(m,y)= \sum_{k=1}^n \int_{\r^{n-1}} \phi(x\backslash_k y)dm^{\otimes n-1}(x\backslash_k) + C(m).
\end{equation}

Let us begin with the case~$n=1$. So, let us consider some~$\phi\in C^p_b(\r)$ and define
$$G:m\in\P_1(\r) \longmapsto \int_\r \phi(x)dm(x).$$

Since, for any~$m,m_0\in\P_1(\r),$
$$G(m) - G(m_0) = \int_\r \phi(x)d(m-m_0)(x),$$
we know that, the function~$(m,x)\mapsto \phi(x)$ is one version of the derivative of~$F$. Its canonical derivative therefore satisfies: for all~$m\in\P_1(\r),x\in\r$,
$$\delta G(m,x) = \phi(x) - G(m).$$


Now, let us fix some~$n\in\n^*$, and assume that~\eqref{induction} holds true for any function~$G$ of the form~\eqref{Gdefinduction} (for any choice of~$\phi\in C^p_b(\r^n)$). Let $\phi\in C^p_b(\r^{n+1})$ and
$$G : m\in\P_1(\r)\longmapsto \int_{\r^{n+1}}\phi(x)dm^{\otimes n+1}(x).$$

Let us introduce
$$H : (m,y)\in \P_1(\r)\times\r \longmapsto \int_{\r^n}\phi(x,y)dm^{\otimes n}(x),$$
such that, for all~$m\in\P_1(\r)$,
$$G(m) = \int_\r H(m,y)dm(y).$$

Now, we prove that $G$ is differentiable using Corollary~\ref{corderivint1}. Firstly, by induction hypothesis, applied twice successively, for all~$y\in\r$, the function~$H_y$ is twice differentiable on~$\P_1(\r)$, with: for all~$m\in\P_1(\r),z_1,z_2\in\r$,
\begin{multline*}
\delta^2 H_y(m,z_1,z_2) = \sum_{k=1}^n \sum_{\substack{l=1\\l\neq k}}^n \int_{\r^{n-2}}\phi(x\backslash_{(k,l)}(z_1,z_2),y)dm^{\otimes n-2}(x\backslash_{(k,l)})\\ + C_y(m,z_1) + C_y(m,z_2).
\end{multline*}

Then,
$$\partial^2_{z_1 z_2}\delta^2 H_y(m,z_1,z_2) = \sum_{k=1}^n \sum_{\substack{l=1\\l\neq k}}^n \int_{\r^{n-2}}\partial^2_{kl}\phi(x\backslash_{(k,l)}(z_1,z_2),y)dm^{\otimes n-2}(x\backslash_{(k,l)}),$$
which is bounded uniformly w.r.t.~$(y,z_1,z_2,m)$ (recalling that $\phi\in C^p_b(\r^{n+1})$ with $p\geq 2$). Hence, the function~$H$ satisfies Condition~$(i)$ of Corollary~\ref{corderivint1}.

Besides, by differentiation under the integral sign, we have, for all~$m\in\P_1(\r),y\in\r,$
$$\partial_y H_y(m) = \int_{\r^n} \partial_n \phi(x,y)dm^{\otimes n}(x).$$

And, by induction hypothesis and differentiating again under the integral sign, for all~$m\in\P_1(\r)$ and~$y,z\in\r,$
$$\partial_z \delta \ll(\partial_y H_y\rr)(m,z) = \sum_{k=1}^n \int_{\r^{n-1}}\partial^2_{kn}\phi\ll(x\backslash_k z,y\rr)dm^{\otimes n-1}(x),$$
which is bounded uniformly w.r.t.~$(y,z,m)$. So $H$ also satisfies Condition~$(ii)$ of Corollary~\ref{corderivint1}.

Then, by Corollary~\ref{corderivint1}, the function~$G$ is differentiable, and for all~$m\in\P_1(\r),y\in\r,$
\begin{align*}
&\delta G(m,y) = H(m,y) + \int_\r \delta H_z(m,y)dm(z) - G(m)\\
&~~~~= \int_{\r^n}\phi(x,y)dm^{\otimes n}(x) + \sum_{k=1}^{n}\int_\r\int_{\r^{n-1}}\phi(x\backslash_k y,z)dm^{\otimes n-1}(x\backslash_k)dm(z) + C(m)\\
&~~~~= \int_{\r^n}\phi(x,y)dm^{\otimes n}(x) + \sum_{k=1}^{n}\int_{\r^{n}}\phi(x\backslash_k y)dm^{\otimes n}(x\backslash_k) + C(m),
\end{align*}
which is exactly~\eqref{induction}.

Now, one can note that, thanks to~\eqref{induction}, for any~$G\in\poly^p$ (for $p\geq 2$), for any~$x\in\r$, the function $(\delta G)_x$ still belongs to~$\poly^p$ and the function~$\partial_x (\delta G)_x$ belongs to~$\poly^{p-1}$. This implies that $\poly^p$ is included in $C^{\infty,0}(\P_1(\r))$. All the other statements of the proposition are a direct consequence of this result and the fact that, by definition, any $G\in\poly^p$ can be written as in~\eqref{Gdefinduction} with $\phi\in C^p_b(\r^n).$
\end{proof}

\begin{rem}
Similarly as it was noted in Remark~\ref{remcorderiv}, it is possible to differentiate measure-variable functions that include all the functions of the forms given at Corollaries~\ref{corderivint1},~\ref{corderivint2} and Proposition~\ref{regupoly}, considering
$$G : m \in\P_1(\r)\longmapsto \int_{\r^d} H(m,x)dm^{\otimes d}(x).$$

However, such a result would need particularly technical and strong hypothesis that what is actually needed in this paper.
\end{rem}

\section{Markov theory of measure-valued processes}\label{markov}

The aim of this section is to study the Markov properties of two kind of measure-valued Markov processes: the conditional laws of McKean-Vlasov processes (cf Theorem~\ref{barXmarkov}) and the empirical measures of some McKean-Vlasov particle systems (cf Theorem~\ref{XNmarkov}). Before proving the results of Section~\ref{resultmarkov}, we need to introduce some definitions and some useful lemmas.

\subsection{General results}

Let us begin by introducing the notion of semigroups and generators that we use in the paper. This section is written for $E$-valued Markov processes, and is applied for~$E=\P_p(\r)$. It is assumed in all the section that $E$ is a Polish space (the fact that $\P_p(\r)$ is Polish is guaranteed by Theorem~6.18 of \cite{villani_optimal_2009}). The notion of semigroup is the usual one, and the notion of infinitesimal generator has to be understood in a sense of martingale problem. Notice that these are the definitions used in \cite{meyn_stability_1993}, and also in \cite{erny_mean_2022} only in the case $E=\r$. The results of this section generalize the ones of Appendix~$A$ and~$B$ of \cite{erny_mean_2022}.

\begin{defi}\label{defi:markov}
Let $E$ be a Polish space, and $(X_t)_t$ be some time homogeneous $E$-valued Markov process w.r.t. some filtration~$(\F_t)_t$. Let $\D_S(X)$ denote the domain of the semigroup of~$X$:
$$\D_\S(X) = \ll\{g : E\rightarrow\r\textrm{ measurable}~:~\forall t>0, \espcc{x}{|g(X_t)|}<\infty\rr\}.$$

The semigroup~$(P_t)_t$ of~$(X_t)_t$ is a family of operators defined by: for any~$g\in \D_\S(X)$, for all~$t\geq 0$,
$$P_tg : x\in E\longmapsto \espcc{x}{g(X_t)},$$
with $\E_x$ the expectation related to the probability measure~$\p_x$ under which~$X_0=x$.

Let $\D_\G(X)$ denote the domain of the generator of~$X$: $g\in \D_\G(X)$ if and only if $g\in \D_\S(X)$ and there exists~$h_g\in \D_\S(X)$ such that, for all~$t\geq 0$,
$$\ll\{\begin{array}{l}
\displaystyle\E_x \ll[\int_0^t |h_g(X_s)|ds\rr] <\infty,\\
\displaystyle P_tg(x) - g(x) - \int_0^t P_sh_g(x)ds = 0.
\end{array}\rr.$$
In this case, the generator~$A$ of~$(X_t)_t$ is the operator defined on~$\D_\G(X)$ as $A g= h_g$.
\end{defi}

Note that, thanks to Remark~\ref{polycontinu}, any polynomial is continuous (hence measurable) w.r.t. the topology of~$\P_p(\r)$, for all~$p\geq 1$. In this paper, these are the only functions that are considered in the domains of the generators of the Markov processes.

\begin{rem}
In this paper, we give the expressions of the generators for two classes of measure-valued Markov processes. In Theorem~\ref{barXmarkov}, we study the generator of processes defined as conditional laws of solutions of McKean-Vlasov Ito-SDEs, conditionally on a Brownian motion. In Theorem~\ref{XNmarkov}, we give a similar result for processes defined as empirical measures of particle systems whose dynamics are driven by McKean-Vlasov Ito-SDEs.
\end{rem}

We can now state a criterion for a $\P_p(\r)$-valued process to be Markov. Since it is a direct consequence of the fact that the polynomials form a separating class for~$\P_p(\r)$ (i.e. Proposition~\ref{polysepare}), the proof is omitted.
\begin{prop}\label{markovcritere}
Let $p\geq 1$ and $(\mu_t)_t$ be a $\P_p(\r)$-valued process adapted to some filtration~$(\F_t)_t$. Assume that, for any polynomial~$F\in\poly^\infty$, for all~$0\leq s\leq t$,
$$\espc{F(\mu_t)}{\F_s} = \espc{F(\mu_t)}{\mu_s}.$$
Then $(\mu_t)_t$ is a Markov process w.r.t. the filtration~$(\F_t)_t$.
\end{prop}


The exhaustive study of Markov theory for measure-valued processes is beyond the scope of this section. The aim is only to prove the following result stating a classical Trotter-Kato formula under some ad hoc assumptions.
%

\begin{prop}\label{prop:trotter}
Let~$(\bar P_t)_t$ and~$(P^N_t)_t$ (resp. $\bar A$ and $A^N$) be the semigroups (resp. generators) of two $E$-valued Markov processes~$(\bar X_t)_t$ and~$(X^N_t)_t$, let $\A$ be a subset of~$\mathcal{C}(E)$ and~$x\in E$. Assume that:
\begin{enumerate}
\item for all~$g\in\A$ and~$t\geq 0,$ the function $y\in E\mapsto \bar P_tg(y)$ still belongs to~$\A$,
\item the set~$\A$ is included in the domains of the generators~$\bar A$ and~$A^N$,
\item for any~$g\in\A$, the three functions
$$t\mapsto \bar P_t\bar Ag(x)\textrm{ , }~t\mapsto P^N_t A^Ng(x)\textrm{ and }t\mapsto P^N_t \bar Ag(x)$$
 are continuous,
\item and, for all~$T\geq 0,$ $g\in \A$ and~$t\in[0,T]$,
$$\underset{r\leq T}{\sup}\ll|P^N_r\bar A\ll(\bar P_t - \bar P_s\rr)g(x)\rr| + \underset{r\leq T}{\sup}\ll|P^N_r A^N\ll(\bar P_t - \bar P_s\rr)g(x)\rr|\underset{s\rightarrow t}{\longrightarrow}0.$$
\end{enumerate}

Then, for all~$t\geq 0$,~$g\in\A$,
\begin{equation}\label{trotter}
\ll(\bar P_t g - P^N_t g\rr)(x) = \int_0^t P^N_{t-s}\ll(\bar A - A^N\rr)\bar P_s g(x)ds.
\end{equation}
\end{prop}

The proof of Proposition~\ref{prop:trotter} requires the following result. Its classical, so its proof is omitted (see e.g. Remark~A.2 and Proposition~A.3 of \cite{erny_mean_2022}).
\begin{lem}\label{derivsemi}
Let $(X_t)_t$ be some $E$-valued Markov process with semigroup~$(P_t)_t$ and generator~$A$ (in the sense of Definition~\ref{defi:markov}). For any~$g\in\D_\G(X)$ and~$x\in E$,  if the function~$t\mapsto P_t Ag(x)$ is continuous on~$\r_+$, then, for all~$t\geq 0$,
$$\partial_t \ll(P_tg(x)\rr) = P_t Ag(x).$$
In addition, if for all~$t\geq 0,$ $P_tg\in \D_\G(X)$, then
$$\partial_t \ll(P_tg(x)\rr) = AP_tg(x) = P_tAg(x).$$
\end{lem}

\begin{proof}[Proof of Proposition~\ref{prop:trotter}]
Let us fix in all the proof some $g\in \A$, $t>0$ and define
$$\Phi:s\in[0,t]\longmapsto P^N_{t-s}\bar P_s G(x).$$

By Lemma~\ref{derivsemi} and conditions~1-3 of Proposition~\ref{prop:trotter}, the function~$\Phi$ is differentiable with, for all~$s\in [0,t]$,
\begin{align*}
\Phi'(s) =& -\partial_u \ll(P^N_u \bar P_s G(x)\rr)_{|u=t-s} + \partial_v\ll(P^N_{t-s} \bar P_v G(x)\rr)_{|v=s}\\
=& - P^N_{t-s} A^N \bar P_sG(x) + P^N_{t-s}\bar P_s \bar A G(x)\\
=& P^N_{t-s}\ll(\bar A - A^N\rr)\bar P_s G(x).
\end{align*}

The only property required to prove the result of Proposition~\ref{prop:trotter} is that~$\Phi'$ is continuous on~$[0,t]$. Indeed, it would imply (by the second fundamental theorem of calculus),
\begin{equation}\label{phitphi0}
\Phi(t) - \Phi(0) = \int_0^t \Phi'(s)ds,
\end{equation}
which is exactly~\eqref{trotter}.

To prove the continuity of~$\Phi'$, let us fix, for the remainder of the proof, some $s\in[0,t]$ and some sequence $(s_k)_k$ converging to~$s$. Then
\begin{align}
\ll|\Phi'(s) - \Phi'(s_k)\rr|\leq& \ll|P^N_{t-s}\ll(\bar A - A^N\rr)\bar P_s G(x) - P^N_{t-s_k}\ll(\bar A - A^N\rr)\bar P_s G(x)\rr|\nonumber\\
&+\ll|P^N_{t-s_k}\ll(\bar A - A^N\rr)\bar P_s G(x) - P^N_{t-s_k}\ll(\bar A - A^N\rr)\bar P_{s_k} G(x)\rr|\nonumber\\
&=: A_k + B_k.\label{AkBk}
\end{align}

Since
$$A_k\leq \ll|\ll(P^N_{t-s} - P^N_{t-s_k}\rr) \bar A \bar P_sg(x)\rr| + \ll|\ll(P^N_{t-s} - P^N_{t-s_k}\rr) A^N \bar P_sg(x)\rr|,$$
we know, by conditions~1-3 that $A_k$ vanishes as $k$ goes to infinity.

On the other hand,
$$B_k\leq \ll|P^N_{t-s_k}\bar A\ll(\bar P_s - \bar P_{s_k}\rr)g(x)\rr| + \ll|P^N_{t-s_k}A^N\ll(\bar P_s - \bar P_{s_k}\rr)g(x)\rr|,$$
also vanishes as~$k$ goes to infinity by condition~4.
Consequently~\eqref{AkBk} implies that~$\Phi'$ is continuous on~$[0,t]$, implying that~\eqref{phitphi0} holds true, which proves the proposition.
\end{proof}

\subsection{Proof of Theorem~\ref{barXmarkov}}\label{conditionallaw}

In this section, we prove Theorem~\ref{barXmarkov}: the process~$(\bar\mu_t)_{t\geq 0}$, that is defined as the conditional law of the solution~$(\bar X_t)_{t\geq 0}$ of~\eqref{defbarX}, is a Markov process, and the expression of its generator is given by~\eqref{limitgenerator}.

This proof uses the same ideas as in \cite{guo_itos_2023} and \cite{cox_controlled_2021}: using Ito's formula for $\r^n$-valued processes, and then integrating the result. In both references, the authors extend the expressions of their measure-valued Ito's formulas to sufficiently smooth test-functions. This is done by approximating smooth measure-variable functions with cylinder functions. It is not clear that the approximation scheme of \cite{cox_controlled_2021} can be adapted to our framework, whereas the one of  \cite{guo_itos_2023} seems to be adaptable. This paper being sufficiently long and technical, we prefer to omit the use of approximation schemes, and restrict the result of Theorem~\ref{barXmarkov} (and Theorem~\ref{XNmarkov}) for measure-variable polynomial test-functions.

Let
$$G : m\in\P_1(\r)\longmapsto \int_{\r^n} \phi(x)dm^{\otimes n}(x),$$
with $n\in\n^*$ and $\phi\in C^2_b(\r^n).$

For a fixed~$m\in\P_2(\r)$, let us introduce, for all~$1\leq k\leq n$, the process~$\bar X^k$ solution to
\begin{align}
\bar X^k_t = & \bar X^k_0 + \int_0^t b(\bar\mu_s,\bar X^k_s)ds  + \int_0^t\sigma(\bar\mu_s,\bar X^k_s)dB^k_s+ \int_0^t \varsigma(\bar\mu_s,\bar X^k_s)dW_s\label{defbarXk}\\
&+\int_{[0,t]\times\r_+}h(\bar \mu_{s-},\bar X^k_{s-})\uno{z\leq f(\bar \mu_{s-},\bar X^k_{s-})}d\pi^k(s,z),\nonumber
\end{align}
with $\bar X^1_0,...,\bar X^n_0$ i.i.d.~$m$-distributed, $B^1,...,B^n,W$ independent Brownian motions, and $\pi^1,...,\pi^n$ independent Poisson measures on~$\r_+^2$ with Lebesque intensity, such that all these objects are mutually independent. The random measure $\bar \mu_t$ is the conditional law of any of the process given the filtration of~$W$ up to time~$t$: $\bar\mu_t = \mathcal{L}(\bar X^1_t | \W_t).$ The interest of this construction is to guarantee the following property: for all~$t\geq 0,$
\begin{equation}\label{barXiid}
\textrm{conditionally on }\bar\mu_t,~~\bar X^k_t~(1\leq k\leq n)\textrm{ are i.i.d. }\bar\mu_t\textrm{-distributed.}
\end{equation}

Let us denote ${\bf \bar X}_t = (\bar X^1_t,...,\bar X^n_t)$, and, for $1\leq k\leq n$, $\bar X^{k,c}$ the continuous part of the semimartingale~$\bar X^k$. By Ito's formula, for all~$t\geq r\geq 0,$

\begin{align*}
\phi({\bf \bar X}_t) = & \phi({\bf \bar X}_r) + \sum_{k=1}^n \int_r^t \partial_k \phi({\bf \bar X}_s)d\bar X^{k,c}_s + \frac12\sum_{k,l=1}^n \int_r^t \partial^2_{kl}\phi({\bf \bar X}_s)d\ll\langle \bar X^{k,c},\bar X^{l,c}\rr\rangle_s\\
&+ \sum_{r< s\leq t}\ll(\phi({\bf \bar X}_t) - \phi({\bf \bar X}_{t-})\rr)\\
=& \phi({\bf \bar X}_r) + \sum_{k=1}^n \int_r^t \partial_k \phi({\bf \bar X}_s)b(\bar\mu_s,\bar X^k_s)ds + \frac12\sum_{k=1}^n \int_r^t \sigma(\bar\mu_s,\bar X^k_s)^2\partial^2_{kk} \phi({\bf \bar X}_s)ds\\
&+ \frac12\sum_{k,l=1}^n \int_r^t \varsigma(\bar\mu_s,\bar X^k_s)\varsigma(\bar\mu_s,\bar X^l_s)\partial^2_{kl} \phi({\bf \bar X}_s)ds\\
&+ \sum_{k=1}^n\int_r^t \partial_k \phi({\bf \bar X}_s)\sigma(\bar\mu_s,\bar X^k_s)dB^k_s+ \sum_{k=1}^n\int_r^t \partial_k \phi({\bf \bar X}_s)\varsigma(\bar\mu_s,\bar X^k_s)dW_s\\
&+\sum_{k=1}^n \int_{]r,t]\times\r_+} \uno{z\leq f(\bar\mu_{s-},\bar X^k_{s-})}\\
&\hspace*{4cm}\ll[\phi\ll({\bf \bar X}_{s-} + h(\bar\mu_{s-},\bar X^k_{s-})\cdot e_k\rr) - \phi\ll({\bf \bar X}_{s-}\rr)\rr]d\pi^k(s,z),
\end{align*}
with $e_k = (\uno{k=l})_{1\leq l\leq n}\in \r^n$. Then, thanks to the calculation above, the property~\eqref{barXiid}, and Lemmas~\ref{lemintw} and~\ref{lemintb},
\begin{align*}
G(\bar\mu_t)=& G(\bar\mu_r)+ \sum_{k=1}^n \int_r^t\int_{\r^n} \partial_k\phi(x)b(\bar\mu_s,x_k)d\bar\mu_s^{\otimes n}(x)ds\\
&+ \frac12 \sum_{k=1}^n \int_r^t\int_{\r^n} \sigma(\bar\mu_s,x_k)^2\partial^2_{kk}(x)d\bar\mu_s^{\otimes n}(x)ds\\
&+ \frac12\sum_{k,l=1}^n \int_r^t \int_{\r^n} \varsigma(\bar\mu_s,x_k)\varsigma(\bar\mu_s,x_l)\partial^2_{kl}\phi(x)d\bar\mu_s^{\otimes n}(x)ds\\
&+ \sum_{k=1}^n \int_r^t\int_{\r^n} \partial_k \phi(x)\varsigma(\bar\mu_s,x_k)d\bar\mu_s^{\otimes n}(x)dW_s\\
&+ \sum_{k=1}^n \int_r^t \int_{\r^n} f(\bar\mu_s,x_k)\ll[\phi\ll(x + h(\bar\mu_s,x_k)\cdot e_k\rr) - \phi(x)\rr]d\bar\mu_s^{\otimes n}(x)ds\\
&+\sum_{k=1}^n \int_{]r,t]\times\r_+}\int_{\r^n} \uno{z\leq f(\bar\mu_{s-},x_k)}\\
&\hspace*{4cm}\ll[\phi\ll(x + h(\bar\mu_{s-},x_k)\cdot e_k\rr) - \phi\ll(x\rr)\rr]d\bar\mu_{s-}^{\otimes n}(x)d\tilde\pi^k(s,z),
\end{align*}
with $\tilde\pi^k$ the compensated version of~$\pi^k$:
$$d\tilde\pi^k(t,z) = d\pi^k(t,z) - dt\cdot dz.$$

Then, using the computation above and Proposition~\ref{markovcritere}, we know that the process~$(\bar\mu_t)_t$ is a Markov process. In addition, by~\eqref{controlsde} and the fact that~$\varsigma$ and~$f$ are sublinear, we know that the stochastic integrals w.r.t.~$W$ and~$\tilde\pi^k$ above are real martingales. So the function~$G$ belongs to~$\D_\G(\bar \mu)$, and
\begin{align}
\bar A G(m) = & \sum_{k=1}^n \int_{\r^n} \partial_k\phi(x) b(m,x_k)dm^{\otimes n}(x) + \frac12\sum_{k=1}^n\int_{\r^n}\sigma(m,x_k)^2\partial^2_{kk}\phi(x)dm^{\otimes n}(x)\label{barA}\\
&+\frac12\sum_{k,l=1}^n \int_{\r^n} \varsigma(m,x_k)\varsigma(m,x_l)\partial^2_{kl}\phi(x)dm^{\otimes n}(x)\nonumber\\
&+\sum_{k=1}^n \int_{\r^n} f(m,x_k)\ll[\phi\ll(x + h(m,x_k)\cdot e_k\rr) - \phi(x)\rr]dm^{\otimes n}(x).\nonumber
\end{align}


Then, by Proposition~\ref{regupoly}, for all~$x,y\in\r,m\in\P_1(\r),$
\begin{align}
\delta G(m,x) =& \sum_{k=1}^n \int_{\r^{n-1}} \phi\ll(y\backslash_k x\rr)dm^{\otimes n-1}(y\backslash_k) + C_0(m),\label{calculdelta}\\
\delta^2 G(m,x,y)=& \sum_{\substack{k,l=1\\k\neq l}}^n \int_{\r^{n-2}} \phi(z\backslash_{(k,l)}(x,y))dm^{\otimes n-2}(z\backslash_{(k,l)}) + C_1(m,x) + C_2(m,y),\label{calculdelta2}
\end{align}
with~$C_2(m,y)$ independent of~$x$, $C_1(m,x)$ independent of~$y$, and $C_0(m)$ independent of both~$x$ and~$y$.

Consequently,
\begin{align*}
\partial_x \delta G(m,x) =& \sum_{k=1}^n \int_{\r^{n-1}}\partial_k\phi\ll(z\backslash_k x\rr)dm^{\otimes n-1}(z\backslash_k),\\
\int_\r b(m,x)\partial_x \delta G(m,x)dm(x) =& \sum_{k=1}^n \int_\r \int_{\r^{n-1}}b(m,x)\partial_k\phi\ll(y\backslash_k x\rr)dm^{\otimes n-1}(y\backslash_k)dm(x)\\
=&\sum_{k=1}^n \int_{\r^n}b(m,y_k)\partial_k \phi(y)dm^{\otimes n}(y).
\end{align*}
Using the same trick on the other terms of~\eqref{barA}, we have
\begin{align*}
\int_\r \sigma(m,x)^2\partial^2_{xx} \delta G(m,x) dm(x) =& \sum_{k=1}^n \int_{\r^n} \sigma(m,y_k)^2\partial^2_{kk}\phi (y)dm^{\otimes n}(y),\\
\int_\r \varsigma(m,x)^2\partial^2_{xx} \delta G(m,x) dm(x) =& \sum_{k=1}^n \int_{\r^n} \varsigma(m,y_k)^2\partial^2_{kk}\phi (y)dm^{\otimes n}(y),\\
\int_{\r^2} \varsigma(m,x)\varsigma(m,y)\partial^2_{xy}\delta^2 G(m,x,y) dm^{\otimes 2}(x,y)=& \sum_{\substack{k,l=1\\ k\neq l}}^n \int_{\r^n}\varsigma(m,z_k)\varsigma(m,z_l)\partial^2_{kl}\phi(z)dm^{\otimes n}(z).
\end{align*}

As a consequence, it is possible to rewrite~\eqref{barA} exactly as~\eqref{limitgenerator}.

\subsection{Regularity of semigroups of conditional laws of diffusions}

In this section, we only consider a particular case of the SDE~\eqref{defbarX} without the jump term:
\begin{equation}\label{barXsanssaut}
d\bar X_t= b(\bar\mu_t,\bar X_t)dt + \sigma(\bar\mu_t,\bar X_t)dB_t + \varsigma(\bar\mu_t,\bar X_t)dW_t,
\end{equation}
where $B,W$ are still independent Brownian motions, and $\bar\mu_t = \mathcal{L}(\bar X_t|\W_t)$. In addition, for any~$x\in\r$, let $(\bar X^{(x)}_t)_t$ be the (only) strong solution of~\eqref{barXsanssaut} with initial condition~$\bar X^{(x)}_0=x.$

\begin{lem}\label{reguflot}
Let $k\in\n^*$. Assume that the functions~$b,\sigma,\varsigma$ admit $k$-th order mixed derivatives such that the mixed derivatives of orders from one to $k$ are bounded. Then, there exists some~$T_k>0$ such that, for all~$0\leq t\leq T_k$, the function
$$x\in\r\longmapsto \bar X_t^{(x)}$$
belongs to~$C^{k-1}_b(\r)$.

In addition, for all~$T>0$,~$p\in\n^*$, and~$1\leq j\leq k-1$,
$$\underset{x\in\r}{\sup}~\esp{\underset{t\leq T}{\sup}~\ll|\partial^j_{x^j}\bar X_t^{(x)}\rr|^p} <\infty.$$
\end{lem}

The proof of the previous lemma uses classical technics (see e.g. the proof of Lemma~4.17 of \cite{chassagneux_probabilistic_2022}), but due to the particularity of our framework, we prefer to write an explicit proof at Appendix~\ref{proofreguflot}. Indeed, since we work on conditional McKean-Vlasov equations, it is not clear that the proofs of the literature can directly be applied.

\begin{prop}\label{regusemigroup}
Let $\bar P$ be the semigroup of the process~$\bar X$ defined at~\eqref{barXsanssaut}. Under the assumption of Lemma~\ref{reguflot}, for any~$G\in\poly^4$ and~$t\geq 0$,
\begin{itemize}
\item[$(a)$] the function
$$m\in\P_1(\r)\longmapsto \bar P_tG(m)$$
belongs to~$\poly^4$,
\item[$(b)$] there exists~$C_t>0$ independent of~$G$ such that, for all~$0\leq k\leq 4$,
$$\underset{s\leq t}{\sup}~||\bar P_s G||_k\leq C_t ||G||_k,$$
where $||\bullet||_k$ is introduced in Definition~\ref{bornederivn},
\item[$(c)$] for all~$m\in\P_1(\r)$ and~$x,y\in\r$, the following functions are continuous
\begin{align*}
t\in\r_+&\longmapsto \bar P_tG(m),\\
t\in\r_+&\longmapsto \delta\ll(\bar P_tG\rr)(m,x),\\
t\in\r_+&\longmapsto \partial_x \delta\ll(\bar P_tG\rr)(m,x),\\
t\in\r_+&\longmapsto \partial^2_{xx}\delta\ll(\bar P_tG\rr)(m,x),\\
t\in\r_+&\longmapsto \partial^2_{xy}\delta^2\ll(\bar P_tG\rr)(m,x,y).
\end{align*}
\end{itemize}
\end{prop}

\begin{proof}
Let $G\in\poly^4$ satisfying for all~$m\in\P_1(\r)$,
$$G(m) = \int_{\r^n}\phi(x) dm^{\otimes n}(x),$$
with $n\in\n^*$ and $\phi\in C^4_b(\r^n)$.

Let us prove Item~$(a)$. For all~$t\geq 0,$
$$G(\bar\mu_t)= \espc{\phi(\bar X^1_t,...,\bar X^n_t)}{\W_t},$$
where the processes~$\bar X^k$ ($1\leq k\leq n$) are solutions to~\eqref{barXsanssaut} w.r.t. the same~$W$, but independent~$B^k$ and $\bar X^k_0$ (similarly as in~\eqref{defbarXk}). Since the variables~$\bar X^1_0,...,\bar X^n_0$ are i.i.d. $\bar\mu_0$-distributed, we have, for any~$m\in\P_1(\r)$,
\begin{equation}\label{expressionPtG}
\bar P_tG(m)=\espcc{m}{G(\bar\mu_t)} = \int_{\r^n} \esp{\phi\ll(\bar X^{1,(x_1)}_t,...,\bar X^{n,(x_n)}_t\rr)}dm^{\otimes n}(x) = \int_{\r^n}\psi_t(x)dm^{\otimes n}(x),
\end{equation}
with
$$\psi_t : x\in\r^n\longmapsto\esp{\phi\ll(\bar X^{1,(x_1)}_t,...,\bar X^{n,(x_n)}_t\rr)},$$

and $\bar X^{k,(x_k)}$ the process $\bar X^k$ starting from initial condition~$\bar X^{k,(x_k)}_0=x_k$ ($1\leq k\leq n$). Then, by Lemma~\ref{reguflot} and by differentiation under the integral sign (using uniform integrability condition, like, for example, in Lemma~6.1 of \cite{eldredge_strong_2018}, which relies on Vitali's theorem rather than the dominated convergence theorem), there exists some~$T>0$ such that, for all~$t\in [0,T],$ the function $\psi_t$ belongs to~$C^4_b(\r^n)$. This proves that, for all~$0\leq t\leq T$, $\bar P_tG$ belongs to~$\poly^4$.

Then, since the process~$(\bar\mu_t)_t$ is Markov, we have: for all~$t\in[T,2T],$ for all~$m\in\P_1(\r)$,
$$\bar P_t G(m) = \bar P_{t-T}\ll(\bar P_T G\rr)(m).$$
Whence, since $\bar P_T G\in\poly^4$ and $t-T\in[0,T]$, the above equality implies that, for all~$T\leq t\leq 2T$, the function $\bar P_tG$ also belongs to~$\poly^4$. Iterating this argument concludes the proof.

The proof of Item~$(b)$ follows from the fact that we have the explicit expression of $\bar P_tG$ at~\eqref{expressionPtG} and the control given at Lemma~\ref{reguflot} (recalling Proposition~\ref{regupoly}). The proof of Item~$(c)$ relies on the same arguments as for Item~$(b)$ with the additional use of the Dominated Convergence Theorem.
\end{proof}

\subsection{Proof of Theorem~\ref{XNmarkov}}\label{empiricalmeasure}


In this section, we prove that the empirical measure~$(\mu^N_t)_{t\geq 0}$ of the particle system defined at~\eqref{XNk} is a Markov process, with an explicit expression for its generator.

The fact that $\mu^N$ is a Markov process follows from Proposition~2.3.3 of \cite{dawson_measure-valued_1993}. It can also be proved as in Theorem~\ref{barXmarkov} using Proposition~\ref{markovcritere} and the following computation that are in any case required to prove the remains of the statement of the theorem.

So, as in the proof of Theorem~\ref{barXmarkov}, let $G\in\poly^2$ such that: for all~$m\in\P_1(\r)$,
$$G(m) = \int_{\r^n} \phi(x)dm^{\otimes n}(x),$$
with $n\in\n^*$ and $\phi\in C^2_b(\r^n).$ 

Then,
$$G\ll(\mu^N_t\rr) = \frac{1}{N^n}\sum_{k_1,...,k_n=1}^N \phi\ll(X^{N,k_1}_t,...,X^{N,k_n}_t\rr).$$

For a given tuple $\k = (k_1,...,k_n)\in \llbracket 1,N\rrbracket^n$, let us denote
$${\bf X}^{N,\k}_t = \ll(X^{N,k_1}_t,...,X^{N,k_n}_t\rr),$$
and $X^{N,k_i,c}$ the continuous part of the semimartingale~$X^{N,k_i}$.

By Ito's formula, for all~$t\geq r\geq 0$,
\begin{align*}
\phi\ll(\X^{N,\k}_t\rr) =& \phi\ll(\X^{N,\k}_r\rr) + \sum_{i=1}^n \int_r^t \partial_i \phi\ll(\X^{N,\k}_s\rr)dX^{N,k_i,c}_s \\
&+ \frac12\sum_{i,j=1}^n \int_r^t \partial^2_{ij}\phi\ll(\X^{N,\k}_s\rr)d\ll\langle X^{N,k_i,c},X^{N,k_j,c}\rr\rangle_s\\
&+\sum_{r<s\leq t}\ll(\phi\ll(\X^{N,\k}_s\rr)-\phi\ll(\X^{N,\k}_{s-}\rr)\rr)\\
=& \phi\ll(\X^{N,\k}_r\rr) + \sum_{i=1}^n \int_r^t \partial_i \phi\ll(\X^{N,\k}_s\rr)b\ll(\mu^{N}_s,X^{N,k_i}_s\rr)ds \\
&+ \sum_{i=1}^n \int_r^t \partial_i \phi\ll(\X^{N,\k}_s\rr)\sigma\ll(\mu^{N}_s,X^{N,k_i}_s\rr)dB^{k_i}_s\\
&+ \sum_{i=1}^n \int_r^t \partial_i \phi\ll(\X^{N,\k}_s\rr)\varsigma\ll(\mu^{N}_s,X^{N,k_i}_s\rr)dW_s\\
&+\frac12\sum_{i,j=1}^n \partial^2_{ij}\phi\ll(\X^{N,\k}_s\rr) \sigma\ll(\mu^{N}_s,X^{N,k_i}_s\rr)\sigma\ll(\mu^{N}_s,X^{N,k_j}_s\rr)\uno{k_i = k_j}ds\\
&+\frac12\sum_{i,j=1}^n \partial^2_{ij}\phi\ll(\X^{N,\k}_s\rr) \varsigma\ll(\mu^{N}_s,X^{N,k_i}_s\rr)\varsigma\ll(\mu^{N}_s,X^{N,k_j}_s\rr)ds\\
&+\sum_{l=1}^N \int_{]r,t]\times\r_+\times\r} \uno{z\leq f\ll(\mu^N_{s-},X^{N,l}_{s-}\rr)}\\
&\hspace*{2cm}\ll[\phi\ll(\X^{N,\k}_{s-} + h\ll(\mu^N_{s-},X^{N,l}_{s-},u\rr)\cdot {\bf 1}\rr) - \phi\ll(\X^{N,\k}_{s-}\rr)\rr]d\pi^l(s,z,u),
\end{align*}
with ${\bf 1} = (1,1,...,1)\in\r^n$. Consequently,
\begin{align}
G(\mu^N_t)=& \frac{1}{N^n}\sum_{k_1,...,k_n=1}^N \phi\ll(\X^{N,\k}_t\rr)\nonumber\\
=& G(\mu^N_r) + M^{N,r}_t+ \sum_{i=1}^n \int_r^t \int_{\r^n}\partial_i \phi(x) b\ll(\mu^N_s,x_i\rr)d\ll(\mu^N_s\rr)^{\otimes n}(x)ds\nonumber\\
&+\frac{1}{N^n}\sum_{\k\in\llbracket 1,N\rrbracket^n}\frac12\sum_{i,j=1}^n \label{ligneB}\\ 
&\hspace*{3cm}\int_r^t\partial^2_{ij}\phi\ll(\X^{N,\k}_s\rr) \sigma\ll(\mu^{N}_s,X^{N,k_i}_s\rr)\sigma\ll(\mu^{N}_s,X^{N,k_j}_s\rr)\uno{k_i = k_j}ds\nonumber\\
&+ \frac12\sum_{i,j=1}^n \int_r^t \int_{\r^n} \partial^2_{ij}\phi(x) \varsigma(\mu^N_s,x_i)\varsigma(\mu^N_s,x_j)d\ll(\mu^N_s\rr)^{\otimes n}(x)ds\nonumber\\
&+ \sum_{l=1}^N\int_r^t\int_\r \int_{\r^n} f\ll(\mu^N_s, X^{N,l}_s\rr)\label{lignesaut}\\
&\hspace*{2.6cm}\ll[\phi\ll(x + h\ll(\mu^N_{s},X^{N,l}_{s},u\rr)\cdot {\bf 1}\rr) - \phi\ll(x\rr)\rr]d\ll(\mu^N_s\rr)^{\otimes n}(x)d\nu(u)ds,\nonumber
\end{align}
with $(M^{N,r}_t)_{t\geq r}$ some martingale (the fact that $M^{N,r}$ is not only a local martingale holds true thanks to controls like~\eqref{controlsde} and the assumption that~$f,\sigma,\varsigma$ are sublinear).

Then, noticing that, for any~$x_1,...,x_N\in\r$ and~$\lambda\in\r$,
$$\Sh\ll(\frac1N\sum_{k=1}^N \delta_{x_k}, \lambda\rr) = \frac1N\sum_{k=1}^N \delta_{x_k + \lambda},$$
it is possible to rewrite the expression at~\eqref{lignesaut} as
\begin{multline*}
N\int_r^t \int_\r\int_{\r^n}\int_\r f\ll(\mu^N_s,y\rr)\ll[\phi\ll(x+h(\mu^N_s,y,u)\cdot {\bf 1}\rr) - \phi(x)\rr]d\mu^N_s(y)d\ll(\mu^N_s\rr)^{\otimes n}(x)d\nu(u)ds\\
= N\int_r^t \int_\r\int_\r f\ll(\mu^N_s,y\rr)\ll[G\ll(\Sh\ll(\mu^N_s,h(\mu^N_s,y,u)\rr)\rr) - G\ll(\mu^N_s\rr)\rr]d\mu^N_s(y)d\nu(u)ds,
\end{multline*}

Besides, separating the terms $i=j$ and $i\neq j$ in the double-sum of the expression at~\eqref{ligneB}, this expression can be rewritten as

\begin{align*}
&\frac12\int_r^t\sum_{i=1}^n \frac1{N^n} \sum_{\k\in\llbracket  1,N\rrbracket^{n}} \partial^2_{ii}\phi(\X^{N,\k}_s)\sigma\ll(\mu^N_s,X^{N,k_i}_s\rr)^2 ds\\
&+\frac12\int_r^t\sum_{\substack{i,j=1\\i\neq j}}^n \frac1N\cdot \frac1{N^{n-1}} \sum_{\k\backslash_j \in\llbracket  1,N\rrbracket^{n-1}} \partial^2_{ij}\phi(\X^{N,(\k\backslash_j k_i)}_s)\sigma\ll(\mu^N_s,X^{N,k_i}_s\rr)^2 ds\\
&~~~~= \frac1{2}\int_r^t \sum_{i=1}^n \int_{\r^{n}}\partial^2_{ii}\phi(x)\sigma\ll(\mu^N_s,x_i\rr)^2d\ll(\mu^N_s\rr)^{\otimes n}(x)ds\\
&~~~~~~~~+\frac1{2N}\int_r^t \sum_{\substack{i,j=1\\i\neq j}}^n \int_{\r^{n-1}}\partial^2_{ij}\phi(x\backslash_j x_i)\sigma\ll(\mu^N_s,x_i\rr)^2d\ll(\mu^N_s\rr)^{\otimes n-1}(x\backslash_j)ds\\
&~~~~= \frac12\int_r^t \int_\r \sigma\ll(\mu^N_s,x\rr)^2\partial^2_{xx}\delta G\ll(\mu^N_s,x\rr)d\mu^N_s(x)ds\\
&~~~~~~~~+\frac{1}{2N}\int_r^t \int_\r  \sigma\ll(\mu^N_s,x\rr)^2\ll(\partial^2_{y_1y_2}\delta^2 G\ll(\mu^N_s,y_1,y_2\rr)\rr)_{|y_1=y_2=x}d\mu^N_s(x)ds,
\end{align*}
where the last equality uses~\eqref{calculdelta} and~\eqref{calculdelta2}, and is mostly a notation matter.

In particular, doing the same computation trick as the one done at the end of the proof of Theorem~\ref{barXmarkov}, it is possible to write the dynamics of $G(\mu^N_t)$  as follows
\begin{align*}
G(\mu^N_t)=& G(\mu^N_r) + M^{N,r}_t +  \int_r^t \int_{\r}b\ll(\mu^N_s,x\rr)\partial_x \delta G(m,x)d\mu^N_s(x)ds\\
&+\frac12\int_r^t \int_\r \sigma\ll(\mu^N_s,x\rr)^2\partial^2_{xx}\delta G\ll(\mu^N_s,x\rr)d\mu^N_s(x)ds\\
&+\frac{1}{2N}\int_r^t \int_\r  \sigma\ll(\mu^N_s,x\rr)^2\ll(\partial^2_{y_1y_2}\delta^2 G\ll(\mu^N_s,y_1,y_2\rr)\rr)_{|y_1=y_2=x}d\ll(\mu^N_s(x)\rr)^{\otimes 2}(x)ds\\
&+ \frac12 \int_r^t \int_{\r^2} \varsigma(\mu^N_s,x)\varsigma(\mu^N_s,y)\partial^2_{xy}\delta^2 G(\mu^N_s,x,y)d\ll(\mu^N_s\rr)^{\otimes 2}(x,y)ds\\
&+N\int_r^t \int_\r\int_\r f\ll(\mu^N_s,y\rr)\ll[G\ll(\Sh\ll(\mu^N_s,h(\mu^N_s,y,u)\rr)\rr) - G\ll(\mu^N_s\rr)\rr]d\mu^N_s(y)d\nu(u)ds,
\end{align*}
with $(M^{N,r}_t)_t$ a martingale that depends only on the process~$\mu^N$, on the Brownian motions~$W,B^k$ ($1\leq k\leq N$) and on the Poisson measures~$\pi^l$ ($1\leq l\leq N$) such that it adapted to the union of the filtrations of these objects. The equality above being true for all~$G\in\poly^2$, Proposition~\ref{markovcritere} entails that the measure-valued process~$(\mu^N_t)_t$ is a Markov process (which was already guaranteed by Proposition~2.3.3 of \cite{dawson_measure-valued_1993}), but also that the domain of the generator of~$\mu^N$ contains the set~$\poly^2$, and that the expression of the generator on any~$G\in\poly^2$ is the one given in the statement of Theorem~\ref{XNmarkov}.

\section{Proof of the results of Section~\ref{resultconvergence}}\label{proofmain}

\subsection{Proof of Theorem~\ref{mainresultdiffu}}

It is sufficient to prove the result for $n=1$. Indeed, the general case can be deduced inductively using the following Markovian properties of the processes: if $(\mu_t)_t$ is an $E$-valued homogeneous Markov process with semigroup~$(P_t)_t$, then for~$t_1<t_2$, and $G_1,G_2$ ``smooth enough'',
$$\esp{G_1(\mu_{t_1})G_2(\mu_{t_2})} = \esp{G_1(\mu_{t_1}) P_{t_2-t_1}G_2(\mu_{t_1})}.$$

So we only prove the result for~$n=1$. Let us recall that, for each~$N\in\n^*$, the processes~$X^{N,k}$ ($1\leq k\leq N$) are defined as the (strong) solutions of~\eqref{diffuXN}, and~$\bar X$ as the (strong) solution of~\eqref{diffubarX}. And, by definition, for all~$t\geq 0$,
$$\mu^N_t = \frac1N\sum_{k=1}^N \delta_{X^{N,k}_t}~~\textrm{ and }~~\bar\mu_t = \mathcal{L}\ll(\bar X_t|\W_t\rr).$$

Let us denote
$$\mathcal{E}_N = \ll\{m_w = \frac1N\sum_{k=1}^N \delta_{w_k}~:~w_1,...,w_N\in\r\rr\}\subseteq \P_\infty(\r).$$

It is classical that for any~$t\geq 0$, there exists~$C_t>0$ such that for all~$p\in\n^*$,~$m\in\P_p(\r)$ and~$w_1,...,w_N\in\r$
\begin{align}
&\espcc{m}{\underset{s\leq t}{\sup}~\ll|\bar X_s\rr|^p}\leq C_t\ll(1 + \int_\r |x|^pdm(x)\rr),\label{controlsdecondinit1}\\
&\frac1N\sum_{k=1}^N\espcc{m_w}{\underset{s\leq t}{\sup}~\ll|X^{N,k}_s\rr|^p} \leq C_t\ll(1 + \frac1N\sum_{k=1}^N|w_k|^p\rr),\label{controlsdecondinit2}
\end{align}
where $m_w = N^{-1}\sum_{k=1}^N w_k$, $\E_m$ (resp. $\E_{m_w}$) denotes the expectation w.r.t. the probability measure under which the Markov process $\bar\mu$ (resp. $\mu^N$) has $m$ (resp. $m_w$) as initial condition.

{\it Step~1.} The first step of the proof consists in showing that, for any fixed~$N\in\n^*$, the $\P_1(\r)$-valued Markov processes~$\mu^N$ and~$\bar\mu$ satisfy the following Trotter-Kato formula: for all~$G\in\poly^4$,~$m\in \mathcal{E}_N$ and~$t\geq 0$,
$$\ll(\bar P_t - P^N_t\rr)G(m) = \int_0^t P^N_{t-s}\ll(\bar A - A^N\rr)\bar P_s G(m)ds,$$
with $P^N,\bar P$ (resp. $A^N,\bar A$) the semigroups (resp. generators) of~$\mu^N,\bar\mu$, in the sense of Definition~\ref{defi:markov}.


To prove this, it is sufficient to check that $\mu^N,\bar\mu$ satisfy the four conditions of Proposition~\ref{prop:trotter}. Condition~1 holds true for~$\A = \poly^4$ by Proposition~\ref{regusemigroup}. Besides, Condition~2 is a direct consequence of Theorems~\ref{barXmarkov} and~\ref{XNmarkov}.

Now, let us verify Condition~3 of Proposition~\ref{prop:trotter}. Let $G\in\poly^4$, by Theorem~\ref{barXmarkov}, the generator~$\bar A$ of~$\bar\mu$ satisfy: for all~$m\in\P_1(\r)$,
$$\bar A G(m) = \int_\r I(m,x)dm(x) + \int_{\r^2}J(m,x,y)dm^{\otimes 2}(x,y),$$
with
\begin{align}
I(m,x)=& b(m,x)\partial_x\delta G(m,x)  + \frac12\sigma(m,x)^2\partial^2_{xx}\delta G(m,x)\label{Imx}\\
&+ \frac12\ll(\int_\r u^2~d\nu(u)\rr)\ll(\int_\r h(m,z)^2 f(m,z)~dm(z)\rr)\partial^2_{xx}\delta G(m,x),\nonumber\\
J(m,x,y)=& \frac12\ll(\int_\r u^2~d\nu(u)\rr)\ll(\int_\r h(m,z)^2 f(m,z)~dm(z)\rr)\partial^2_{xy}\delta^2 G(m,x,y).\label{Jmxy}
\end{align}

In particular thanks to Proposition~\ref{regupoly}, we have the following control: there exists~$C>0$ such that for all~$G\in\poly^4$ and~$m\in\P_2(\r)$,
\begin{equation}\label{controlbarA}
\ll|\bar AG(m)\rr|\leq C\ll(1 + \int_\r |x|^2dm(x)\rr)||G||_2.
\end{equation}

And, for all~$t_0,t\leq T,$
\begin{align*}
\ll|\bar A G(\bar\mu_{t_0}) - \bar A G(\bar\mu_t)\rr|\leq& \ll|\int_{\r}I(\bar\mu_{t_0},x) d\ll(\bar\mu_{t_0} - \bar\mu_t\rr)(x)\rr| + \int_{\r} \ll|I(\bar\mu_{t_0},x) - I(\bar\mu_t,x)\rr|d \bar\mu_t(x)\\
&+\ll|\int_{\r^2}J(\bar\mu_{t_0},x,y) d\ll(\bar\mu_{t_0}^{\otimes 2} - \bar\mu_t^{\otimes 2}\rr)(x,y)\rr| \\
&+ \int_{\r^2} \ll|J(\bar\mu_{t_0},x,y) - J(\bar\mu_t,x,y)\rr|d \bar\mu_t^{\otimes 2}(x,y)\\
\leq& \espc{\ll|I(\bar \mu_{t_0},\bar X_{t_0}) - I(\bar \mu_{t_0},\bar X_{t})\rr|}{\W_T} \\
&+\espc{\ll|I(\bar \mu_{t_0},\bar X_{t}) - I(\bar \mu_{t},\bar X_{t})\rr|}{\W_T}\\
&+\espc{\ll|J(\bar \mu_{t_0},\bar X_{t_0}^1,\bar X_{t_0}^2) - J(\bar \mu_{t_0},\bar X_{t}^1,\bar X_t^2)\rr|}{\W_T}\\
&+\espc{\ll|J(\bar \mu_{t_0},\bar X_{t}^1,\bar X_{t}^2) - J(\bar \mu_{t},\bar X_{t}^1,\bar X_t^2)\rr|}{\W_T},
\end{align*}
with $\bar X^1,\bar X^2$ solutions to the SDE~\eqref{diffubarX} w.r.t. the same Brownian motion~$W$, but w.r.t. two independent Brownian motions~$B^1,B^2$ instead of~$B$. Let us note that, the functions~$I,J$ are locally Lipschitz continuous in the following sense: there exists~$C_G>0$ such that, for all~$m_1,m_2\in\P_2(\r),$ and~$x_1,x_2,y_1,y_2\in\r$,
\begin{multline*}
\ll|I(m_1,x_1) - I(m_2,x_2)\rr| + \ll|J(m_1,x_1,y_1) - J(m_2,x_2,y_2)\rr|\\
\leq C_G\ll(1 + |x_1|^2 + |x_2|^2 + |y_1|^2 + |y_2|^2 + \int_\r |z|^2 dm_1(z) + \int_\r |z|^2dm_2(z)\rr)\\
\ll(|x_1 - x_2| + D_{KR}(m_1,m_2)\rr).
\end{multline*}

In particular, by Cauchy-Schwarz' inequality and Jensen's inequality and~\eqref{controlsde}, denoting $\E_m$ the expectation w.r.t. the probability measure under which $\bar\mu_0 = m$,
\begin{align*}
&\ll|\bar P_{t_0}\bar AG(m) - \bar P_{t}\bar AG(m)\rr| \leq \espcc{m}{\ll|\bar A G(\bar\mu_{t_0}) - \bar A G(\bar\mu_t)\rr|}\\
&~~~~\leq C_G \espcc{m}{\ll(1 + \underset{s\leq T}{\sup}~\espc{|\bar X_s|^2}{\W_T}\rr)\espc{\ll|\bar X_{t_0} - \bar X_t\rr|}{\W_T}}\\
&~~~~\leq C_G \espcc{m}{\ll| \bar X_{t_0} - \bar X_t\rr|^2}^{1/2}.
\end{align*}

Recalling that~$\bar X$ is solution to~\eqref{diffubarX}, it is classical (using for example Burkholder-Davis-Gundy's inequality with~\eqref{controlsde}) that the RHS of the above inequality vanishes as~$t$ goes to~$t_0$. This implies that, for any~$m\in\P_2(\r)$ and~$G\in\poly^4$, the function
$$t\mapsto \bar P_t \bar AG(m)$$
is continuous.

The continuity of the function
$$t\mapsto P^N_t A^NG(m),$$
for $m\in\mathcal{E}_N$ follows from the same reasoning provided Lemma~\ref{lemshift2} below to deal with the jump term that depends on a shift operator~$\Sh$. Note that it is important to work with an initial condition belonging to~$\mathcal{E}_N$ in order to interpret $\mu^N_t$ as the empirical measure of some $N$-particle system satisfying the SDEs~\eqref{diffuXN}.
\begin{lem}\label{lemshift2}
Let $m_1,m_2\in\P_1(\r)$ and~$h_1,h_2\in\r$. Then
$$D_{KR}\ll(\Sh(m_1,h_1),\Sh(m_2,h_2)\rr)\leq D_{KR}(m_1,m_2) + |h_1 - h_2|.$$
\end{lem}

\begin{proof}
Let $\phi : \r\rightarrow\r$ be any Lipschitz continuous function with Lipschitz constant non-greater than one.
\begin{align*}
&\ll|\int_\r \phi(x)d\ll(\Sh(m_1,h_1)\rr)(x) - \int_\r \phi(x)d\ll(\Sh(m_2,h_2)\rr)(x)\rr|\\
&~~= \ll|\int_\r \phi(x+h_1)dm_1(x) - \int_\r \phi(x+h_2)dm_2(x)\rr|\\
&~~\leq \ll|\int_\r \phi(x+h_1)d(m_1-m_2)(x)\rr| + \int_\r \ll|\phi(x+h_1) - \phi(x+h_2)\rr|dm_2(x)\\
&~~\leq D_{KR}(m_1,m_2) + |h_1-h_2|.
\end{align*}
Taking the supremum over all such functions~$\phi$ proves the result.
\end{proof}

Indeed, with the same computation, we obtain
\begin{align*}
&\ll|P^N_{t_0}A^NG(m) - P^N_{t}A^N G(m)\rr| \leq \espcc{m}{\ll|A^N G(\mu^N_{t_0}) - A^N G(\mu^N_t)\rr|}\\
&~~~~\leq C_G \espcc{m}{\ll(1 + \underset{s\leq T}{\sup}\frac1N\sum_{k=1}^N\ll|X^{N,k}_s\rr|^2\rr) \frac1N\sum_{k=1}^N\ll|X^{N,k}_{t_0} - X^{N,k}_t\rr|}\\
&~~~~\leq C_G \ll(\frac1N\sum_{k=1}^N\espcc{m}{\ll| X^{N,k}_{t_0} - X^{N,k}_t\rr|^2}\rr)^{1/2}.
\end{align*}

Once again, this proves that for any~$m\in\mathcal{E}_N$ and~$G\in\poly^4$, the function
$$t\mapsto P^N_t A^NG(m)$$
is continuous. The continuity of $t\mapsto P^N_t\bar A G(m)$ is guaranteed by the same arguments. So Condition~3 of Proposition~\ref{prop:trotter} is verified.

To end the {\it Step~1} of the proof, let us now check Condition~4. Let us fix in the rest of {\it Step~1} some $G\in\poly^4$, $0\leq t<T$ and some sequence~$(s_n)_n$ converging to~$t$ such that $s_n\leq T$ for all~$n\in\n$. For~$M>0$, let us introduce
$$\K_M = \ll\{m\in\P_1(\r)~:~\textrm{Supp}(m)\subseteq [-M,M]\rr\}.$$

For any~$M>0$, $\K_M$ is a compact set of~$\P_p(\r)$, for all~$p\geq 1$ (see Remark~6.19 of \cite{villani_optimal_2009}). Let us denote, for every~$n\in\n$,
$$G_n = \bar P_t G - \bar P_{s_n} G.$$

Thanks to Proposition~\ref{regusemigroup}, we know that each $G_n$ belongs to~$\poly^4$. Hence, by Proposition~\ref{regupoly}, we can use Corollaries~\ref{corderivint1} and~\ref{corderivint2} (and Lemma~\ref{lemshift} to differentiate the jump term from the generator~$A^N$) to compute $\partial (\bar AG_n)(m,x)$ and $\partial(A^Ng_n)(m,x)$ (recalling the expressions of $\bar A$ and~$A^N$ given in Theorems~\ref{barXmarkov} and~\ref{XNmarkov}) to prove that, for all~$M>0,$
$$\underset{\substack{n\in\n\\m\in \K_M\\|x|\leq M}}{\sup}\ll|\partial_x \delta \ll(\bar A G_n\rr)(m,x)\rr| + \underset{\substack{n\in\n\\m\in \K_M\\|x|\leq M}}{\sup}\ll|\partial_x \delta \ll(A^N G_n\rr)(m,x)\rr| \leq C_G\ll(1 + M^2\rr).$$
Note that the fact that the control above is uniform in~$n$ is guaranteed by Proposition~\ref{regusemigroup}$.(b)$.

Then, by the Mean Value Theorem for measure-variable functions (i.e. Proposition~\ref{mvthm}) and since $\K_M$ ($M>0$) are convex sets, we know that the families of functions $(\bar A G_n)_n$ and $(A^N G_n)_n$ are both uniformly equicontinuous on each compact set~$\K_M$ (for any~$M>0$). In addition, let us recall that it is possible to write the generator~$\bar A$ and~$A^N$ in the following form: for any~$m\in\P_2(\r)$,
\begin{align*}
\bar A G_n(m) =& \int_{\r} I_n(m,x)dm(x) + \int_{\r^2} J_n(m,x,y)dm^{\otimes 2}(x,y)\\
A^N G_n(m) =& \int_{\r} I_n^N(m,x)dm(x) + \int_{\r^2} J_n^N(m,x,y)dm^{\otimes 2}(x,y),
\end{align*}
with $I_n,J_n,I_n^N$ and~$J_n^N$ defined similarly as $I,J$ in~\eqref{Imx} and~\eqref{Jmxy}, for the function~$G_n$ instead of~$G$.

And, by Proposition~\ref{regusemigroup}$.(c)$, for any~$x,y\in\r$ and~$m\in\P_2(\r)$, the integrands above vanish as $n$ goes to infinity, and are bounded by
$$C_G\ll(1 + x^2  + y^2 + \int|z|^2dm(z)\rr),$$
where $C_G>0$ does not depend on~$(m,x,y,n)$. Then, by the Dominated Convergence Theorem, for any~$m\in\K_M$ (for any~$M>0$), both $\bar AG_n(m)$ and~$A^N G_n(m)$ vanish as~$n$ goes to infinity. Since these two sequences are uniformly equicontinuous on each compact set~$\K_M$, this implies that, for all~$M>0$,
\begin{equation}\label{genCVU}
\underset{m\in \K_M}{\sup}~\ll|\bar AG_n(m)\rr| + \underset{m\in \K_M}{\sup}~\ll|A^NG_n(m)\rr|\underset{n\rightarrow\infty}{\longrightarrow}0.
\end{equation}

Then, let us introduce, for~$M>0$ the following event
$$D_M = \ll\{\underset{\substack{r\leq T\\1\leq k\leq N}}{\sup}~\ll|X^{N,k}_r\rr|\leq M\rr\},$$
and $D_M^c$ its complementary. We have, by Cauchy-Schwarz' inequality and Markov's inequality, for any~$m\in \mathcal{E}_N$,
\begin{align*}
\underset{r\leq T}{\sup}~\ll|P^N_r \bar A G_n(m)\rr|\leq &\espcc{m}{\underset{r\leq T}{\sup}~\ll|\bar A G_n(\mu^N_r)\rr|} \\
\leq& \espcc{m}{\underset{r\leq T}{\sup}~\ll|\bar A G_n(\mu^N_r)\rr|\un_{D_M}} +\espcc{m}{\underset{r\leq T}{\sup}~\ll|\bar A G_n(\mu^N_r)\un_{D_M^c}\rr|}\\
\leq&  \espcc{m}{\underset{r\leq T}{\sup}~\ll|\bar A G_n(\mu^N_r)\rr|\un_{D_M}}  + C_{T,G,m} M^{-1/2},
\end{align*}
where we have used the controls given at~\eqref{controlsdecondinit2},~\eqref{controlbarA} and Proposition~\ref{regusemigroup}$.(b)$.

In particular, for any~$\eps >0$ it is possible to fix some~$M_\eps>0$ (whose also depends on~$T,G,m$, but not on~$n$) such that
$$\underset{r\leq T}{\sup}~\ll|P^N_r \bar A G_n(m)\rr|\leq\espcc{m}{\underset{r\leq T}{\sup}~\ll|\bar A G_n(\mu^N_r)\rr|\un_{D_{M_\eps}}} + \eps.$$

Noticing that,
$$D_{M_\eps}\subseteq \ll\{\forall r\leq T, \mu^N_r\in \K_{M_\eps}\rr\},$$
we know, by~\eqref{genCVU}, that, almost surely,
$$\underset{r\leq T}{\sup}~\ll|\bar A G_n(\mu^N_r)\rr|\un_{D_{M_\eps}}\underset{n\rightarrow\infty}{\longrightarrow}0,$$
and, by the Dominated Convergence Theorem (and~\eqref{controlsdecondinit2},~\eqref{controlbarA} and Proposition~\ref{regusemigroup}$.(b)$), this entails that, for~$m\in\mathcal{E}_N$,
$$\underset{r\leq T}{\sup}~\ll|P^N_r \bar A G_n(m)\rr|\underset{n\rightarrow\infty}{\longrightarrow}0.$$

With exactly the same reasoning, we also have
$$\underset{r\leq T}{\sup}~\ll|P^N_r A^N G_n(m)\rr|\underset{n\rightarrow\infty}{\longrightarrow}0.$$

So Condition~4 of Proposition~\ref{prop:trotter} is satisfied.

{\it Step~2.} We have proved that for all~$G\in\poly^4$,~$m\in \mathcal{E}_N$ and~$t\geq 0$,
$$\ll(\bar P_t - P^N_t\rr)G(m) = \int_0^t P^N_{t-s}\ll(\bar A - A^N\rr)\bar P_s G(m)ds.$$

This formula allows to obtain a convergence speed for the semigroups provided a convergence speed of the generators. So, the goal of {\it Step~2} is to prove the following assertion: for $G\in\poly^3$ and~$m\in\P_4(\r)$,
\begin{equation}\label{diffgendiffu}
\ll|\ll(A^N - \bar A\rr)G(m)\rr|\leq C||G||_3 \frac{1}{\sqrt{N}}\ll(1 + \int_\r |x|^4dm(x)\rr).
\end{equation}

To this end, let us admit the following lemma, whose proof consists in using three times the Taylor-Lagrange's inequality. This lemma allows to simplify computation in our proof.
\begin{lem}\label{simplTL3}
Let $g\in C^3_b(\r^2)$. Then, for all~$x,y,\lambda\in\r$,
$$\ll|g(x+\lambda,y+\lambda) - g(x+\lambda,y) - g(x,y+\lambda) + g(x,y)  - \frac{\lambda^2}{2}\partial^2_{xy} g(x,y)\rr|\leq \frac{|\lambda|^3}{2}\sum_{|\alpha|=3}||\partial_\alpha g||_\infty.$$
\end{lem}

Let us denote $\zeta>0$ the standard deviation of the probability measure~$\nu$:
$$\zeta = \ll(\int_\r u^2d\nu(u)\rr)^{1/2.}$$

Then, by Theorems~\ref{barXmarkov} and~\ref{XNmarkov}, for all~$G\in\poly^2$ and~$m\in\P_2(\r)$,
\begin{align*}
\bar AG(m)=& \int_\r \ll[b(m,x)\partial_x \delta G(m,x) + \frac12\sigma(m,x)^2\partial^2_{xx}\delta G(m,x)\rr]dm(x)\\
&+\frac12\zeta^2\int_\r \ll(\int_\r h(m,y)^2 f(m,y)dm(y)\rr)\partial^2_{xx}\delta G(m,x)dm(x)\\
&+\frac12\zeta^2 \int_{\r^2}\ll(\int_\r h(m,z)^2 f(m,z)dm(z)\rr)\partial^2_{xy}\delta^2 G(m,x,y)dm^{\otimes 2}(x,y),\\
A^N G(m)=& \int_\r \ll[b(m,x)\partial_x \delta G(m,x) + \frac12\sigma(m,x)^2\partial^2_{xx}\delta G(m,x)\rr]dm(x)\\
&+\frac1{2N}\int_\r \sigma(m,x)^2\ll(\partial^2_{y_1y_2}\delta^2 G(m,y_1,y_2)\rr)_{\ll|y_1=y_2=x\rr.}dm(x)\\
&+N\int_\r \int_\r f(m,x)\ll[G\ll(\Sh\ll(m,\frac{u}{\sqrt{N}} h(m,x)\rr)\rr) - G(m)\rr]d\nu(u)dm(x).
\end{align*}

So
\begin{align}
\ll(A^N - \bar A\rr)G(m) = & \frac1{2N}\int_\r \sigma(m,x)^2\ll(\partial^2_{y_1y_2}\delta^2 G(m,y_1,y_2)\rr)_{\ll|y_1=y_2=x\rr.}dm(x)\label{termsimpl2}\\
&+N\int_\r f(m,z)\int_\r\ll[G\ll(\Sh\ll(m,\frac{u}{\sqrt{N}} h(m,z)\rr)\rr) - G(m) \rr.\label{termTL22}\\
&\hspace*{0.8cm}\ll.-\frac{u^2}{2N}h(m,z)^2\int_\r \partial^2_{xx} \delta G(m,x)dm(x)\rr.\nonumber\\
&\hspace*{0.8cm}\ll.- \frac{u^2}{2N}h(m,z)^2\int_{\r^2}\partial^2_{xy}\delta^2 G(m,x,y)dm^{\otimes 2}(x,y)\rr]d\nu(u)dm(z).\nonumber
\end{align}

The absolute value of the term at~\eqref{termsimpl2} is bounded by
$$C\frac1N ||G||_2 \ll(1 + \int_\r x^2dm(x)\rr).$$

For the rest of the proof, let us introduce, for~$u,z\in\r$, the signed measure
\begin{equation}\label{mNzu}
m^N_{z,u} = \Sh\ll(m,\frac{u}{\sqrt{N}} h(m,z)\rr) - m.
\end{equation}

To control the term at~\eqref{termTL22}, by using Theorem~\ref{taylorlagrange}, we need to compute the integrals of the $k$-th order derivative of~$G$ w.r.t. the power of~$m^N_{z,u}$ to~$k$, for $k\in\{1,2\}$.
\begin{align*}
\int_\r\int_\r \delta G(m,x)dm^N_{z,u}(x)d\nu(u)=& \int_\r\int_\r \ll[\delta G\ll(m,x + \frac{u}{\sqrt{N}}h(m,z)\rr)- \delta G(m,x)\rr.\\
&\hspace*{1.5cm} \ll. - \frac{u}{\sqrt{N}}h(m,z)\partial_x \delta G(m,x)\rr]d\nu(u)dm(x),
\end{align*}
where the term at the second line above has been added artificially since $\nu$ is a centered measure. In addition,
\begin{align*}
&\int_\r\int_{\r^2} \delta^2 G(m,x,y)d\ll(m^N_{z,u}\rr)^{\otimes 2}(x,y)d\nu(u)\\
&= \int_\r\int_\r\int_\r\ll[\delta^2 G\ll(m,x+\frac{u}{\sqrt{N}}h(m,z),y+\frac{u}{\sqrt{N}}h(m,z)\rr)\rr.\\
&\hspace*{3cm}\ll. - \delta^2 G\ll(m,x+\frac{u}{\sqrt{N}}h(m,z),y\rr) - \delta^2 G\ll(m,x,y\frac{u}{\sqrt{N}}h(m,z)\rr) \rr.\\
&\hspace*{3cm}\bigg.+ \delta^2 G(m,x,y)\bigg]d\nu(u)dm(x)dm(y).
\end{align*}

Whence, the absolute value of the term within the brackets at~\eqref{termTL22} is non-greater than the sum of the three following quantities:
\begin{align}
&\ll|G\ll(\Sh\ll(m,\frac{u}{\sqrt{N}} h(m,z)\rr)\rr) - G(m) -\int_\r \delta G(m,x)dm^N_{z,u}(x)dm(x)\rr.\label{fin1}\\
&\hspace*{7.9cm}\ll. -\frac12\int_{\r^2} \delta^2 G(m,x,y)d\ll(m^N_{z,u}\rr)^{\otimes 2}(x,y)\rr|,\nonumber\\
&\int_\r\ll|\delta G\ll(m,x + \frac{u}{\sqrt{N}}h(m,z)\rr)- \delta G(m,x) - \frac{u}{\sqrt{N}}h(m,z)\partial_x \delta G(m,x)\rr.\label{fin2}\\
&\hspace*{8.5cm}\ll.- \frac{u^2}{2N}h(m,z)^2\partial^2_{xx}\delta G(m,x)\rr|dm(x),\nonumber\\
&\int_{\r^2}\ll|\delta^2 G\ll(m,x+\frac{u}{\sqrt{N}}h(m,z),y+\frac{u}{\sqrt{N}}h(m,z)\rr)- \delta^2 G\ll(m,x+\frac{u}{\sqrt{N}}h(m,z),y\rr) \rr.\label{fin3}\\
&\ll. - \delta^2 G\ll(m,x,y\frac{u}{\sqrt{N}}h(m,z)\rr) + \delta^2 G(m,x,y)- \frac{u^2}{2N}h(m,z)^2\partial^2_{xy}\delta^2 G(m,x,y)\rr|dm^{\otimes 2}(x,y).\nonumber
\end{align}

Each of the three terms above is bounded by
$$C||G||_3 \frac{1}{N\sqrt{N}}\ll(1+|z|^3 + \int_\r |x|^3dm(x)\rr)\ll(1+|u|^3\rr).$$

Indeed, for~\eqref{fin1}, it is a consequence of the Taylor-Lagrange's inequality for measure-variable functions (i.e. Theorem~\ref{taylorlagrange}), and the fact that (by Lemma~\ref{lemshift})
$$D_{KR}\ll(\Sh\ll(m,\frac{u}{\sqrt{N}} h(m,z)\rr),m\rr) \leq \frac{|u|}{\sqrt{N}}|h(m,z)|,$$
for~\eqref{fin2}, it comes from the Taylor-Lagrange's inequality for real-variable functions, and the control of~\eqref{fin3} is obtained by Lemma~\ref{simplTL3}.

Finally using this last bound to control~\eqref{termTL22} is sufficient to prove~\eqref{diffgendiffu}.

{\it Step~3.} From~\eqref{diffgendiffu}, we deduce a control between the semigroups of~$\mu^N,\bar\mu$. Let~$w_1,...,w_N\in\r$, and
$$m_w = \frac1N\sum_{k=1}^N w_k.$$

Then, thanks to {\it Steps}~1 and~2,
\begin{align}
\ll|\ll(\bar P_t - P^N_t\rr)G(m_w)\rr|\leq& \int_0^t \espcc{m_w}{\ll|\ll(\bar A - A^N\rr)\bar P_sG(\mu^N_{t-s})\rr|}ds\nonumber\\
\leq& C\frac1{\sqrt{N}}\int_0^t ||\bar P_sG||_3\ll(1 + \frac1N{\sum_{k=1}^N \espcc{m_w}{|X^{N,k}_{t-s}|^4}}\rr)ds\nonumber\\
\leq& C_t\frac1{\sqrt{N}} ||G||_3 \ll( 1 +\frac1N\sum_{k=1}^N |w_k|^4\rr),\label{diffsemigroup}
\end{align}
where we have used Proposition~\ref{regusemigroup}.

{\it Step~4.} Then end of the proof is now quite classical. For $G\in\poly^4$,
\begin{align*}
\ll|\esp{G(\bar\mu_t)}-\esp{G(\mu^N_t)}\rr|\leq & \ll|\int_{\P_1(\r)} \bar P_t G(m)\,d\ll(\delta_{\bar\mu_0} - \mathcal{L}\ll(\mu^N_0\rr)\rr)(m)\rr|\\
&+\int_{\P_1(\r)}\ll|\ll(\bar P_t - P^N_t\rr) G(m)\rr|d\ll(\mathcal{L}\ll(\mu^N_0\rr)\rr)(m)\\
\leq& C_t ||G||_1 \mathcal{D}_{KR}\ll(\delta_{\bar\mu_0},\mathcal{L}\ll(\mu^N_0\rr)\rr) + C_t ||G||_3 \frac1{\sqrt{N}},
\end{align*}
where the last line has been obtained using Proposition~\ref{regusemigroup} (and the definition of the Kantorovich-Rubinstein metric) for the first term of the sum, and~\eqref{diffsemigroup} for the second one.

Finally, by Theorem~1 of \cite{fournier_rate_2015} (or using Theorem~3.2 of \cite{bobkov_one-dimensional_2019} and the discussion thereafter),
$$\mathcal{D}_{KR}\ll(\delta_{\bar\mu_0},\mathcal{L}\ll(\mu^N_0\rr)\rr) = \esp{D_{KR}\ll(\bar \mu_0,\frac1N\sum_{k=1}^N \delta_{X^{N,k}_0}\rr)}\leq C~N^{-1/2}.$$

This ends the proof of Theorem~\ref{mainresultdiffu}.

\subsection{Proof of Corollary~\ref{cordiffu}}

The proof of Corollary~\ref{cordiffu} consists in remarking that the SDEs~\eqref{diffuYN} and~\eqref{diffubarY} are particular cases of~\eqref{diffuXN} and~\eqref{diffubarX} with: for~$m\in\P_1(\r)$ and~$x\in\r$,
\begin{align*}
&b(m,x) = \tilde b(x);~~\sigma(m,x) = \tilde \sigma(x);~~f(m,x)=\tilde f(x);~~h(m,x) = 1;\\
&\varsigma(m,x)=\ll(\int_{\r}u^2d\nu(u)\rr)^{1/2}\ll(\int_\r \tilde f(y)~dm(y)\rr)^{1/2}.
\end{align*}

Then, the result of Corollary~\ref{cordiffu} comes from the fact that, for any~$\phi\in C^4_b(\r)$, the function
$$G : m \in\P_1(\r)\longmapsto \int_\r \phi(x)dm(x)$$
belongs to $\poly^4,$ and that
$$\esp{G(\bar\mu_t)} = G(\bar\mu_t) = \esp{\phi(\bar Y_t)}~\textrm{ and }~\esp{G(\mu^N_t)} = \esp{\frac1N\sum_{k=1}^N \phi(Y^{N,k}_t)} = \esp{\phi(Y^{N,1}_t)}.$$

\begin{appendix}

\section{Separating class on spaces of probability measures}\label{append:separe}

The goal of this section is to prove Proposition~\ref{polysepare}. It means that the set~$\poly^\infty$ is a separating class for~$\P_p(\r)$ (for any~$p\geq 1$). Before proving this statement, let us introduce some definitions and useful results about separating classes. This section relies strongly on the content of Section~3.4 of \cite{ethier_markov_2005}.

In all this section, $E$ denotes some Polish space, and $\P(E)$ the space of probability measures on~$E$ endowed with Prohorov metric (i.e. the topology of the weak convergence).
\begin{defi}
Some set~$\A$ of measurable functions~$f:E\rightarrow\r$ is said to separate~$E$ if, for all~$x,y\in E$,
$$\ll(\forall f\in\A, f(x)=f(y)\rr)\Longrightarrow x=y.$$
A set~$\A$ of measurable functions~$f:E\rightarrow\r$ is called a separating class for~$E$ if, for all~$m,\mu\in\P(E)$,
$$\ll(\forall f\in\A, \int_E f(x)dm(x) = \int_E f(x)d\mu(x) \rr)\Longrightarrow m=\mu.$$
\end{defi} 

In other words, $\A$ is a separating class for~$E$ means that
$$\ll\{m\in\P(E)\longmapsto \int_E f(x)dm(x)~:~f\in\A\rr\}$$
separates~$\P(E).$ The interesting point of a separating class for~$E$ is that it allows to identify the laws on~$E$ in the following sense (it is a mere rephrasing of the definition of a separating class).
\begin{lem}
If $\A$ is a separating class for~$E$, and $X,Y$ are two $E$-valued random variables such that: for all~$f\in \A$,
$$\esp{f(X)} = \esp{f(Y)}.$$
Then $X$ and $Y$ have the same law.
\end{lem}

The lemma below should be classical, but since we have not found a proof in the literature, we provide one for self-completeness.
\begin{lem}\label{ccsepare}
The set~$C^\infty_c(\r)$ is a separating class for~$\r$.
\end{lem}

\begin{proof}
Let $m,\mu\in\P(\r)$ such that, for all~$g\in C_c^\infty(\r)$,
\begin{equation}\label{intmintmu}
\int_\r g(x)dm(x) = \int_\r g(x)d\mu(x).
\end{equation}

Let $a<b$. Let $(g_n)_n$ be a $C^\infty_c(\r)-$valued sequence such that 
$$\ll\{\begin{array}{l}
(g_n)_n\textrm{ converges pointwise to }\un_{]a,b[},\\
\forall (x,n)\in\r\times\n,~~ 0\leq g_n(x)\leq 1.
\end{array}\rr.$$
Then, rewriting~\eqref{intmintmu} with~$g_n$ instead of~$g$ (for all~$n\in\n^*$), and using the Dominated Convergence Theorem,
$$m(]a,b[) = \mu(]a,b[).$$
The previous equality being true for any~$a<b$, this proves that~$m=\mu$.
\end{proof}

The proof that the polynomials form a separating class for~$\P_p(\r)$ relies on the following criterion.
\begin{thm}[Theorem~3.4.5.(a) of \cite{ethier_markov_2005}]\label{algebresepare}
Any algebra that separates~$E$ is a separating class for~$E$.
\end{thm} 

Besides, since $\P_p(\r)$ is Polish (see Theorem~6.18 of \cite{villani_optimal_2009}), Theorem~\ref{algebresepare} above can be used with $E=\P_p(\r)$. Now we can prove that the polynomials form a separating class of~$\P_p(\r)$.
\begin{proof}[Proof of Proposition~\ref{polysepare}]
By Lemma~\ref{ccsepare}, the set
$$\ll\{m\in\P(\r)\longmapsto \int_\r h(x)dm(x)~:~h\in C^\infty_c(\r)\rr\}$$
separates~$\P(\r)$, hence it also separates~$\P_p(\r)$. The set $\poly^\infty$ being an algebra containing the set above, it is a separating class by Theorem~\ref{algebresepare}.
%
\end{proof}

\section{Proof of Lemma~\ref{reguflot}}\label{proofreguflot}

The proof of Lemma~\ref{reguflot} uses the following classical lemma whose proof is omitted.
\begin{lem}\label{analysederivee}
Let $(f_n)_n$ be a sequence of~$C^1(\r)$ such that:
\begin{itemize}
\item $(f_n)_n$ converges point-wisely to some function~$f$,
\item $(f'_n)_n$ converges uniformly on every compact set to some function~$g$.
\end{itemize}
Then, $f\in C^1(\r)$ and~$f'=g$.
\end{lem}

Let us recall that, by definition, for all~$x\in\r,t\geq 0$,
$$\bar X_t^{(x)} = x + \int_0^t b\ll(\bar\mu^{(x)}_s,\bar X^{(x)}_s\rr)ds + \int_0^t \sigma\ll(\bar\mu^{(x)}_s,\bar X^{(x)}_s\rr)dB_s+\int_0^t \varsigma\ll(\bar\mu^{(x)}_s,\bar X^{(x)}_s\rr)dW_s,$$
where $\bar\mu^{(x)}_s = \mathcal{L}(\bar X^{(x)}_s|\W_s).$

In order to prove Lemma~\ref{reguflot}, we use the Banach-Picard iteration scheme related to the above equation. Namely, for all~$x\in\r,t\geq 0,n\in\n,$
\begin{align*}
\bar X^{(x),[0]}_t=& x,\\
\bar\mu^{(x),[0]}_t = & \delta_x,\\
\bar X^{(x),[n+1]}_t=& x + \int_0^t b\ll(\bar\mu^{(x),[n]}_s,\bar X^{(x),[n]}_s\rr)ds + \int_0^t \sigma\ll(\bar\mu^{(x),[n]}_s,\bar X^{(x),[n]}_s\rr)dB_s\\
&+\int_0^t \varsigma\ll(\bar\mu^{(x),[n]}_s,\bar X^{(x),[n]}_s\rr)dW_s,\\
\bar\mu^{(x),[n+1]}_t=& \mathcal{L}\ll(\bar X^{(x),[n+1]}_t|\W_t\rr).
\end{align*}

{\it Step~1.} In this first step, we prove the almost sure convergence of $\bar X^{(x),[n]}_t$ to $\bar X^{(x)}_t$ as $n$ goes to infinity, for any $t$ belonging to some sufficiently small interval, and locally uniformly w.r.t.~$x$.

For~$n\in\n,p\in\n^*,t\geq 0$ and $M>0$,
$$u^{[n],0}_t(M,p) = \esp{\underset{s\leq t,|x|\leq M}{\sup}\ll|\bar X^{(x),[n+1]}_s - \bar X^{(x),[n]}_s\rr|^p}.$$

With classical computation (using Burkholder-Davis-Gundy's inequality and the assumption that $b,\sigma,\varsigma$ are Lipschitz continuous), we have that, for all~$T>0$, $p\in\n^*$, $n\in\n$,
$$u^{[n+1],0}_T(M,p) \leq C_p\ll(T^p + T^{p/2}\rr)u^{[n],0}_T(M,p),$$
with $C_p>0$ independent of~$T$, $M$ and~$n$. In particular, let us fix some small enough~$T_p>0$ such that, for all~$n\in\n,p\in\n,M>0$,
$$u^{[n+1],0}_{T_p}(M,p) \leq \frac12 u^{[n],0}_{T_p}(M,p),$$
whence
\begin{equation}\label{controlun0}
u^{[n],0}_{T_p}(M,p) \leq C_p\frac{1}{2^{n}}\ll(1 + M^p\rr).
\end{equation}
and
$$\esp{\underset{s\leq T_p,|x|\leq M}{\sup}\ll|\bar X^{(x),[n+1]}_s - \bar X^{(x),[n]}_s\rr|}\leq u^{[n],0}_{T_p}(M,p)^{1/p} \leq C_p\frac{1}{2^{n/p}}\ll(1 + M\rr).$$

As a consequence, for all~$M>0,p\in\n^*$,
$$\esp{\underset{t\leq T_p,|x|\leq M}{\sup}\ll|\sum_{n=0}^{+\infty}\ll(\bar X^{(x),[n+1]}_t - \bar X^{(x),[n]}_t\rr)\rr|}<\infty,$$
and, almost surely,
$$\underset{t\leq T_p,|x|\leq M}{\sup}\ll|\sum_{n=0}^{+\infty}\ll(\bar X^{(x),[n+1]}_t - \bar X^{(x),[n]}_t\rr)\rr|<\infty.$$

This proves that, almost surely, for all~$t\in[0,T_p],$ $\bar X^{(x),[n]}_t$ converges as $n$ goes to infinity to $\bar X^{(x)}_t$ (recalling that $\bar X^{(x)}$ is solution to~\eqref{barXsanssaut}) locally uniformly w.r.t.~$x$.

In addition, with similar technics and Fatou's lemma, we can prove that, for all~$p\in\n^*,$ $M>0$, $T>0$,
\begin{equation}\label{momentp0}
\esp{\underset{t\leq T,|x|\leq M}{\sup}\ll|\bar X^{(x)}_t\rr|^p}<\infty.
\end{equation}

{\it Step~2.} Now we study the regularity of the functions
$$x\longmapsto \bar X^{(x),[n]}_t.$$

Let us prove by induction on~$n\in\n$ that there exists~$T>0$ (independent of~$n$) such that, for all~$t\in[0,T]$ the function $x\mapsto \bar X^{(x),[n]}_t$~$C^1$ and that, for all~$M>0$,
$$\esp{\underset{t\leq T,|x|\leq M}{\sup}\ll|\partial_x\bar X^{(x),[n]}_t\rr|}<\infty~\textrm{ and }~\underset{x\in\r}{\sup}~\esp{\underset{t\leq T}{\sup}\ll|\partial_x\bar X^{(x),[n]}_t\rr|}<\infty.$$

For~$n=0$,
$$\partial_x \bar X^{(x),[0]}_t = 1.$$

Then fix some~$n\in\n$ and assume that the induction hypothesis holds true for this~$n$. Then, by Lemma~\ref{chainruleloi} (and Remark~\ref{remchainruleloi}), and assuming that~$b,\sigma,\varsigma$ admits bounded second order mixed derivatives (what implies that the first order mixed derivatives are continuous w.r.t. all their variables, by Remark~\ref{remmix}),
\begin{align}
\partial_x \bar X^{(x),[n+1]}_t = & 1 + \int_0^t \ll(\partial_x \bar X^{(x),[n]}_s\rr)\partial_{(0,1)} b\ll(\bar\mu^{(x),[n]}_s,\bar X^{(x),[n]}_s\rr)ds\label{partial1barX}\\
&+ \int_0^t\tilde{\mathbb{E}}\ll[\ll.\ll(\partial_x \tilde X^{(x),[n]}_s\rr)\partial \ll(b_{\bar X^{(x),[n]}_s}\rr)\ll(\bar \mu^{(x),[n]}_s,\tilde X^{(x),[n]}_s\rr)\rr|\W_s\rr]ds\nonumber\\
&+\int_0^t \ll(\partial_x \bar X^{(x),[n]}_s\rr)\partial_{(0,1)} \sigma\ll(\bar\mu^{(x),[n]}_s,\bar X^{(x),[n]}_s\rr)dB_s\nonumber\\
&+ \int_0^t\tilde{\mathbb{E}}\ll[\ll.\ll(\partial_x \tilde X^{(x),[n]}_s\rr)\partial \ll(\sigma_{\bar X^{(x),[n]}_s}\rr)\ll(\bar \mu^{(x),[n]}_s,\tilde X^{(x),[n]}_s\rr)\rr|\W_s\rr]dB_s\nonumber\\
&+\int_0^t \ll(\partial_x \bar X^{(x),[n]}_s\rr)\partial_{(0,1)} \varsigma\ll(\bar\mu^{(x),[n]}_s,\bar X^{(x),[n]}_s\rr)dW_s\nonumber\\
&+ \int_0^t\tilde{\mathbb{E}}\ll[\ll.\ll(\partial_x \tilde X^{(x),[n]}_s\rr)\partial \ll(\varsigma_{\bar X^{(x),[n]}_s}\rr)\ll(\bar \mu^{(x),[n]}_s,\tilde X^{(x),[n]}_s\rr)\rr|\W_s\rr]dW_s,\nonumber
\end{align}
where $\tilde X$ is defined in the same way as $\bar X$ w.r.t. the same Brownian motion~$W$, but w.r.t. a Brownian motion~$\tilde B$ independent of~$(B,W)$, and $\tilde\E$ is the expectation w.r.t. the law of~$\tilde B$. One can note that, the expression above of $\partial_x \bar X^{(x),[n+1]}_t$ is closely related to the ones of Proposition~3.1 of \cite{crisan_smoothing_2018} (we actually use their notation for $\tilde\E$ and~$\tilde X$).

Let us remark an important point: the functions appearing as coefficients in the expression above are exactly the first order mixed derivatives of the functions~$b,\sigma,\varsigma$.

And, with similar computation as in {\it Step~1}, we can prove that, for all~$p\in\n^*$, $M>0$,
\begin{equation}\label{momentp1}
\underset{n\in\n}{\sup}~\esp{\underset{t\leq T_p,|x|\leq M}{\sup}\ll|\partial_x\bar X^{(x),[n]}_t\rr|^p}<\infty\textrm{ and }\underset{n\in\n,x\in\r}{\sup}\esp{\underset{t\leq T_p}{\sup}\ll|\partial_x\bar X^{(x),[n]}_t\rr|}<\infty,
\end{equation}
for some $T_p>0$ independent of~$n,M$.

Let
$$u^{[n],1}_t(M,p) = \esp{\underset{s\leq t,|x|\leq M}{\sup}\ll|\partial_x\bar X^{(x),[n+1]}_s - \partial_x\bar X^{(x),[n]}_s\rr|^p}.$$

To control the above quantity in a similar way as in {\it Step~1}, it is required to control two kind of terms. In order to simplify the reading we only handle the drift terms (the Brownian terms can be treated in the exact same way after using Burkholder-Davis-Gundy's inequality). The first term is of the following form: for $|x|\leq M$,
\begin{align*}
&\ll|\ll(\partial_x \bar X^{(x),[n+1]}_s\rr)\partial_{(0,1)} b\ll(\bar\mu^{(x),[n+1]}_s,\bar X^{(x),[n+1]}_s\rr)-\ll(\partial_x \bar X^{(x),[n]}_s\rr)\partial_{(0,1)} b\ll(\bar\mu^{(x),[n]}_s,\bar X^{(x),[n]}_s\rr)\rr|\\
&~~\leq \ll(\underset{m\in\P_1(\r),|y|\leq M}{\sup}|\partial_{(0,1)} b(m,y)|\rr)\ll|\partial_x \bar X^{(x),[n+1]}_s-\partial_x \bar X^{(x),[n]}_s\rr|\\
&~~~~~~+C\ll|\partial_x X^{(x),[n]}_s\rr|\ll(\ll|\bar X^{(x),[n+1]}_s - \bar X^{(x),[n]}_s\rr| + D_{KR}\ll(\bar \mu^{(x),[n+1]}_s,\bar \mu^{(x),[n]}_s\rr)\rr).
\end{align*}

In particular, thanks to~\eqref{controlun0},~\eqref{momentp1} and Cauchy-Schwarz' inequality, for any~$p\in\n^*$, for $t\leq T_p$ (where $T_p$ has possibly been reduced compared to the previous one),
\begin{align}
&\mathbb{E}\ll[\ll(\int_0^t\ll|\ll(\partial_x \bar X^{(x),[n+1]}_s\rr)\partial_{(0,1)} b\ll(\bar\mu^{(x),[n+1]}_s,\bar X^{(x),[n+1]}_s\rr)\rr.\rr.\rr.\nonumber\\
&\hspace*{5cm}\ll.\ll.\ll.-\ll(\partial_x \bar X^{(x),[n]}_s\rr)\partial_{(0,1)} b\ll(\bar\mu^{(x),[n]}_s,\bar X^{(x),[n]}_s\rr)\rr|ds\rr)^p\rr]\nonumber\\
&~~\leq C_{M,p} T^p u^{[n],1}_t(M,p) + C_{M,p} T^p\frac1{2^{n/2}}\label{firsttermb1}.
\end{align}

The second term involves the following term
$$\ll(\partial_x \tilde X^{(x),[n]}_s\rr)\partial \ll(b_{\bar X^{(x),[n]}_s}\rr)\ll(\bar \mu^{(x),[n]}_s,\tilde X^{(x),[n]}_s\rr),$$
and can be handled exactly as~\eqref{firsttermb1}, since the function $(m,x,y)\mapsto \partial_y \delta (b_x)(m,y)$ is assumed to be Lipschitz continuous and bounded (recalling that this function belongs to the set of first order mixed derivatives of~$b$). Whence, for $T\leq T_p$,
$$u^{[n+1],1}_T(M,p)\leq C_{M,p}\ll(T^p + T^{p/2}\rr)\ll(u^{[n],1}_T(M,p) + \frac1{2^{n/2}}\rr).$$

And, possibly by reducing~$T_p$,
$$u^{[n+1],1}_T(M,p)\leq \frac12u^{[n],1}_T(M,p) + \frac1{2^{1+n/2}}.$$

Consequently, for $T\leq T_p$,
\begin{equation}\label{controlun1}
u^{[n],1}_T(M,p) \leq C_{M,p} 2^{-n/2},
\end{equation}
and
$$\esp{\underset{s\leq T_p,|x|\leq M}{\sup}\ll|\partial_x\bar X^{(x),[n+1]}_s - \partial_x\bar X^{(x),[n]}_s\rr|}\leq u^{[n],1}_T(M,p)^{1/p}\leq C_{M,p} {2^{-n/(2p)}}.$$

The quantities above being the terms of a convergent series, we can conclude as in {\it Step~1} that, almost surely, for all~$t\in[0,T_p],$ $\partial_x \bar X^{(x),[n]}_t$ converges as $n$ goes to infinity to some function locally uniformly w.r.t.~$x$. Then, by Lemma~\ref{analysederivee}, we know that, almost surely, for all $t\in[0,T_p],$ the function $x\mapsto \bar X^{(x)}_t$ is $C^1$, and that its derivative is the limit of $\partial_x \bar X^{(x),[n]}_t$. In particular, by Fatou's lemma and~\eqref{momentp1},
\begin{equation}\label{momentpderiv2}
\esp{\underset{t\leq T,|x|\leq M}{\sup}\ll|\partial_x\bar X^{(x)}_t\rr|^p}<\infty~\textrm{ and }~\underset{x\in\r}{\sup}~\esp{\underset{t\leq T}{\sup}\ll|\partial_x\bar X^{(x)}_t\rr|^p}<\infty.
\end{equation}

{\it Step~3.} The proof that the function $x\mapsto \bar X^{(x)}_t$ is $C^2$ uses the same arguments as the ones used in {\it Step~2}. In order to make it clear, we just write the terms of the dynamics of~$\partial^2_{xx} \bar X^{(x),[n+1]}_t$ coming from the two first lines of~\eqref{partial1barX}.
\begin{align*}
&\partial^2_{xx}\bar X^{(x),[n+1]}_t= \int_0^t \ll(\partial^2_{xx} \bar X^{(x),[n]}_s\rr)\partial_{(0,1)} b\ll(\bar\mu^{(x),[n]}_s,\bar X^{(x),[n]}_s\rr)ds\\
&+  \int_0^t \ll(\partial_{x} \bar X^{(x),[n]}_s\rr)^2\partial_{(0,2)} b\ll(\bar\mu^{(x),[n]}_s,\bar X^{(x),[n]}_s\rr)ds\\
&+\int_0^t \ll(\partial_{x} \bar X^{(x),[n]}_s\rr)\tilde{\E}\ll[\ll.\ll(\partial_{x} \tilde{X}^{(x),[n]}_s\rr)\partial \ll(\ll(\partial_{(0,1)} b\rr)_{\bar X^{(x),[n]}_s}\rr)\ll(\bar \mu^{(x),[n]}_s, \tilde{X}^{(x),[n]}\rr)\rr|\W_s\rr]ds\\
&+\int_0^t\tilde{\mathbb{E}}\ll[\ll.\ll(\partial^2_{xx} \tilde X^{(x),[n]}_s\rr)\partial \ll(b_{\bar X^{(x),[n]}_s}\rr)\ll(\bar \mu^{(x),[n]}_s,\tilde X^{(x),[n]}_s\rr)\rr|\W_s\rr]ds\\
&+\int_0^t\ll(\partial_{x} \bar X^{(x),[n]}_s\rr)\tilde{\mathbb{E}}\ll[\ll.\ll(\partial_x \tilde X^{(x),[n]}_s\rr)\ll(\partial_{\bar x}\partial \ll(b_{\bar x}\rr)\ll(\bar \mu^{(x),[n]}_s,\tilde X^{(x),[n]}_s\rr)\rr)_{\ll|\bar x = \bar X^{(x),[n]}_s\rr.}\rr|\W_s\rr]ds\\
&+\int_0^t \check{\mathbb{E}}\tilde{\mathbb{E}}\ll[\ll.\ll(\partial_x \tilde X^{(x),[n]}_s\rr)\ll(\partial_x \check X^{(x),[n]}_s\rr)\partial^2 \ll( b_{\bar X^{(x),[n]}_s}\rr)\ll(\bar \mu^{(x),[n]}_s,\tilde  X^{(x),[n]}_s,\check  X^{(x),[n]}_s\rr)\rr|\W_s\rr]ds\\
&+...
\end{align*}
where the omitted terms (hidden in the ellipsis) are the same as the ones written replacing the function~$b$ respectively by~$\sigma$ and~$\varsigma$, and~$ds$ by~$dB_s$ and~$dW_s$. As previously, $\tilde X$ and $\check X$ are defined as $\bar X$ w.r.t. the same Brownian motion~$W$, and w.r.t. respective Brownian motions~$\tilde B$ and~$\check B$ such that $W,\hat B$ and~$\check B$ are independent. The expectations $\tilde \E$ and~$\check \E$ are the expectation w.r.t. the laws of~$\tilde B$ and~$\check B$.

Once again, it can be noted that the functions appearing in the expression of $\partial^2_{xx}\bar X^{(x),[n+1]}_t$ are the mixed-derivatives of $b,\sigma,\varsigma$ up to order two. Hence the same reasoning as in {\it Step~2} allows to conclude. Finally, the proof that $x\mapsto \bar X^{(x)}_t$ is $C^j$ (for any $j\leq k-1$) is the same as previously remarking that, calculating the expression of $\partial^j_{x^j}\bar X^{(x),[n+1]}_t$ makes appear the mixed-derivatives of $b,\sigma,\varsigma$ up to order $j$. Since it is assumed that all the mixed-derivatives up to order $k$ are bounded, we know by, Remark~\ref{remchainruleloi}, that the mixed-derivatives up to order $k-1$ are Lipschitz continuous. So the end of the proof follows the same arguments as {\it Step~2}, and similar controls as~\eqref{momentpderiv2} can be proved for the $j$-th order derivatives, for any~$j\leq k-1$.

\section{Proofs of some technical results}\label{appendtechnic}

\begin{proof}[Proof of Lemma~\ref{lem:unique}]
Let us consider $m_0\in\P_1(\r)$ and $h_1,h_2:\r\rightarrow\r$ such that, for all~$m\in\P_1(\r)$,
$$\int_\r h_1(x)d(m-m_0)(x)=\int_\r h_2(x)d(m-m_0)(x) + \eps_{m_0}(m),$$
where $\eps_{m_0}(m)/D_{KR}(m,m_0)$ vanishes as~$m$ converges to~$m_0$.

Let us define
$$\Phi : m\in\P_1(\r)\longmapsto \int_\r \ll(h_1(x) - h_2(x)\rr)d(m-m_0)(x).$$

So, by hypothesis, $\Phi(m)/D_{KR}(m,m_0)$ vanishes as $m$ goes to~$m_0$.

{\it Step~1.} In a first time, we prove that $\Phi$ is the zero function. Let us fix $m\in\P_1(\r)$ and introduce, for any~$n\in\n^*$,
$$\mu_n = \frac{n-1}{n}m_0 + \frac1n m \in\P_1(\r).$$

We have that, for all~$n\in\n^*,$
$$D_{KR}(m_0,\mu_n)\leq \frac1n D_{KR}(m_0,m),$$
hence, $\mu_n$ converges to~$m_0$ as $n$ goes to infinity.

On the other hand,
$$\Phi(m) = n\Phi(\mu_n) = n\frac{\Phi(\mu_n)}{D_{KR}(m_0,\mu_n)}D_{KR}(m_0,\mu_n)\leq \frac{\Phi(\mu_n)}{D_{KR}(m_0,\mu_n)} D_{KR}(m_0,m)\underset{n\rightarrow\infty}{\longrightarrow}0.$$

This proves that, for all~$m\in\P_1(\r)$, $\Phi(m)= 0$.

{\it Step~2.} To conclude the proof, it is then sufficient to notice that, thanks to {\it Step~1}, for any~$y\in\r$, $\Phi(\delta_y)=0$. Whence, for any~$y\in\r$,
$$h_1(y) - h_2(y) = \int_\r \ll(h_1(x) - h_2(x)\rr)dm_0(x).$$

In particular, the function $h_1-h_2$ is constant, and the lemma is proved.
\end{proof}

\begin{proof}[Proof of Corollary~\ref{corderivint1}]
Let $m,m_0\in\P_1(\r)$ be fixed in all the proof.
\begin{align*}
F(m) - F(m_0)=& \int_\r H(m,x)d(m-m_0)(x) + \int_\r \ll(H(m,x) - H(m_0,x)\rr)dm_0(x)\\
=&\int_\r H(m_0,x)d(m-m_0)(x) + \int_\r \ll(H(m,x) - H(m_0,x)\rr)dm_0(x)\\
&+\int_\r \ll(H(m,x) - H(m_0,x)\rr)d(m-m_0)(x).
\end{align*}

Consequently,
\begin{align}
&\ll|F(m) - F(m_0) - \int_\r \ll(H(m_0,x) -  \int_\r \delta H_y(m_0,x)dm_0(y)\rr)d(m-m_0)(x)\rr|\nonumber\\
&~~~~\leq \int_\r \ll|H_x(m) - H_x(m_0) - \int_\r \delta H_x(m_0,y)d(m-m_0)(y)\rr|dm_0(x)\label{terme1cor}\\
&~~~~~~~~+ \ll|\int_\r \ll(H(m,x) - H(m_0,x)\rr)d(m-m_0)(x)\rr|.\label{terme2cor}
\end{align}

Then, by the hypothesis~$(i)$ of the corollary and Theorem~\ref{taylorlagrange}, the integrand of~\eqref{terme1cor} is bounded by
$$C(1+|x|)D_{KR}(m,m_0)^2$$
for any~$x\in\r$ (with $C>0$ independent of~$x,m,m_0$), whence the quantity at~\eqref{terme1cor} is bounded by
$$C\ll(1 + \int_\r |x|dm_0(x)\rr)D_{KR}(m,m_0)^2$$
which is negligible compared to~$D_{KR}(m,m_0)$ when $m$ converges to~$m_0$ (for a fixed~$m_0$).

In addition, the term~\eqref{terme2cor} is non-greater than
$$D_{KR}(m,m_0) \cdot \underset{x\in\r}{\sup}~\ll|\partial_x \ll(H(m,x) - H(m_0,x)\rr)\rr|.$$

And, by the Mean Value Theorem (i.e. Proposition~\ref{mvthm}) and the hypothesis~$(ii)$ of the corollary, for all~$x\in\r$,
\begin{align*}
\ll|\partial_x \ll(H(m,x) - H(m_0,x)\rr)\rr| =& \ll|\partial_x H_x(m) - \partial_x H_x(m_0)\rr|\\
\leq& D_{KR}(m,m_0)\cdot \underset{\substack{m\in\P_1(\r)\\x,y\in\r}}{\sup}\ll|\partial_y \delta \ll(\partial_x H_x\rr)(m,y)\rr|,
\end{align*}
which entails that the term at~\eqref{terme2cor} is bounded by
$$C D_{KR}(m,m_0)^2$$
for some constant $C>0$ independent of $m,m_0$.

So we have proved that $F$ is differentiable on $\P_1(\r)$, and that the following function is one version of its derivative:
$$(m,x)\in\P_1(\r)\times\r\longmapsto H(m,x) + \int_\r \delta H_y(m,x)dm(y).$$

It is then sufficient to subtract $F(m)$ to the quantity above to obtain the canonical derivative of~$F$.
\end{proof}

The proof of Corollary~\ref{corderivint2} requires the following lemma.

\begin{lem}\label{dkrotimes}
Let $d\in\n^*$ and $m_1,...,m_d,\mu_1,...,\mu_d\in\P_1(\r)$. Then,
$$D_{KR}\ll(\bigotimes_{k=1}^d m_k,\bigotimes_{k=1}^d \mu_k\rr)\leq \sum_{k=1}^d D_{KR}(m_k,\mu_k).$$
\end{lem}

\begin{proof}
For each~$1\leq k\leq d$, there exists a probability space~$(\Omega_k,\F_k,\mathbb{P}_k)$ and two random variables on this space~$X_k,Y_k$ of respective laws $m_k,\mu_k$ such that
$$D_{KR}(m_k,\mu_k) = \mathbb{E}_k\ll[|X_k - Y_k|\rr],$$
with $\mathbb{E}_k$ the expectation w.r.t.~$\mathbb{P}_k$ (the existence of these random variables is guaranteed by Theorem~4.1 of \cite{villani_optimal_2009}). Let us assume that all the probability spaces~$(\Omega_k,\F_k,\mathbb{P}_k)$ are disjoint, and consider the product probability space of these spaces, denoted by $(\Omega,\F,\mathbb{P})$. In particular, the variables $X_k$ (resp. $Y_k$) $(1\leq k\leq d)$ are independent, whence
$$\mathcal{L}\ll(X_1,...,X_d\rr) = \bigotimes_{k=1}^d m_k~\textrm{ and }~\mathcal{L}\ll(Y_1,...,Y_d\rr) = \bigotimes_{k=1}^d \mu_k.$$

Consequently,
$$D_{KR}\ll(\bigotimes_{k=1}^d m_k,\bigotimes_{k=1}^d \mu_k\rr)\leq \esp{\sum_{k=1}^d \ll|X_k - Y_k\rr|} =  \sum_{k=1}^d \mathbb{E}_k\ll[\ll|X_k - Y_k\rr|\rr] = \sum_{k=1}^d D_{KR}(m_k,\mu_k),$$
which proves the result.
\end{proof}

\begin{proof}[Proof of Corollary~\ref{corderivint2}]
Let us fix some $m,m_0\in\P_1(\r)$ in the proof.

\begin{align*}
F(m) - F(m_0)=& \int_{\r^2} H(m,x) d\ll(m^{\otimes 2} - m_0^{\otimes 2}\rr)(x)\\
& + \int_{\r^2}\ll(H(m,x) - H(m_0,x)\rr)dm_0^{\otimes 2}(x)\\
=& \int_{\r^2} H(m_0,x_1,x_2)dm(x_1)d(m-m_0)(x_2)\\
&+ \int_{\r^2} H(m_0,x_1,x_2)d(m-m_0)(x_1)dm_0(x_2)\\
&+\int_{\r^2}\ll(H(m,x) - H(m_0,x)\rr)dm_0^{\otimes 2}(x) \\
&+  \int_{\r^2}\ll(H(m,x) - H(m_0,x)\rr)d\ll(m^{\otimes 2} - m_0^{\otimes 2}\rr)(x).
\end{align*}

Then,
\begin{align}
&\biggl|F(m) - F(m_0)\biggr. - \int_\r \int_{\r^2}\delta (H_x)(m_0,y)dm_0^{\otimes 2}(x) d(m-m_0)(y)\nonumber\\
&~~~~\ll.-\int_\r\ll(\int_\r H(m_0,x_1,y)dm_0(x_1) + \int_\r H(m_0,y,x_2)dm_0(x_2)\rr)d(m-m_0)(y)\nonumber\rr|\\
&~~\leq \ll|\int_{\r^2}H(m_0,x_1,x_2)d(m-m_0)^{\otimes 2}(x_1,x_2)\rr|\label{terme17}\\
&~~~~~~ + \ll| \int_{\r^2}\ll(H(m,x) - H(m_0,x) - \int_\r \delta (H_x)(m_0,y)d(m-m_0)(y)\rr)dm_0^{\otimes 2}(x)\rr|\label{terme18}\\
&~~~~~~ + \ll|\int_{\r^2}\ll(H(m,x) - H(m_0,x)\rr)d\ll(m^{\otimes 2} - m_0^{\otimes 2}\rr)(x)\rr|.\label{terme19}
\end{align}

By Lemma~\ref{DKRn} and thanks to the hypothesis~$(iii)$ of the corollary, the term at~\eqref{terme17} is bounded by
$$D_{KR}(m,m_0)^2 \underset{\mu\in\P_1(\r),x_1,x_2\in\r}{\sup}\ll|\partial^2_{x_1x_2} H(\mu,x)\rr|.$$

According to Taylor-Lagrange's inequality (i.e. Theorem~\ref{taylorlagrange}) and using Hypothesis~$(i)$, we know that the term at~\eqref{terme18} is non-greater than
$$C \cdot D_{KR}(m,m_0)^2 \ll(1 + \ll(\int_\r |x|dm_0(x)\rr)^2\rr),$$
for some~$C>0$ independent of~$(m,m_0)$.

For the last term (i.e.~\eqref{terme19}), by definition of Kantorovich-Rubinstein's metric, is bounded by
$$D_{KR}\ll(m^{\otimes 2},m_0^{\otimes 2}\rr) \cdot \underset{x_1,x_2\in\r}{\sup}\ll|\partial_{x_1}H_x(m) - \partial_{x_1} H_x(m_0)\rr|+\ll|\partial_{x_2}H_x(m) - \partial_{x_2} H_x(m_0)\rr|.$$

As in the end of the proof of the previous corollary, we have
$$\underset{x_1,x_2\in\r}{\sup}\ll|\partial_{x_1}H_x(m) - \partial_{x_1} H_x(m_0)\rr|+\ll|\partial_{x_2}H_x(m) - \partial_{x_2} H_x(m_0)\rr|\leq C\cdot D_{KR}(m,m_0).$$

On the other hand, by Lemma~\ref{dkrotimes}, we have that
$$D_{KR}\ll(m^{\otimes 2},m_0^{\otimes 2}\rr) \leq 2 D_{KR}(m,m_0).$$

So the term at~\eqref{terme19} is bounded by $C\cdot D_{KR}(m,m_0)^2.$ This proves finally that the function $F$ is differentiable on~$\P_1(\r)$ and that the function
\begin{multline*}
(m,y)\in\P_1(\r)\times\r\longmapsto \int_{\r^2}\delta (H_x)(m,y)dm^{\otimes 2}(x) + \int_\r H(m,x_1,y)dm(x_1)\\ + \int_\r H(m,y,x_2)dm(x_2)
\end{multline*}
is one version of the derivative of~$F$. It is then sufficient to subtract $2F(m)$ to the quantity above to obtain the canonical derivative of~$F$.
\end{proof}

\section{Some technical lemmas about Ito's integrals}\label{appendito}

Let $(X_t)_{t\geq 0}$ be some c\`adl\`ag $\r$-valued process that is locally $L^2$: for all~$t\geq 0,$
\begin{equation}\label{locl2X}
\int_0^t \esp{X_s^2}ds<\infty.
\end{equation}

Let $W,B$ be two standard Brownian motions of dimension one, and $(\F_t)_t$ be a filtration such that $\F$,~$B$ and~$W$ are independent. Let us denote $(\W_t)_t$ (resp. $(\B_t)_t$) the filtration of~$W$ (resp. $B$), meaning
$$\W_t = \sigma\ll(W_s~:~s\leq t\rr)\textrm{ and }\B_t = \sigma\ll(B_s~:~s\leq t\rr),$$
and define $(\G_t)_t$ the union (in the filtration sense) of~$\F$,~$\W$ and~$\B$: for all~$t\geq 0,$
$$\G_t = \F_t\vee \W_t\vee \B_t.$$

\begin{lem}\label{lemintw}
Assume that $X$ is $\G$-adapted. Then, for all~$t\geq 0$,
$$\espc{\int_0^t X_s dW_s}{\W_t} = \int_0^t \espc{X_s}{\W_s}dW_s.$$
\end{lem}

\begin{proof}
Let us fix~$t\geq 0$. By definition of Ito's integral,
$$\int_0^t X_sdW_s = \underset{n\rightarrow\infty}{\lim}\sum_{k=0}^{n-1} X_{s_k}\ll(W_{s_{k+1}} - W_{s_k}\rr)$$
in probability, where, for any~$0\leq k\leq n$, $s_k = t\cdot k/n$.

Then, thanks to~\eqref{locl2X}, the (conditional) Vitali's convergence theorem implies that, almost surely,
\begin{align*}
\espc{\int_0^t X_s dW_s}{\W_t} =&  \underset{n\rightarrow\infty}{\lim}\sum_{k=0}^{n-1} \espc{X_{s_k}\ll(W_{s_{k+1}} - W_{s_k}\rr)}{\W_t}\\
=& \underset{n\rightarrow\infty}{\lim}\sum_{k=0}^{n-1} \espc{X_{s_k}}{\W_t}\ll(W_{s_{k+1}} - W_{s_k}\rr)\\
=& \underset{n\rightarrow\infty}{\lim}\sum_{k=0}^{n-1} \espc{X_{s_k}}{\W_{s_k}}\ll(W_{s_{k+1}} - W_{s_k}\rr)\\
=& \int_0^t \espc{X_s}{\W_s}dW_s,
\end{align*}
where the before last equality above comes from the fact that~$X$ is $\G$-adapted and that $\W_t$ can be written as the union (in the filtration sense) of $\W_{s_k}$ and $\sigma(W_r - W_{s_k}:s_k<r\leq t)$ which are independent (whence, $X_{s_k}$ is also necessarily independent of $\sigma(W_r - W_{s_k}:s_k<r\leq t)$).
\end{proof}

\begin{lem}\label{lemintb}
Assume that $X$ is $\G$-adapted. Then, for all~$t\geq 0$,
$$\espc{\int_0^t X_s dW_s}{\B_t} = 0.$$
\end{lem}

\begin{proof}
With the same reasoning (and using the same notation) as in the beginning of the proof of Lemma~\ref{lemintw}, for all~$t\geq 0$,
\begin{equation}\label{eqintbw}
\espc{\int_0^t X_s dB_s}{\W_t} =\underset{n\rightarrow\infty}{\lim}\sum_{k=0}^{n-1} \espc{X_{s_k}\ll(B_{s_{k+1}} - B_{s_k}\rr)}{\W_t}.
\end{equation}

Let us recall that $X$ is $\G$-adapted and that $B,W$ are independent. Let us consider any $\W_t$-measurable random variable~$Z$. Then, both~$Z$ and~$X_{s_k}$ are independent of~$B_{s_{k+1}} - B_{s_k}$. So,
$$\esp{X_{s_k}\ll(B_{s_{k+1}} - B_{s_k}\rr)Z} = \esp{X_{s_k}Z}\esp{B_{s_{k+1}} - B_{s_k}}=0,$$
implying
$$\espc{X_{s_k}\ll(B_{s_{k+1}} - B_{s_k}\rr)}{\W_t}=0.$$

Combining the equation above with~\eqref{eqintbw} proves the result.
\end{proof}

\end{appendix}

\bibliography{biblio}	
\end{document}